\documentclass[10pt]{article}

\usepackage{titlesec}
\usepackage{amssymb,amsmath,amsfonts,amsthm,enumerate,color,mathrsfs,amssymb,extarrows,empheq,cite}
\usepackage[shortlabels]{enumitem}
\usepackage[toc,title,titletoc,header]{appendix}

\usepackage{etoolbox}
\patchcmd{\thebibliography}
  {\settowidth}
  {\setlength{\itemsep}{0pt plus 0.1pt}\settowidth}
  {}{}
\apptocmd{\thebibliography}
  {\small}
  {}{}

\usepackage[usenames,dvipsnames]{xcolor}

\usepackage{esint}

\usepackage[hidelinks]{hyperref}
\hypersetup{colorlinks = true,
linkcolor = MidnightBlue, anchorcolor =red,
citecolor = ForestGreen,
filecolor = red,urlcolor = red,
            pdfauthor=author}
\usepackage{cleveref}

\usepackage{graphicx}
\usepackage{indentfirst}
\usepackage{multicol}
\usepackage{booktabs} 
\usepackage{float} 
\usepackage{subfigure}

\theoremstyle{plain}

\newtheorem{theorem}{Theorem}[section]
\newtheorem{lemma}[theorem]{Lemma}
\newtheorem{proposition}[theorem]{Proposition}
\newtheorem{corollary}[theorem]{Corollary}

\numberwithin{equation}{section}

\theoremstyle{definition}
\newtheorem{definition}[theorem]{Definition}
\newtheorem{remark}[theorem]{Remark}

\newcommand{\norm}[1]{\lVert #1 \rVert}
\newcommand{\normm}[1]{\left\lVert #1 \right\rVert}

\def\q {\quad}

\def \l{\langle}
\def \r{\rangle}

\def\bb{\begin{equation}
  \left\{\ 
   \begin{aligned}}
\def\ee{   \end{aligned}
  \right.
  \end{equation}}

\def\mm{ \left[
 \begin{matrix}}
\def\nn{\end{matrix} \right] } 

\def\p{\partial}
\def \dd{\cdot}
\def \t{\times}

\def\n{\nu}
\def \w {\widetilde}
\def \h{\widehat}

\def\d{\delta}
\def \vp{\varphi}

\def \na{\nabla}
\def \ep{\varepsilon}
\def \lad{\lambda}
\def \Lad{\Lambda}

\def \cc{\mathfrak{C}}

\def \ddiv {{\rm div}}

\def \sdiv {\sdiv_S}

\def \NN{\mathbb{N}}
\def \R{\mathbb{R}}

\def \S{\mathbb{S}}
\def \AA{\mathbb{A}}

\def \P{\mathbb{P}}
\def \xx {\mathbb{X}}

\def \MM{\mathbb{M}}

\def \TT{\mathcal{T}}

\def \M{\mathcal{M}}

\def \L{\mathcal{L}}

\def \ce{\mathcal{CE}}

\def \ff{{\rm F}}

\def \ww{\omega}

\def \jj{\mathcal{J}}

\DeclareMathOperator{\Tan}{Tan}
\DeclareMathOperator{\dom}{dom}
\DeclareMathOperator{\tr}{Tr}

\def \mc {\mathcal}

\def \ran {{\rm Ran}}
\def \ker {{\rm Ker}}

\def \sym {{\rm sym}}

\def \ms{\mathsf}
\def \rd {{\rm d}}

\title{On a general matrix-valued unbalanced optimal transport problem}

\begin{document}
\author{}
\author{
Bowen Li\footnote{Department of Mathematics, Duke University, Durham, NC 27708, USA. (bowen.li200@duke.edu).}
\and Jun Zou\footnote{Department of Mathematics, The Chinese University of Hong Kong, Shatin, N.T., Hong Kong. (zou@math.cuhk.edu.hk).}
}
\date{}
\maketitle

\begin{abstract}
We introduce a general class of transport distances ${\rm WB}_{\Lad}$ over the space of positive semi-definite matrix-valued Radon measures $\mc{M}(\Omega,\S_+^n)$, called the weighted Wasserstein-Bures distance. Such a distance is defined via a generalized Benamou-Brenier
formulation with a weighted action functional and an abstract matricial continuity equation, which leads to a convex optimization problem. Some recently proposed models, including the Kantorovich-Bures distance and the Wasserstein-Fisher-Rao distance, can naturally fit into ours. We give a complete characterization of the minimizer and explore the topological and geometrical properties of the space $(\mc{M}(\Omega,\S_+^n),{\rm WB}_{\Lad})$. In particular, we show that $(\mc{M}(\Omega,\S_+^n),{\rm WB}_{\Lad})$ is a complete geodesic space and exhibits a conic structure.
\end{abstract}

\section{Introduction}

\noindent
\textbf{Classical optimal transport.} Optimal transport (OT) \cite{villani2003topics, villani2008optimal, santambrogio2015optimal} provides a versatile framework for defining metrics and studying geometric structures on probability measures. It has been an active research area over the past decades with fruitful applications in various areas, including functional inequalities \cite{otto2000generalization,lott2009ricci,sturm2006geometry}, gradient flow \cite{otto2001geometry,jordan1998variational}, and more recently, image processing and machine learning \cite{ferradans2014regularized,frogner2015learning,arjovsky2017wasserstein}. The OT problem was first proposed by Monge in 1781 \cite{monge1781memoire}:  given probabilities $\rho_0$ and $\rho_1$, find a measure-preserving transport map $T$ 
minimizing
\begin{align} \label{basicot0}
    \min_{T_{\#}\rho_0 = \rho_1} \int |x - T(x)|^2\, {\rm d} \rho_0(x)\,.
\end{align} 
However, its solution (i.e., the optimal transport map) may not exist. 
This question remained open for a long time until 1942 when Kantorovich introduced a relaxed problem based on the so-called transport plans \cite{kantorovich1942translocation}:
\begin{align}\label{basicot}
    {\rm W}_2^2(\rho_0,\rho_1) := \min \Big\{\int |x-y|^2\, {\rm d}\gamma\,;\  \gamma \ \text{is a  probability with}\ (\pi_\#^x \gamma, \pi_\#^y \gamma) = (\rho_0,\rho_1)\Big\}\,,
\end{align}
where $\pi_\#^x \gamma$ and $\pi_\#^y \gamma$ are the first and second marginals of $\gamma$, respectively. The $2$-Wasserstein distance \eqref{basicot} turns out to exhibit intriguing mathematical properties. Brenier \cite{brenier1991polar} proved that under mild conditions, the optimal transport map $T$ to \eqref{basicot0} exists and is uniquely given by the gradient of a convex function $\vp$. Thanks to the measure-preserving property of the transport map $T = \na \vp$, it is easy to see that $\vp$ satisfies the Monge-Amp\`{e}re equation, which provides a PDE-based approach for solving the OT problem \eqref{basicot0}. One can also show that $({\rm id}, \na \vp)_\# \rho_0$ gives a minimizer to \eqref{basicot}. Equipped with the distance ${\rm W}_2(\dd,\dd)$, the probability measure space becomes a geodesic space, where the geodesic is characterized by McCann’s displacement interpolation
$\rho_t:= ((1-t)I + t \na \vp)_{\#}\rho_0$  \cite{mccann1997convexity}. 
In Benamou and Brenier's seminal work \cite{benamou2000computational}, an equivalent fluid mechanics formulation 
 was proposed for computational purposes: 
\begin{equation} \tag{$\mathcal{P}_{{\rm W}_2}$}    \label{prob:BB_2}
  {\rm W}_2^2(\rho_0,\rho_1) =  \min_{\rho,m}\big\{ \frac{1}{2}\iint \rho^{-1}|m|^2\, {\rm d}t\, {\rm d}x\,; \ \p_t \rho + \ddiv\, m = 0 \big\}\,.
\end{equation}
This dynamic point of view has since stimulated numerous follow-up studies, including the present work. We refer the interested readers to \cite{villani2003topics, villani2008optimal} for the precise statements of aforementioned results and a detailed overview. 

\medskip 

\noindent
\textbf{Unbalanced optimal transport.}  
Although the OT theory has become a popular tool in learning theory and data science for its geometric nature and capacity for large-scale simulation, a limitation is that the associated metric is only defined for measures of equal mass, while in many applications, it is more desirable to allow measures with different masses. This leads to the problem of extending the classical OT theory to the unbalanced case. The early effort in this direction may date back to the works \cite{kantorovich1957functional,kantorovich1958space} by Kantorovich and Rubinshtein in the 1950s, where a simple static formulation with an extended Kantorovich norm was introduced. The underlying idea is to allow the mass to be sent to (or come from) a point at infinity, which was further investigated and extended in \cite{hanin1999extension,guittet2002extended}. Similarly, Figalli and Gigli \cite{figalli2010new} introduced an unbalanced transportation distance via a variant of Kantorovich formulation \eqref{basicot} by allowing taking the mass from (or giving it back to) the boundary of the domain. Another closely related  approach is the optimal partial transport \cite{figalli2010optimal,caffarelli2010free}, which is also based on \eqref{basicot} but involves a relaxed constraint $(\pi_\#^x \gamma, \pi_\#^y \gamma) \le (\rho_0,\rho_1)$ and a shifted cost $|x-y|^2-\alpha$. 

In addition to the static models, there is a large number of works devoted to defining an unbalanced OT model via a dynamic formulation in the spirit of
\cite{benamou2000computational}; see for example \cite{lombardi2015eulerian,maas2015generalized,benamou2003numerical,chizat2018interpolating,piccoli2016properties}. In these works, a source term and a corresponding penalization term are introduced in the continuity equation and the action functional, respectively, in order to model the mass change. In particular, Piccoli and Rossi \cite{piccoli2014generalized,piccoli2016properties} defined a generalized Wasserstein distance by relaxing the marginal constraint $(\pi_\#^x \gamma, \pi_\#^y \gamma) = (\rho_0,\rho_1)$ by a total variation regularization, which turns out to be equivalent to the optimal partial transport in certain scenarios \cite{chizat2018interpolating}. Moreover, an equivalent dynamic formulation has also been given in \cite{piccoli2016properties}. Later, a new transport model, called the Wasserstein-Fisher-Rao (WFR) or Hellinger-Kantorovich distance (in this work we adopt the former one), was introduced independently and almost simultaneously by three research groups with different perspectives and techniques \cite{chizat2018interpolating, liero2016optimal,
kondratyev2016new}. This model can be regarded as an inf-convolution of the 
Wasserstein and Fisher-Rao metric tensors, as the name suggests. In their subsequent work \cite{chizat2018unbalanced}, Chizat et al. presented a class of unbalanced transport distances in a unified framework via both static and dynamic formulations, thanks to the notions of semi-couplings and Lagrangians. Meanwhile,  Liero et al. \cite{liero2018optimal} proposed a related optimal entropy-transport approach and discussed its properties in detail.
 It was proved that both the optimal partial transport and the WFR distance can be viewed as the special cases of the general frameworks in  \cite{chizat2018unbalanced,liero2018optimal}. After that, the unbalanced OT  theory is further developed in various directions, such as gradient flows \cite{kondratyev2019spherical,kondratyev2020nonlinear}, Sobolev inequalities \cite{kondratyev2020convex}, and the JKO scheme \cite{gallouet2017jko,fleissner2021minimizing}. We also want to mention a recent work \cite{lombardini2022obstructions} by Lombardini and Rossi, which gave a negative answer to an interesting question of whether it is possible to define an unbalanced transport distance that coincides with the Wasserstein one when the measures are of equal mass.

\medskip

\noindent \textbf{Noncommutative optimal transport.} More recently, there is also an increasing interest in generalizing the OT theory to the noncommutative setting, namely, the quantum states or matrix-valued measures. The first line of research is motivated by the ergodicity of open quantum dynamics \cite{gross1975hypercontractivity,kastoryano2013quantum,olkiewicz1999hypercontractivity}. In the seminar works \cite{carlen2014analog,carlen2017gradient} by Carlen and Maas, a quantum Wasserstein distance was introduced with a Benamou-Brenier dynamic formulation such that the primitive quantum Markov semigroup satisfying the detailed balance condition can be formulated as the gradient flow of the logarithmic relative entropy, which opens the door to investigating the noncommutative functional inequalities via the gradient flow techniques and 
the geodesic convexity; see for example \cite{rouze2019concentration,datta2020relating,li2022interpolation,wirth2021curvature}. Meanwhile, Golse et al. proposed another quantum transport model via a generalized Monge-Kantorovich formulation, when they studied the 
mean-field and classical limits of the Schr\"{o}dinger equation; see\cite{golse2016mean,golse2017schrodinger,golse2018wave}. Other static quantum Wasserstein distances can be found in \cite{de2021quantum,de2021quantum2,cole2023quantum}, just to name a few. 

The second research line is driven by the advances in diffusion tensor imaging \cite{le2014diffusion,wandell2016clarifying}, where a tensor field (usually, a positive semi-definite matrix) is generated at each spatial position to encode the local diffusivity of water molecules in the brain. It gives rise to a natural question of how to compare
two brain tensor fields, or mathematically how to define a reasonable distance between matrix-valued measures. 
Chen et al. \cite{chen2017matrix, chen2020matrix} introduced a dynamic matricial Wasserstein distance for matrix-valued densities with unit mass, drawing inspiration from \cite{benamou2000computational} and leveraging the Lindblad equation in quantum mechanics, which was later extended to the unbalanced case \cite{chen2019interpolation} in a manner similar to \cite{chizat2018interpolating}. 
In particular,
Brenier and Vorotnikov \cite{brenier2020optimal} recently proposed a different dynamic 
OT model for unbalanced matrix-valued measures called the Kantorovich-Bures metric, which is motivated by the observation in \cite{brenier2018initial} that the incompressible Euler equation admits a dual concave maximization problem. Regarding static formulations, Peyr\'{e} et al. \cite{peyre2019quantum} introduced a quantum transport distance with entropic regularization inspired by \cite{liero2018optimal} and proposed an associated scaling algorithm that generalized the results in \cite{chizat2018scaling}. Additionally, Ryu et al. defined a matrix optimal transport model of order $1$ by a Beckmann-type flux formulation and presented a scalable and parallelizable numerical method. Applications in tensor field imaging were also explored in \cite{peyre2019quantum,ryu2018vector}.

\medskip

\noindent \textbf{Contribution.} 
The initial motivation for this work is the numerical study of the unbalanced matricial OT models proposed in \cite{chen2019interpolation,brenier2020optimal}; see \eqref{def:kan_bure} and \eqref{def:chen19int}. We find that despite their distinct formulations, these models actually share many mathematical properties. In this work, we consider an abstract continuity equation 
$\p_t \ms{G} + \ms{D} \ms{q} = \ms{R}^{{\sym}}$ in Definition \ref{def:abs_conti_eq} with $\ms{D}$ being a first-order constant coefficient linear differential operator such that $\ms{D}^*(I) = 0$, in analogy with the one $\p_t \ms{G} + 2 (\ms{L}^* \circ \ms{P})\, \ms{q} = 0$ for the matrix-valued optimal ballistic transport problem  (cf.\,\cite[(1.4)--(1.5)]{vorotnikov2022partial}). Here, $\ms{q}(t,x)$ can be intuitively seen as a momentum variable; $\ms{D}q$ is the matricial analog of 
the advection term $\ddiv\, m$ in \eqref{prob:BB_2} controlling the mass transportation in space and between components; $\ms{R}^{{\sym}}$ is the reaction part describing the variation of mass. Then, thanks to the weighted infinitesimal cost $J_\Lad(G_t, q_t, R_t) = 
\frac{1}{2} (q_t \Lad_1^\dag) \dd G_t^{\dag} (q_t \Lad_1^\dag) + \frac{1}{2} (R_t \Lad_2^\dag) \dd G^{\dag}_t (R_t \Lad_2^\dag)$ given in Proposition \ref{prop:subgrad_J} with the weighted matrices $\Lad_1$ and $\Lad_2$ representing the contributions of each component of $q$ and $G$ in $J_\Lad$, we define a general matrix-valued unbalanced OT distance ${\rm WB}_{\Lad}(\dd,\dd)$ \eqref{eq:distance} as a convex optimization, similarly to the classical case \eqref{prob:BB_2}, which we call the weighted Wasserstein-Bures distance; see Definition \ref{def:metric}. We note that the problems \eqref{def:kan_bure} and \eqref{def:chen19int}, as well as the scalar WFR distance \eqref{def:wfr_metric}, can be viewed as the special instances of our model \eqref{eq:distance}. See Section \ref{sec:example_model} for more details.

Our main contribution is a comprehensive and self-contained study of the properties of the weighted distance ${\rm WB}_{\Lad}$ on the positive semi-definite matrix-valued Radon measure space $\mc{M}(\Omega,\S_+^n)$. We establish the a priori estimates for solutions of the abstract continuity equation \eqref{eq:weak_ctneq} in Lemmas \ref{lem:com_kb_and_b}, \ref{lem:priorimu} and Proposition \ref{prop:disinte_path}, which consequently gives the well-posedness of the model \eqref{eq:distance} and a useful compactness result Proposition \ref{prop:compact}. Then, by leveraging tools from convex analysis, we show the existence of the minimizer (i.e., the minimizing geodesic) to \eqref{eq:distance} with a characterization of the optimality conditions; see Theorems \ref{thm:strong_dual} and \ref{thm:dual_minmax}. Moreover, we prove that the topology induced by  ${\rm WB}_\Lad(\dd,\dd)$ is stronger than the weak* one and study the limit model when a weighted matrix goes to zero; see Propositions \ref{prop:metric_prop} and \ref{prop:limitweight}, respectively. With the help of these results, in Theorem \ref{prop:for_geodesic} and Corollary \ref{cor:geo_me}, 
we characterize the absolutely continuous curve with respect to the 
metric ${\rm WB}_\Lad$ and show that $(\mc{M}(\Omega, \S_+^n), {\rm WB}_\Lad)$ is a complete geodesic space. We further consider its conic structure and prove in Theorem \ref{thmcone} 
that the space $(\mc{M}(\Omega, \S_+^n), {\rm WB}_\Lad)$ is a metric cone over $(\mc{M}_1, {\rm SWB}_\Lad)$, where $\M_1$ is a normalized matrix-valued measure space \eqref{def:normeasure}, which corresponds to a noncommutative probability space, 
and ${\rm SWB}_\Lad$ is a spherical distance \eqref{def:sphdist} induced by ${\rm WB}_\Lad$. Recalling the Riemannian interpretation in Corollary \ref{coro:riemann}, we can formally view $(\mc{M}(\Omega, \S_+^n), {\rm WB}_\Lad)$ as a Riemannian manifold and $\mc{M}_1$ as its submanifold with the induced metric ${\rm SWB}_\Lad$, which allows developing the Otto calculus in the spirit of \cite{otto2000generalization}. These results can be readily applied to the models \eqref{def:kan_bure} and \eqref{def:chen19int}, which lay a solid mathematical foundation for
the distance \eqref{def:chen19int} and complement the results in \cite{brenier2020optimal} for \eqref{def:kan_bure} (note that our approach is quite different from theirs).

In the companion work \cite{li2020general2}, we have designed a convergent discretization scheme for the general model \eqref{eq:distance}, which directly applies to the Kantorovich-Bures distance \eqref{def:kan_bure} \cite{brenier2020optimal}, the matricial interpolation distance \eqref{def:chen19int} \cite{chen2019interpolation}, and the WFR metric \eqref{def:wfr_metric} \cite{chizat2018interpolating}, thanks to the discussion in Section \ref{sec:example_model} of the present work. 

\medskip
\noindent \textbf{Layout.} The rest of this work is organized as follows. In Section \ref{sec:notation}, we give a list of basic notations that will be used throughout this work and recall some preliminary results. In Section \ref{sec:basicdef}, we define a class of weighted Wasserstein-Bures distances for matrix-valued measures via a dynamic formulation. Sections \ref{sec:existenandopt}
and \ref{sec:Riema_iter} are devoted to its topological, metric, and geometric properties, while in Section \ref{seccone}, we discuss its conic structure.
In Section \ref{sec:example_model}, we connect our general model with several existing models in the literature. Some auxiliary proofs are included in Appendix \ref{app_B}.

\section{Preliminaries and Notation} \label{sec:notation} 

\noindent \textbf{Notation and convention.}
\begin{itemize}
 \setlength\itemsep{-0.1mm}
    \item We denote by $\R^{n \t m}$ the space of $n \t m$ real matrices. If $m = n$, we simply write it as $\MM^{n}$. Moreover, we use $\S^n$, $\S_+^n$, and $\S^n_{++}$ to denote symmetric matrices, positive semi-definite matrices, and positive definite matrices, respectively. $\AA^n$ denotes the space of $n \t n$ antisymmetric matrices. 
    \item We denote by $|\dd|$ the Euclidean norm on $\R^n$.
    We equip the matrix space $\R^{n \t m}$ with the Frobenius inner product $A \dd B = \tr(A^{\rm T} B)$ and the associated norm  $\norm{A}_\ff = \sqrt{ A \dd A}$. 
    \item The symmetric and antisymmetric parts of $A \in \MM^n$ are given by 
    \begin{equation} \label{eq1}
        A^{\sym} = (A + A^{\rm T})/2\,,\q A^{{\rm ant}} = (A - A^{\rm T})/2\,,
    \end{equation}
    respectively. We also write $A \preceq B$ (resp., $A \prec B$) for $A, B \in \S^n$ if $B - A \in \S^n_+$ (resp., $B - A \in \S^n_{++}$). 
    \item $\mathcal{X}$ denotes a generic compact separable metric space with Borel $\sigma$-algebra $\mathscr{B}(\mathcal{X})$, unless otherwise specified.
    \item $C(\mathcal{X},\R^n)$ denotes the space of $\R^n$-valued continuous functions on $\mathcal{X}$ with the supremum norm  $\norm{\dd}_\infty$. Its dual space, denoted by $\M(\mathcal{X},\R^n)$, is $\R^n$-valued Radon measure space with the total variation norm $\norm{\dd}_{\rm TV}$.
\item Let $\mc{B}$ be a Banach space with the dual space $\mc{B}^*$. We denote by $\l \dd, \dd \r_{\mc{B}}$ the duality pairing between $\mc{B}$ and $\mc{B}^*$. When $\mc{B} = C(\mathcal{X},\R^n)$, we usually write it as $\l \dd, \dd \r_{\mc{X}}$ for short. We will also consider the weak and weak* convergences on $\mc{B}$ and $\mc{B}^*$, respectively. In particular, a sequence of measures $\{\mu_j\}$ weak* converges to $\mu \in \M(\mathcal{X},\R^n)$ if for any $\phi \in C(\mathcal{X},\R^n)$, there holds 
    $\l\mu_j,  \phi\r_{\mathcal{X}} \to \l\mu, \phi\r_{\mathcal{X}}$ as $j \to +\infty$.  
\item Let $\R_+: = [0,\infty)$, and $\M(\mathcal{X},\R_+)$ be the space of nonnegative finite Radon measures. 
For $\mu \in \M(\mathcal{X},\R^n)$, we have  
an associated variation measure $|\mu|\in \M(\mathcal{X},\R_+)$ such that ${\rm d} \mu = \sigma {\rm d} |\mu|$ with $|\sigma(x)| = 1$ for $|\mu|$-a.e.\,$x \in \mathcal{X}$, where $\sigma: \mathcal{X} \to \R^n$ is 
the Radon-Nikodym derivative (density) of $\mu$ with respect to $|\mu|$ \cite{evans2015measure,rudin2006real}. 
\item We identify the space of matrix-valued Radon measures  $\M(\mathcal{X},\R^{n \t m})$ with 
$\M(\mathcal{X},\R^{nm})$ by vectorization. It is easy to see that both sets of $\S^n$-valued Radon measures $\M(\mathcal{X},\S^n)$ and $\S^n_+$-valued Radon measures $\M(\mathcal{X},\S_+^n)$ are closed in $\M(\mathcal{X},\MM^n)$  with respect to the weak* topology \cite[Theorem 3.5]{duran1997lpspace}.  Moreover, we have the following characterization: 
\begin{equation*}
    (C(\mathcal{X}, \S^n))^* \simeq  (C(\mathcal{X},\MM^n) /C(\mathcal{X}, \AA^n))^* \simeq \M(\mathcal{X}, \S^n)\,,
 \end{equation*}
where $\simeq$ means the isometric isomorphism and $C(\mathcal{X},\MM^n) /C(\mathcal{X}, \AA^n)$ is the quotient space. Indeed, 
we observe that $\mu \in \M(\mathcal{X},\S^n) \subset \M(\mathcal{X},\MM^n) \simeq C(\mathcal{X},\MM^n)^* $ if and only if its induced linear functional on $C(\mathcal{X},\MM^n)$ has the kernel $C(\mathcal{X}, \AA^n)$, which yields, by \cite[Proposition 11.9]{brezis2010functional}, 
 \begin{equation*}
    (C(\mathcal{X},\MM^n) /C(\mathcal{X}, \AA^n))^* \simeq \M(\mathcal{X}, \S^n)\,.
\end{equation*}
Meanwhile, $C(\mathcal{X}, \S^n) \simeq  C(\mathcal{X},\MM^n) /C(\mathcal{X}, \AA^n)$ is a consequence of $\S^n \perp \AA^n$ and $\S^n \simeq \MM^n/\AA^n$.

\item For $\mu \in \M(\mathcal{X}, \S_+^n)$, we define an associated trace measure
$\tr\mu$ by the set function $E \to \tr (\mu (E))$, $E \in \mathscr{B}(\mathcal{X})$. It is clear that $ 0 \preceq \mu(E) \preceq \tr (\mu (E)) I$ and $ \tr\mu$ is equivalent to
$|\mu|$, denoted by $\tr\mu \sim |\mu|$. That is, 
\begin{equation} \label{eq:trtotal}
   |\mu| \ll \tr\mu \q \text{and} \q \tr\mu \ll |\mu|\,.
\end{equation} 
We will usually use $\tr\mu$ as the dominant measure for $\mu \in \M(\mathcal{X},\S_+^n)$. In addition, note that for $\lad \in \M(\mathcal{X}, \R_+)$ with $|\mu| \ll \lad$, there holds $\frac{\rd \mu}{\rd \lad} \in \S^n_+$ for $\lad$-a.e.\,$x \in \mathcal{X}$, which is an equivalent characterization of $\M(\mathcal{X}, \S_+^n)$.

\item We will use sans serif letterforms to denote vector-valued or matrix-valued measures, e.g., $\ms{A} \in \M(\mathcal{X}, \MM^n)$, while letters with serifs are reserved for their densities with respect to some reference measure, e.g., $A_\lad: = \frac{\rd \ms{A}}{\rd \lad}$ for $|\ms{A}| \ll \lad$. The symmetric and antisymmetric parts $\ms{A}^\sym$ and $\ms{A}^{\rm ant}$ of $\ms{A} \in \M(\mathcal{X}, \MM^n)$ are defined as in \eqref{eq1}. 

\item We identify a measure and its density with respect to the Lebesgue measure (if exists) unless otherwise specified. 

 \item  
 For $\lad \in \M(\mathcal{X}, \R_+)$, we denote by $L^p_\lad (\mathcal{X},\R^n)$ with $p \in [1, +\infty]$ the standard space of $p$-integrable $\R^n$-valued functions. For $\ms{G} \in \M(\mathcal{X}, \S_+^n)$, we consider the space of $\R^{n \times m}$-valued measurable functions endowed with the semi-inner product: 
\begin{equation}  \label{eq2}
    \l P, Q \r_{L^2_{\ms{G}}(\mathcal{X})} := \l \ms{G} , QP^{\rm T} \r_\mathcal{X} = \int_\mathcal{X}  P \dd (\rd \ms{G}\, Q) = \int_\mathcal{X}  P \dd \big(G_\lad Q \big)\, \rd \lad\,, 
\end{equation} 
where $\lad$ is a reference measure such that $|\ms{G}|\ll \lad$ and $G_\lad$ is the density. 
Noting that $\norm{Q}_{L^2_{\ms{G}}(\mathcal{X})} = 0$ is equivalent to $G_\lad Q  = 0$  for $\lad$-a.e.\,$x \in \mathcal{X}$, the kernel of the seminorm $\norm{\dd}_{L^2_{\ms{G}}(\mathcal{X})}$ is given by 
$\{Q\,;\ \ran(Q) \in \ker(G_\lad)\,, ~\lad\text{-a.e.}\}$.  
Then, we define the Hilbert space $L^2_{\ms{G}}(\mathcal{X}, \R^{n \times m})$ as the quotient space by $\ker\big(\norm{\dd}_{L^2_{\ms{G}}(\mathcal{X})}\big)$. 
\end{itemize}

\medskip

\noindent \textbf{Preliminaries.}
We denote by $A^\dag \in \R^{m \t n}$ the pseudoinverse of a matrix $A \in \R^{n \t m}$. If $A \in \S^n$ has the eigendecomposition $A = O \Sigma O^{\rm T}$, then $A^\dag = O \Sigma^\dag O^{\rm T}$ with $\Sigma^\dag = \text{diag}(\lad_1^{-1}, \ldots, \lad_s^{-1},0, \ldots,0)$, where $O$ is an orthogonal matrix and $\Sigma = \text{diag}(\lad_1,\ldots, \lad_s,0, \ldots,0)$
is a diagonal matrix with $\{\lad_i\}$ being nonzero eigenvalues of $A$. 

\begin{lemma} \label{lem:useprop} 
    The following properties hold:
    \begin{enumerate}
    \item If $A \succeq B \succeq 0$ and $\ran(A) = \ran(B)$, then $B^{\dag} \succeq A^{\dag}$. 
    \item The cone $\S^n_+$ in $\S^n$ is self-dual, that is, $(\S_+^n)^* := \{B\in\S^n\,; \ \tr(AB) \ge 0\,,\ \forall A \in \S^n_+ \} = \S^n_+$.
    \item  If $A, B \succeq 0$ and $A \dd B = 0$, then  $\ran B \subset \ker A$, equivalently, $\ran A \subset \ker B$. 
    \item For $A \in \S_+^n, M \in \R^{n \t m}$, there holds 
    \begin{equation} \label{eq:useineq}
        (A M) \dd M  \le \tr(A)\norm{M}_\ff^2\,.
    \end{equation}
\end{enumerate}
\end{lemma}

\begin{remark}
The range condition $\ran(A) = \ran(B)$ for the first statement in Lemma \ref{lem:useprop} above
is necessary, due to the example $A = \text{diag}(1,1,1,0)$ and $B = \text{diag}(1,1,0,0)$. Moreover, we remark that for $\ms{G} \in \M(\mathcal{X}, \S_+^n)$, 
there holds $L_{\tr \ms{G}}^2(\mathcal{X}, \R^n) \subset L_{\ms{G}}^2(\mathcal{X}, \R^n)$ by 
\eqref{eq:useineq}, while the converse is not true; see \cite{duran1997lpspace} for the counterexample. 
\end{remark}

\begin{proof}
We only prove the first statement, as the others are direct. We first note that the orthogonal projection onto 
$\ran(A) = \ran(B)$ is given by $\P = \sqrt{B}^\dag B \sqrt{B}^\dag = \sqrt{A}^\dag A \sqrt{A}^\dag $. By $A - B \succeq 0$, we also have $\sqrt{B}^\dag A \sqrt{B}^\dag - \P \succeq 0$, which means that all 
the eigenvalues of the matrix $\sqrt{B}^\dag A \sqrt{B}^\dag$ restricted on its invariant subspace $\ran(A) = \ran(B)$ is greater than or 
equal to one. It is easy to see that $\sqrt{B}^\dag A \sqrt{B}^\dag$ and $\sqrt{A} B^\dag \sqrt{A}$ have the same eigenvalues. Hence, we 
find $\sqrt{A} B^\dag \sqrt{A} - \P \succeq 0$, which gives $B^\dag \succeq A^\dag$ by conjugating with  $\sqrt{A}^\dag$. 
\end{proof}

The next lemma is about the measurability of matrix-valued functions. 

\begin{lemma}\label{prop:measmatfuns}
    Let $A(x)$ be a $\S^n$-valued Borel measurable function on $\mathcal{X}$. Then, it holds that 
    \begin{enumerate}
        \item The eigenvalues $\{\lad_{A,i}(x)\}^n_{i = 1}$ of $A(x)$ in nondecreasing order are measurable, and the corresponding eigenvectors $\{u_{A,i}(x)\}^n_{i=1}$ can also be selected to be measurable and form an orthonormal basis of $\R^n$ for every $x \in \mathcal{X}$. 
         \item The pseudoinverse $A^\dag(x)$ of $A(x)$ is measurable, and the square root  $A^{1/2}(x)$ of $A(x) \in \S^n_+$ is measurable.  
    \end{enumerate}
\end{lemma}

\noindent The first and second properties are from  
\cite{reid1970some} and \cite{robertson1968decomposition} with the continuity of 
$A^{1/2}$ in  $A \in \S^n_+$, 
respectively.  In fact, Powers-St{\o}rmer inequality \cite{powers1970free} gives 
\begin{align} \label{est:lip_squre}
    \big \lVert\sqrt{A} - \sqrt{B}\,\big\rVert_\ff^2 \le \sqrt{n} \norm{A - B}_\ff\,, \q \forall A,B \in \S^n_+\,.
\end{align}

We finally recall some concepts and useful results from convex analysis.
Let $f: X \to \R \cup \{+ \infty\}$ be an extended real-valued function on a Banach space $X$. We denote by $\p f(x)$ its subgradient at $x \in X$ and by $\dom(f) := f^{-1}(\R)$ its domain. We say that $f$ is proper if $\dom(f) \neq \varnothing$; and that  $f$ is positively homogeneous of degree $k$ if for all $x \in X$ and $\alpha > 0$, $f(\alpha x) = \alpha^k f(x)$. 
The conjugate function $f^*$ of $f$ is defined by 
\begin{equation} \label{def:conjugate}
    f^*(x^*) = \sup_{x \in X} \l x^*, x\r_X - f(x)\,, \q \forall x^* \in X^*\,,
\end{equation}
which is convex and lower semicontinuous with respect to the weak* topology of $X^*$. The following two lemmas are from \cite[Proposition 2.33]{barbu2012convexity} and \cite[Proposition 2.5]{bouchitte2020convex}, respectively. 

\begin{lemma}[Subgradient] \label{lem:subgrad_relation}
Let $f:X \to \R \cup \{+\infty\}$ be a proper convex function on a Banach space $X$. Then, the following three properties are equivalent: {\rm (i)} $x^* \in \p f(x)$; {\rm (ii)} $f(x) + f^*(x^*) =  \l x^*,x\r_X$; {\rm (iii)} $f(x) + f^*(x^*) \le \l x^*,x\r_X$.   In addition, if $f$ is lower semicontinuous, then all of these properties are equivalent to  $x \in \p f^*(x^*)$. 
\end{lemma} 

\begin{lemma}[Fenchel--Rockafellar duality] \label{prop:fench_rock}
 Let $X$ and $Y$ be two Banach spaces and 
 $L: X \to Y$ be a bounded linear operator 
with the adjoint $L^*: Y^* \to X^*$. Let $f$ and $g$ be two proper lower semicontinuous convex functions defined on $X$ and $Y$ valued in 
$\R \cup \{+\infty\}$, respectively. If there exists $x \in \dom (f)$ such that  $g$ is continuous at $Lx$,  then
    \begin{equation} \label{eq:infsup}
        \sup_{x \in X} - f(-x) - g(Lx) = \inf_{y^* \in Y^*}f^*(L^*y^*) + g^*(y^*)\,,
    \end{equation}
    and the $\inf$ in \eqref{eq:infsup} can be attained. Moreover, the $\sup$ in \eqref{eq:infsup} is attained at $x \in X$ if and only if  there exists a $y^* \in Y^*$ such that $L x \in \p g^*(y^*)$ and $L^* y^* \in \p f(-x)$, in which case $y^*$ also
    achieves the $\inf$ in \eqref{eq:infsup}. 
 \end{lemma}

\section{Definition and basic properties} \label{sec:basicdef}

 We shall introduce a new family of distances on the matrix-valued Radon measure space $\M(\Omega, \S^n_+)$ based on a dynamic optimal transport (OT) formulation, which will be the central object of this work. 

 \medskip

\noindent \textbf{Action functional.} To define our dynamic OT model over the space of $\S^n_+$-valued measures, the starting point is a weighted action functional. Let $n, k, m \in \NN$ be positive integers and $\Lad := (\Lad_1,\Lad_2)$ be a pair of matrices with $\Lad_1 \in \S^k_+$ and $\Lad_2 \in \S_+^m$. We define the following closed convex set:
\begin{align} \label{eq:closed_convex_set}
   \mathcal{O}_\Lad=  \Big\{(A,B,C) \in  \S^n \t \R^{n \t k} \t \R^{n \t m}\,;\    A + \frac{1}{2} B \Lad^2_1 B^{\rm T} + \frac{1}{2} C \Lad^2_2 C^{\rm T}  \preceq 0 \Big\}\,.
\end{align}
Note that its characteristic function:
\begin{equation*}
\iota_{\mathcal{O}_\Lad} := \begin{cases}
    0, & (A,B,C) \in \mathcal{O}_\Lad\,,\\
+ \infty, & (A,B,C) \notin \mathcal{O}_\Lad\,,\\
\end{cases}    
\end{equation*}
is proper 
lower semicontinuous and convex \cite[Lemma 1.24]{bauschke2011convex}. We denote by $J_\Lad$ 
the conjugate function \eqref{def:conjugate} of $\iota_{\mathcal{O}_\Lad}$ and derive the explicit expressions for $J_\Lad$ and its subgradient $\p J_\Lad$.  

\begin{proposition} \label{prop:subgrad_J} 
$J_\Lad$ is proper, positively homogeneous of degree one, lower semicontinuous, and convex with the following representation:
\begin{equation} \label{eq:expre_J}
    J_\Lad(X, Y, Z) = 
         \frac{1}{2}  (Y \Lad_1^\dag) \dd (X^{\dag} Y \Lad_1^\dag) 
          + \frac{1}{2} (Z \Lad_2^\dag) \dd (X^{\dag} Z \Lad_2^\dag)\,,
\end{equation}
if $X \in \S_+^n$, $\ran (Y^{\rm T})  \subset \ran ( \Lad_1)$, $\ran (Z^{\rm T})  \subset \ran (\Lad_2)$ and $\ran ([Y,Z]) \subset \ran(X)$; otherwise $J_\Lad(X, Y, Z) = +\infty$. 
Moreover, the subgradient of $J_\Lad$ at $(X,Y, Z) \in \dom (J_\Lad)$ 
is characterized by
\begin{equation} \label{eq:chacsubgradjal}
    \p J_\Lad (X,Y, Z) =  \Big\{(A,B,C) \in \mathcal{O}_\Lad\,; \ Y =  X B \Lad^2_1\,, \ Z = X C \Lad^2_2\,, \   X \dd \Big(A + \frac{1}{2} B \Lad_1^2 B^{{\rm T}} + \frac{1}{2} C \Lad^2_2 C^{{\rm T}} \Big) = 0 \Big\}\,.
\end{equation} 
$\p J_\Lad (X, Y, Z)$ is a singleton if and only if 
$(X, Y, Z) \in \S^n_{++} \t \R^{n \t k} \t \R^{n \t m}$ and $\Lad_1 \in \S_{++}^k$, $\Lad_2 \in \S_{++}^m$.   
\end{proposition}

\begin{proof} 
The properties of $J_\Lad$ are by \cite[Proposition 14.11]{bauschke2011convex}. To derive the formula \eqref{eq:expre_J}, by definition, we have
    \begin{align} \label{def:jalpha}
        J_\Lad(X, Y, Z) &= \sup_{(A,B,C) \in  \mathcal{O}_\Lad}  X \dd A  +  Y \dd B  +  Z \dd C \,,
    \end{align} 
    for $(X, Y, Z) \in  \S^n \t \R^{n \t k} \t \R^{n \t m}$. 
We consider the following four cases. 

\smallskip
\noindent \emph{Case I}: $X \in \S^n\backslash \S_+^n$. We choose a vector $a \in \R^n$ such that $\l a, X a \r < 0$ and set $A = - \lad a a^{\rm T}\preceq 0 $ with $\lad > 0$, $B = 0$, and $C = 0$  in \eqref{def:jalpha}. Then it follows that 
    \begin{equation*}
        J_\Lad(X, Y, Z) \ge \sup_{\lad > 0}  X \dd (- \lad a a^{\rm T}) = + \infty\,.
    \end{equation*}

\noindent \emph{Case II}: $\ran (Y^{\rm T}) \not \subset \ran (\Lad_1)$ or $\ran (Z^{\rm T}) \not \subset \ran (\Lad_2)$. It suffices to consider the case $\ran (Y^{\rm T}) \not \subset \ran (\Lad_1)$, since the same argument applies to the other one. Without loss of generality, we let $Y = [y_1,\ldots,y_n]^{\rm T}$ with $y_i \in \R^k$ and $y_1 \notin \ran (\Lad_1)$. Thanks to $\Lad_1 \in \S^k_+$, $y_1$ has the orthogonal decomposition:
\begin{equation*}
    y_1 = y_1^{(1)} + y^{(2)}_1 \q  \text{with}\ y^{(1)}_1 \in \ran(\Lad_1)\,,\  y^{(2)}_1 \neq 0 \in \ker (\Lad_1)\,.
\end{equation*}
Taking $A = 0$, $B = \lad \big[y_1^{(2)},0\big]^{\rm T}$ with $\lad \in \R$, and $C = 0$ in  \eqref{def:jalpha}, we have 
\begin{equation*}
    J_\Lad(X, Y, Z) \ge \sup_{\lad > 0} \lad \big| y_1^{(2)} \big|^2  = + \infty\,.
\end{equation*} 

\noindent \emph{Case III}:
 $\ran ([Y, Z]) \not \subset \ran(X)$. It suffices to consider $\ran( Y ) \not \subset \ran (X)$. 
We take $(A,B,C)$ in \eqref{def:jalpha} as:
\begin{align*}
    A = - \frac{\lad^2}{2} (\P_{\ker(X)} Y \Lad_1) (\P_{\ker (X)} Y \Lad_1)^{\rm T}\,, \ B = \lad \P_{\ker(X)} Y\,,\  C = 0\,,
\end{align*}
with $\lad > 0$, where $\P_{\ker(X)}: = I - X^\dag X$ is the orthogonal projection onto $\ker(X)$. A direct computation gives
\begin{align*}
    J_\Lad(X,Y, Z) &\ge \sup_{(A,B,0) \in  \mathcal{O}_\Lad}  X \dd  A  +  Y \dd B   \\
    & \ge \sup_{\lad > 0} - \frac{ \lad^2}{2} (\P_{\ker(X)} Y \Lad_1)\dd (X \P_{\ker(X)} Y \Lad_1) + \lad  Y \dd ( \P_{\ker(X)} Y) \\
    & \ge \sup_{\lad > 0} \lad  (\P_{\ker(X)} Y) \dd (\P_{\ker(X)} Y) = + \infty\,,
\end{align*}
since there holds $( \P_{\ker(X)} Y) \dd (\P_{\ker(X)} Y) > 0$ by
$\ran( Y ) \not \subset \ran (X)$.

\smallskip
\noindent \emph{Case IV}: $(X, Y,Z) \in \S_+^{n} \t \R^{n \t k} \t \R^{n \t m}$ with $\ran (Y^{\rm T})  \subset \ran (\Lad_1)$, $\ran (Z^{\rm T})  \subset \ran (\Lad_2)$ and $\ran ([Y, Z]) \subset \ran(X)$.  For this case, we directly compute
\small
  \begin{align} \label{auxcal_1}
        X \dd A +  Y \dd B  +  Z \dd C  = &  X \dd \Big(A + \frac{1}{2} B \Lad^2_1 B^{\rm T}  + \frac{1}{2} C \Lad^2_2 C^{\rm T} \Big) + Y \dd  B  +  Z \dd C  -  X  \dd \Big( \frac{1}{2} B \Lad_1^2 B^{{\rm T}} + \frac{1}{2} C \Lad^2_2 C^{\rm T} \Big)\,,
    \end{align}
\normalsize
and
\small
\begin{align} \label{auxcal_2}
     Y \dd B  +  Z \dd  C  -  \frac{1}{2} X \dd \big(B \Lad_1^2 B^{{\rm T}} +  C \Lad^2_2 C^{\rm T} \big) = & - \frac{1}{2} \Big \lVert \sqrt{X}B \Lad_1 - \sqrt{X}^{\dag} Y \Lad_1^\dag \Big \lVert_\ff^2 - \frac{1}{2} \Big \lVert\sqrt{X}C \Lad_2 - \sqrt{X}^{\dag} Z \Lad_2^\dag \Big \lVert_\ff^2 \notag \\
        & + \frac{1}{2 } \Big \lVert \sqrt{X}^{\dag} Y \Lad_1^\dag \Big \lVert_\ff^2 + \frac{1}{2} \Big \lVert\sqrt{X}^{\dag} Z \Lad_2^\dag \Big \lVert_\ff^2\,,
\end{align}
    \normalsize
    where we have used 
    $$  Y \dd B  +  Z \dd C  = \big (\sqrt{X} \sqrt{X}^\dag Y \Lad_1^\dag \Lad_1 \big) \dd  B  + \big (\sqrt{X} \sqrt{X}^\dag Z \Lad_2^\dag \Lad_2 \big) \dd  C  \,,$$
by the range relations: $\ran (Y^{\rm T})  \subset \ran (\Lad_1)$, $\ran (Z^{\rm T})  \subset \ran (\Lad_2)$, and $\ran ([Y,Z]) \subset \ran(X)$. Also, by \eqref{eq:closed_convex_set}, we have $ X \dd \big(A + \frac{1}{2} B \Lad_1^2 B^{{\rm T}} + \frac{1}{2} C \Lad_2^2 C^{\rm T} \big) \le 0$. Hence, by \eqref{auxcal_1} and \eqref{auxcal_2}, the maximizers to 
  \eqref{def:jalpha} are given by the set 
  \begin{equation} \label{auxeq_2}
    \Big\{(A,B,C) \in \mathcal{O}_\Lad\,; \ Y = X B \Lad^2_1 \,, \ Z = X C \Lad^2_2\,, \   X \dd \Big(A + \frac{1}{2} B \Lad_1^2 B^{{\rm T}} + \frac{1}{2} C \Lad^2_2 C^{\rm T} \Big) = 0 \Big\}\,,
   \end{equation}
   and the corresponding supremum is \eqref{eq:expre_J}.

Finally, to characterize the subgradient of $J_\Lad$, by Lemma \ref{lem:subgrad_relation}, we have that 
$(A,B,C) \in \p J_\Lad(X,Y,Z)$ if and only if $(A,B,C) \in \mathcal{O}_\Lad$ and  
 $ J_\Lad(X,Y,Z) =  X \dd A +  Y \dd B  +  Z \dd C$ holds. Then, \eqref{eq:chacsubgradjal} readily follows from the above argument. For the last statement, we note that $\p J_\Lad (X,Y,Z)$  is a singleton if and only if the equations in \eqref{eq:chacsubgradjal} for $(A,B,C)$ are uniquely solvable, which is equivalent to $\Lad_1 \in \S_{++}^k$, $\Lad_2 \in \S_{++}^m$ and $X \in \S_{++}^n$. 
\end{proof}

In the following discussion, we assume $m = n$ and $\Lad_2 \in \S^n_{++}$; see Remark \ref{rem:bures}. Now, 
for a given triplet of measures $\mu: = \ms{(G,q, R)} \in  \M(\mathcal{X}, \S^n \t \R^{n \t k}  \t \MM^n)$, we define a positive measure $\jj_{\Lad} (\mu)$ on $\mathcal{X}$ by 
\begin{equation} \label{def:costmeasure}
\jj_{\Lad}(\mu)(E): = \int_E J_\Lad \left(\frac{\rd \mu}{\rd \lad}\right) \rd \lad\,,
\end{equation}
for a measurable set $E \in \mathscr{B}(\mathcal{X})$, where $\lad \in \M(\mathcal{X},\R_+)$ is a reference measure such that $|\mu| \ll \lad$. Thanks to the positive homogeneity of $J_\Lad$ by Proposition \ref{prop:subgrad_J}, 
the definition \eqref{def:costmeasure} of $\jj_{\Lad}$ is independent of the choice of $\lad$. To alleviate notations, we adopt the following conventions in the rest of this work. 

\begin{enumerate}[\textbullet] 
 \setlength\itemsep{-0.1mm}
    \item We define the space $\xx := \S^n \t \R^{n \t k}  \t \MM^n$ and then write $\M(\mathcal{X},\xx) = \M(\mathcal{X}, \S^n \t \R^{n \t k}  \t \MM^n) = C(\mathcal{X},\xx)^* $, where  $C(\mathcal{X},\xx) = C(\mathcal{X}, \S^n \t \R^{n \t k}  \t \MM^n)$. 
    \item We often write $\mu$ for  $\ms{(G,q,R)} \in \M(\mathcal{X},\xx)$ for short, which will be clear from the context.
    \item We write $\jj_{\Lad}(\mu)(E)$ as $\jj_{\Lad,E}(\mu)$ for short. Thus, $\jj_{\Lad,\mathcal{X}}(\mu)$ denotes the total measure $\jj_{\Lad}(\mu)(\mathcal{X})$. 
    \item We denote by $(G_\lad,q_\lad,R_\lad)$ the density of $\ms{(G,q,R)} \in \M(\mathcal{X},\xx)$ with respect to a reference measure $\lad \in \M(\mathcal{X},\R_+)$ such that $|\ms{(G,q,R)}| \ll \lad$. The subscript $\lad$ of $(G_\lad,q_\lad,R_\lad)$ will often be omitted for simplicity.  
    \item The generic positive constant $C$ involved in the estimates below may change from line to line. 
\end{enumerate}

\begin{definition} \label{def:actional_fun}
We define the \emph{$\Lad$-weighted action functional} for a measure $\mu \in \M(\mathcal{X}, \xx)$ by $\jj_{\Lad,\mathcal{X}}(\mu)$. 
\end{definition}

By Proposition \ref{prop:subgrad_J} and the formula \eqref{def:costmeasure},
we have the following useful lemma. 

\begin{lemma} \label{lem:fcabs}
    For $\mu = \ms{(G,q,R)} \in \M(\mathcal{X},\xx)$  with $\jj_{\Lad, \mathcal{X}}(\mu) < + \infty$,  we have  $\ms{G} \in \M(\mathcal{X},\S_+^n)$ and 
    $|(\ms{q}, \ms{R})| \ll \tr \ms{G}$ with 
\begin{equation} \label{eq:rela_range}  
        G_\lad \in \S_+^n\,,\  \ran\left([q_\lad, R_\lad ]\right) \subset  \ran \left(G_\lad\right)\,, \ \ran(q_\lad^{\rm T}) \subset \ran( \Lambda_1) \,, \ \ran(R_\lad^{\rm T}) \subset \ran(\Lambda_2)\,,\q \text{$\lad$-a.e.}\,.
\end{equation}
\end{lemma}
\begin{proof}
By $\jj_{\Lad, \mathcal{X}}(\mu) = \int_{\mathcal{X}} J_\Lad(\mu_\lad)\, \rd \lad < + \infty$, $J_\Lad(\mu_\lad)$ is finite for $\lad$-a.e.\,$x 
\in \mathcal{X}$, where $\mu_\lad = (G_\lad,q_\lad,R_\lad)$. It means that 
$\mu_\lad(x) \in \dom(J_{\Lad})$ holds $\lad$-a.e., which immediately gives \eqref{eq:rela_range} by Proposition \ref{prop:subgrad_J}. We next show the absolute continuity of $|\ms{q}|$ and $|\ms{R}|$ with respect to
$\tr \ms{G}$, that is, for $E \in \mathscr{B}(\mathcal{X})$ with  $\tr \ms{G} (E) = 0$, we have $|\ms{q}|(E) = |\ms{R}|(E) = 0$. For this, we consider two measurable subsets $E_1$ and $E_2$ of $E$ with $E = E_1 \cup E_2$:
     \begin{equation*}
         E_1 = \left\{x \in E\,;\ G_\lad(x) \in \S_+^n\backslash \{0\}\right\},\q E_2 = \left\{x \in E\,;\ G_\lad(x) = 0\right\}.
     \end{equation*}
 By $\tr \ms{G}(E_1) = 0$ and $\tr G_\lad > 0$ on $E_1$ everywhere, we have  $\lad(E_1) = 0$. Then $|\ms{q}|(E_1) = 0$ and $|\ms{R}|(E_1) = 0$ follows from $|\ms{q}|, |\ms{R}| \ll \lad$.  Moreover, by \eqref{eq:rela_range} and $G_\lad = 0$ on $E_2$, we have
$q_\lad(x) = 0$ and $R_\lad(x) = 0$ for $\lad$-a.e.\,$x \in E_2$. Then it follows that $|\ms{q}|(E_2) = 0$ and $|\ms{R}|(E_2) = 0$. The proof is complete. 
\end{proof}

\medskip

\noindent \textbf{Continuity equation.} Another key ingredient for the dynamic 
OT formulation is a matricial continuity equation; see Definition \ref{def:abs_conti_eq} below. Let us fix more notations.
\begin{enumerate}[\textbullet] 
 \setlength\itemsep{-0.1mm}
\item Let $\Omega \subset \R^d$ be a compact set with a nonempty interior, a smooth boundary $\p \Omega$, and the exterior unit normal vector $\n$. We denote by $Q_a^b : = [a,b] \t \Omega \subset \R^{1 + d}$ with $b > a > 0$ the associated time-space domain. If $[a,b] = [0,1]$, we simply write it as $Q$.
\item For a function $\Phi(t,x)$ on $Q_a^b$, we write $\Phi_t(\dd) := \Phi(t,\dd)$ if we regard it as a family of functions $\{\Phi_t\}_{t \in [a,b]}$ in $x$. 
\item We denote by $\pi^t: (t,x) \to t$ the projection. 
We use the subscript $\#$ to denote the pushforward by a map. For instance, for a measure $\mu$ on $Q_a^b$, $\pi^t_\# \mu = \mu \circ (\pi^t)^{-1}$ is the pushforward measure on $[a,b]$.   
\item Let $X$ and $Y$ be two Banach spaces. We denote by $\mc{L}(X,Y)$ the space of continuous linear operators from $X$ to $Y$ (simply $\mc{L}(X)$ if $X = Y$) and by $C_c^\infty(\R^d,X)$ the $X$-valued smooth functions with compact support. We also need $C^k$-smooth functions $C^k(\Omega, X)$, where we assume that the derivatives exist in the interior of $\Omega$ and can be continuously extended to the boundary. The norm on $C^k(\Omega, X)$ is defined by $\norm{\Phi}_{k,\infty} := \sum_{|\alpha| \le k} \sup_{x \in \Omega} \norm{D^\alpha \Phi(x)}$. Other similar notations are interpreted accordingly. 
\item We recall the indicator function of a set $A$:
\begin{equation} \label{def:indicator}
    \chi_A(x) = \begin{cases}
        1, & \text{if}\ x \in A\,,\\
        0, & \text{if}\ x \notin A\,.
    \end{cases} 
\end{equation}
\item We use $\ \widehat{\dd}\ $ to denote the Fourier transform of a function, or the symbol of a constant coefficient linear differential operator. 
\end{enumerate}

 We consider a general first-order constant coefficient linear differential operator
 $\ms{D}^*:C_c^\infty(\R^d, \S^n) \to C_c^\infty(\R^d, \R^{n \t k})$ with $\ms{D}^*(I) = 0$. By Fourier transform, it can be characterized by 
\begin{equation} \label{eq:symdre}
    \ms{D}^* (\Phi)(x) = \int_{\R^d} \widehat{\ms{D}^*}(\xi)\big[\h{\Phi}(\xi)\big] e^{i \xi \dd x}\, \rd \xi \,, \q \Phi \in C_c^\infty(\R^d, \S^n)\,,
\end{equation}
where $\h{\Phi}(\xi) = (2 \pi)^{-d} \int_{\R^d} \Phi(x) e^{-i\xi \dd x} \,\rd x$ is the Fourier transform of $\Phi$ and $\widehat{\ms{D}^*}(\xi): \R^d \to 
\L(\S^n,\R^{n \t k})$ is the symbol of $\ms{D}^*$ satisfying that for any $A \in \S^n$ and $B \in \R^{n \t k}$, $B \dd \widehat{\ms{D}^*}(\xi)(A)$ is a first-order polynomial in $\xi$. We write $\widehat{\ms{D}^*}(\xi)$ as the sum of its homogeneous components: $\widehat{\ms{D}^*}(\xi) = \widehat{\ms{D}^*_0} + \widehat{\ms{D}^*_1}(\xi)$,  where $\widehat{\ms{D}^*_0}$ and $\widehat{\ms{D}^*_1}(\xi)$ are homogeneous of degree $0$ and $1$, respectively. 
Then, recalling that the Fourier transform of $I$ is $\d_0 I$,
it is easy to see that the condition $\ms{D}^*(I) = 0$ is equivalent to
$\widehat{\ms{D}^*}(0)(I) = \widehat{\ms{D}_0^*}(I) = 0$.

By abuse of notation, we also define $\ms{D}^*\Phi$ for functions $\Phi (t,x)$ on $\R^{1 + d}$ by acting $\ms{D}^*$ on the spatial variable $x$. Moreover, we define the operator $\ms{D}$ as the adjoint operator of $ - \ms{D}^*$ in the sense of distribution. For instance, if $\ms{D}^* = \na$, then $\ms{D} = \ddiv$; see Section \ref{sec:example_model} for more examples. The homogeneous parts of degree $0$ and $1$ of $\ms{D}$ are denoted by $\ms{D}_0$ and $\ms{D}_1$ with the associated symbols $\widehat{\ms{D}_0}$ and $\widehat{\ms{D}_1}$, respectively. 

\begin{definition} \label{def:abs_conti_eq}
A measure $\ms{G} \in \M(Q_a^b, \S^n)$ connects $\ms{G}_a, \ms{G}_b \in \M(\Omega, \S^n_+)$ over the time interval $[a,b]$, if there exists $\ms{(q, R)} \in \M(Q_a^b, \R^{n \t k}  \t \MM^n)$ satisfying the following general \emph{matrix-valued continuity equation}:
\begin{equation}\label{eq:weak_ctneq}
        \int_{Q_a^b} \p_t\Phi \dd \rd \ms{G} + \ms{D}^* \Phi  \dd  \rd \ms{q}  +  \Phi   \dd  \rd \ms{R} = \int_{\Omega}  \Phi_b  \dd \rd \ms{G}_b -  \int_{\Omega}  \Phi_a \dd \rd \ms{G}_a\,,\q \forall \Phi \in C^1(Q_a^b,\S^n)\,.
\end{equation} 
The measures $\ms{G}_a$ and $\ms{G}_b$ are referred to as the initial and final distributions of $\ms{G}$, respectively.  Moreover,  we denote by $\ce([a,b];\ms{G}_a, \ms{G}_b)$ the set of the measures $\ms{(G,q, R)} \in \M(Q_a^b,\xx)$ satisfying \eqref{eq:weak_ctneq}.
\end{definition}

\begin{remark}
It is easy to derive the distributional equation of \eqref{eq:weak_ctneq}: 
\begin{equation} \label{eq:abs_ctn_eq}
    \p_t \ms{G} + \ms{D}\ms{q} = \ms{R}^{{\rm sym}}\,,
\end{equation} 
with the measure $\ms{q}$ satisfying a homogeneous boundary condition on $\p \Omega$. Indeed, assume that $\ms{q}$ admits a smooth density $q$ with respect to the Lebesgue measure. Then, the Stokes' theorem gives 
\begin{equation*}
    \int_\Omega \ms{D} q \dd \Phi + q \dd \ms{D}^* \Phi \, \rd x = \int_{\p \Omega} q \dd \widehat{\ms{D}_1^*}(-i \n)(\Phi)  \,  \rd x = \int_{\p \Omega} \widehat{\ms{D}_1}(-i \n)(q) \dd \Phi  \, \rd x\,,  \q \forall \Phi \in C^1(\Omega, \S^n)\,.
\end{equation*}  
It follows that the boundary condition $\widehat{\ms{D}_1}(-i \n)(q) = 0$ holds for 
$\ms{q}$ satisfying \eqref{eq:weak_ctneq}.  In the case of $\ms{D} = \ddiv$, we see that $\widehat{\ms{D}_1}(-i \n)(q) = 0$ is the familiar no-flux boundary condition $\nu \dd q = 0$. Moreover, we shall see very soon that under very mild conditions, the temporal boundary condition for $\ms{G}$ is redundant; see Remark \ref{rem1}. 

\end{remark}

\begin{remark}
We give an intuitive interpretation of \eqref{eq:abs_ctn_eq} as a continuity equation. Recall the homogeneous parts $\ms{D}_0$ and $\ms{D}_1$ of $\ms{D}$ with $\ms{D}_0 \in \mc{L}(\R^{n \t k},\S^n)$
and $\ms{D}_1$ vanishing when acting on constant functions.  It allows us to split $\ms{D}\ms{q}$ into two parts: $\ms{D}_0\ms{q}$ and $\ms{D}_1\ms{q}$, where $\ms{D}_0\ms{q}$ and $\ms{D}_1\ms{q}$ describe the mass transportation between components of $\ms{G}$ and the transportation in space, respectively. Moreover, the condition $\ms{D}^*(I) = 0$ can be regarded as a \emph{conservativity condition} in the sense that if $\ms{R} = 0$, then $\tr\ms{G}_t(\Omega) = \tr \ms{G}_0(\Omega)$ for any $t$; see Proposition \ref{prop:disinte_path}.
\end{remark}


The following elementary lemma gives the absolute continuity of the time marginal of $\ms{G}$.

\begin{lemma} \label{lem:abs_time}
    Let $\ms{(G,q,R)}  \in \ce([a,b];\ms{G}_a,\ms{G}_b)$ with $\ms{G}_a, \ms{G}_b \in \M(\Omega, \S^n_+)$. It holds that $\pi^t_\# \ms{G} \in \M([a,b],\S^n)$ has the distributional derivative $(\pi_\#^t \ms{R})^{\sym} \in \M([a,b],\S^n)$ in $t$.  If, further, $\ms{G} \in \M(Q_a^b,\S_+^d)$, then $\pi^t_\# |\ms{G}| \ll \rd t$. 
\end{lemma}

\begin{proof}
It suffices to consider $[a,b] = [0,1]$. By \eqref{eq:weak_ctneq} with test functions $\Phi(t,x) = \phi(t) \in C_c^1((0,1),\S^n)$, we have 
         \begin{equation} \label{auxeq:time_der}
            \int_0^1 \p_t \phi \dd \rd \pi^t_\# \ms{G} + \phi \dd \rd \pi_\#^t \ms{R} = 0\,,
        \end{equation}
which implies that $(\pi_\#^t \ms{R})^{\sym}$ is the distributional derivative of $\pi_\#^t \ms{G}$. Note that $\pi^t_\# \ms{G}$ and $\pi^t_\# \ms{R}$ are Radon measures (since every finite Borel measure on $[0,1]$ is regular). There exists a matrix-valued bounded variation function $M(t)$ that generates the Radon measure $\pi^t_\# \ms{R}$ \cite[Theorem 3.29]{folland1999real}. It follows from 
\eqref{auxeq:time_der} that
    \begin{equation}  \label{auxeq_3}
        \rd \pi^t_\# \ms{G} = (M(t)^{\sym} + C)\, \rd t\,,
    \end{equation}
for some $C \in \S^n$ \cite[Theorem 3.36]{folland1999real}. If $\ms{G} \in \M(Q,\S_+^d)$, then \eqref{auxeq_3} and \eqref{eq:trtotal} readily give $\tr \pi^t_\# \ms{G} \sim |\pi^t_\# \ms{G}| \ll \rd t$, which further yields $\pi^t_\# |\ms{G}| \ll \rd t$ by noting $\tr \pi^t_\# \ms{G} = \pi^t_\# \tr \ms{G} \sim \pi^t_\# |\ms{G}|$. 
\end{proof}

\medskip

\noindent \textbf{Weighted Wasserstein-Bures distance.} We are now ready to define a class of distances on $\M(\Omega,\S^n_+)$ by minimizing the action functional $\jj_{\Lad,Q}(\mu)$ over the solutions to the continuity equation \eqref{eq:weak_ctneq}.

\begin{definition}\label{def:metric}
    The \emph{weighted Wasserstein-Bures distance} between $\ms{G}_0, \ms{G}_1 \in \M(\Omega,\S_+^n)$  
     is defined by 
   \begin{equation} \label{eq:distance}
      {\rm WB}^2_{\Lad}(\ms{G}_0,\ms{G}_1) = \inf_{\mu \in \ce([0,1];\ms{G}_0,\ms{G}_1)} \jj_{\Lad,Q}(\mu).  \tag{$\mathcal{P}$}
   \end{equation}
\end{definition}

We remark that the quantity $\jj_{\Lad,Q}(\mu)$ can be understood as the energy of the measure $\mu \in \ce([0,1];\ms{G}_0,\ms{G}_1)$. The following a priori estimate shows that 
$\ce([0,1];\ms{G}_0,\ms{G}_1)$ is nonempty and $\rm{WB}_{\Lad}(\ms{G}_0,\ms{G}_1)$ is always finite, which means that the problem \eqref{eq:distance} is well-defined.

\begin{lemma} \label{lem:com_kb_and_b}
Given $\ms{G}_0,\ms{G}_1 \in \M(\Omega, \S^n_+)$, let $\lad \in \M(\Omega,\R_+)$ be a reference measure such that $|\ms{G}_0|, |\ms{G}_1| \ll \lad$. Then there exists $\mu = \ms{(G,0,R)} \in \ce([0,1];\ms{G}_0,\ms{G}_1)$ with finite $\jj_{\Lad,Q}(\mu)$.  
Moreover, it holds that
\begin{equation} \label{eq:basic_bound}
{\rm WB}^2_{\Lad}(\ms{G}_0,\ms{G}_1) \le 
{\rm WB}^2_{(0,\Lad_2)}(\ms{G}_0,\ms{G}_1) \le 
2  \big\lVert \Lad_2^{-1} \big\lVert_\ff^2 \int_{\Omega} \big\lVert\sqrt {G_{1,\lad}} - \sqrt{G_{0,\lad}} \big\lVert_\ff^2  \ \rd \lad  \,,
\end{equation}
where $G_{0,\lad}$ and $G_{1,\lad}$ are densities of $\ms{G}_0$ and $\ms{G}_1$ with respect to $\lad$.
\end{lemma}

\begin{proof} 
We omit the subscript $\lad$ of $G_{0,\lad}$ and $G_{1,\lad}$ for simplicity. We define  measures
    \begin{equation*}
        \ms{G} := \left(\sqrt{G_{0}} + t \Big(\sqrt {G_{1}} - \sqrt{G_{0}} \Big) \right)^2  \rd t \otimes \lad \in \M(Q, \S^n_+)\,, 
    \end{equation*}
    and
    \begin{equation*}
        \ms{R} := 2 \left(\sqrt{G_{0}} + t \left(\sqrt {G_{1}} - \sqrt{G_{0}} \right) \right) \left(\sqrt {G_{1}} - \sqrt{G_{0}}\right)  \rd t \otimes   \lad  \subset \M(Q, \MM^n)\,,
    \end{equation*}
which satisfies $\mu = \ms{(G,0,R)} \in \ce([0,1];\ms{G}_0,\ms{G}_1)$ and $\ran \big( \frac{\rd \ms{R}}{\rd t \otimes \lad} \big) \subset  \ran \big( \frac{\rd \ms{G}}{\rd t \otimes \lad} \big)$ for $\rd t \otimes \lad$-a.e..  Moreover, we note
    \begin{align*}
        \ran\left(\sqrt {G_1} - \sqrt{G_0}\right)\subset\ran\left(\sqrt{G_0} + t \left(\sqrt {G_1} - \sqrt{G_0}\right)\right)\,,\q  t \in (0,1)\,,
  \end{align*}
from the relation: $ \ker\big(\sqrt{G_0} + t (\sqrt {G_1} - \sqrt{G_0})\big) = \ker \big(\sqrt{G_0}\big) \cap \ker 
\big(\sqrt{G_1}\big)\subset \ker\big(\sqrt {G_1} - \sqrt{G_0}\big)$. Then, we compute 
  \begin{equation} \label{auxeqscaling} 
    \jj_{\Lad,Q}(\mu) = 2 \int_{\Omega} \Big\lVert \Big(\sqrt {G_1} - \sqrt{G_0}\Big) \Lad_2^{-1} \Big\lVert_\ff^2\ \rd \lad\,, 
  \end{equation} 
  for $\mu$ defined above.  The proof is completed by the submultiplicativity of the Frobenius norm. 
\end{proof}

\begin{remark} \label{rem:bures} 
The proof of Lemma \ref{lem:com_kb_and_b} uses $\ran (\Lad_2) = \R^n$ from the assumption $\Lad_2 \in \S^n_{++}$ we made before \eqref{def:costmeasure}. If we only assume $\Lad_2 \in \S^n_{+}$, the distance ${\rm WB}_\Lad$ is only well-defined (i.e., finite) on a subset of $\M(\Omega,\S^n_+)$.
\end{remark}

\begin{remark} \label{rem:bures_2} 
${\rm WB}_{(0,\Lad_2)}$ is the matricial Hellinger distance $d_H$ in \cite[Definition 4.1]{monsaingeon2020schr}, up to a transformation. Indeed, recalling Lemma \ref{lem:fcabs}, we have that if $\Lad_1 = 0$, then $\ms{q}$ must be zero and \eqref{eq:distance} reduces to 
\begin{equation} \label{eq:distance_bures}
    {\rm WB}_{(0,\Lad_2)}^2(\ms{G}_0,\ms{G}_1) = \inf\{\jj_{(0,\Lad_2),Q}(\mu)\,;\  \mu = \ms{(G,0,R)} \in \ce([0,1];\ms{G}_0,\ms{G}_1)\}. 
\end{equation}
For a given $S \in \S^n_{++}$, we introduce a linear map $g_{S}(A) := S A S : \S^n_{+} \to  \S^n_{+}$ with the inverse $g_{S^{-1}}$. It is easy to see that $\ms{(G,0,R)} \in \ce([0,1];\ms{G}_0,\ms{G}_1)$ if and only if $(g_{\Lad_2^{-1}}(\ms{G}),0, g_{\Lad_2^{-1}}(\ms{R})) \in \ce([0,1];g_{\Lad_2^{-1}}(\ms{G}_0),g_{\Lad_2^{-1}}(\ms{G}_1))$, and there holds 
$\jj_{(0,\Lad_2),Q}(\ms{(G,0,R)}) = \jj_{(0,I),Q}(g_{\Lad_2^{-1}}(\ms{G}),0, g_{\Lad_2^{-1}}(\ms{R}))$. Therefore, we have
$$
{\rm WB}_{(0,\Lad_2)}(\ms{G}_0,\ms{G}_1) =  {\rm WB}_{(0,I)}(g_{\Lad_2^{-1}}(\ms{G}_0),g_{\Lad_2^{-1}}(\ms{G}_1))\,.
$$
From \cite[Definition 4.1]{monsaingeon2020schr} and Theorem \ref{thm:dual_minmax} below, one can see that ${\rm WB}_{(0,I)}$
is nothing else than the convex formulation of the Hellinger distance $d_H$, up to a constant. We refer the readers to 
\cite[Lemma 4.3 and Theorem 2]{monsaingeon2020schr} for the properties of the Hellinger distance and its relation with the Bures-Wasserstein distance on $\S^n_+$ \cite{bhatia2019bures}. 
\end{remark}

\medskip

\noindent \textbf{A priori estimate.}
Thanks to Lemma \ref{lem:com_kb_and_b}, the optimization \eqref{eq:distance}
can be equivalently taken over the following set: 
\begin{equation*}
    \ce_\infty ([0,1];\ms{G}_0,\ms{G}_1): =  \ce([0,1];\ms{G}_0,\ms{G}_1) \bigcap \{\mu \in \M(Q,\xx)\,; \jj_{\Lad,Q}(\mu) < +\infty\}\,.
\end{equation*}
Before we proceed, we give some auxiliary results. First, 
we introduce
\begin{equation} \label{def:jalpha_conj}
    \jj_{\Lad,\mathcal{X}}^*(\ms{G},u,W) := \frac{1}{2} \norm{(u \Lad_1 ,W \Lad_2)}^2_{L^2_{\ms{G}}(\mathcal{X})}\q \text{on}\  \M(\mathcal{X},\S_+^n) \t C(\mathcal{X}, \R^{n \t k} \t \MM^n)\,,
\end{equation} 
where $\norm{\dd}_{L^2_{\ms{G}}(\mc{X})}$ is defined by \eqref{eq2}.
By an argument similar to the one for Lemma \ref{lem: conj_cvx} below, we have that 
the conjugate function \eqref{def:conjugate} of $\jj^*_{\Lad,\mathcal{X}}(\ms{G},u,W)$ with respect to $(u,W)$ is exactly $\jj_{\Lad,\mathcal{X}}(\ms{G},\ms{q},\ms{R})$. Moreover, there holds
\begin{equation} \label{eq:fenconj_cost}
    \jj_{\Lad,\mathcal{X}}(\ms{G},\ms{q},\ms{R}) = \sup_{(u,W) \in L^\infty_{|\ms{(G,q,R)}|}(\mathcal{X}, \R^{n \t k} \t \MM^n) } \l (\ms{q},\ms{R}), (u,W) \r_{\mathcal{X}} - \jj_{\Lad,\mathcal{X}}^*(\ms{G},u,W)\,.
\end{equation}
Since $\jj_{\Lad,\mathcal{X}}(\ms{G}, \ms{q},\ms{R})$ and $\jj_{\Lad,\mathcal{X}}^*(\ms{G}, u, W)$ are homogeneous of degree $2$ in $(\ms{q},\ms{R})$ and $(u,W)$, respectively, by \eqref{eq:fenconj_cost}, it holds that for $\ms{(G,q,R)} \in \M(\mathcal{X},\xx)$ and $(u,W) \in L^\infty_{|\ms{(G,q,R)}|}(\mathcal{X}, \R^{n \t k} \t \MM^n)$, 
\begin{align} \label{eq: ana_Cauchy}
    \l (\ms{q},\ms{R}), (u,W)\r_{\mathcal{X}} \le  \gamma^{-2} \jj_{\Lad,\mathcal{X}}(\ms{G}, \ms{q}, \ms{R}) + \gamma^2\jj_{\Lad,\mathcal{X}}^*(\ms{G},u, W)\,, \q \forall \gamma > 0\,.
\end{align}
We minimize the right-hand side of \eqref{eq: ana_Cauchy}
with respect to $\gamma$ and obtain 
\begin{align}\label{eq: ana_Cauchy_2}
    \l (\ms{q},\ms{R}), (u,W)\r_{\mathcal{X}} \le 2 \sqrt{\jj_{\Lad,\mathcal{X}}(\ms{G}, \ms{q}, \ms{R}) \jj_{\Lad,\mathcal{X}}^*(\ms{G},u, W)}\,,
\end{align}
where we have used non-negativity of $\jj_{\Lad,\mathcal{X}}$ and $\jj_{\Lad,\mathcal{X}}^*$. 

Second, we observe from formulas \eqref{eq:expre_J} and \eqref{def:costmeasure} and Lemmas \ref{prop:measmatfuns} and \ref{lem:fcabs} that for $\mu = (\ms{G},\ms{q},\ms{R}) \in \M(\mc{X},\xx)$ with $\jj_{\Lad, \mc{X}}(\mu) < +\infty$, the functions $G_\lad^\dag q_\lad \Lad_1^\dag$ and $G_\lad^\dag R_\lad \Lad_2^{-1}$ are well-defined, Borel measurable, and independent of the reference measure $\lad$ (hence we omit the subscript $\lad$ in the sequel for simplicity), and there holds 
\begin{align} \label{rep:enerycost}
    \jj_{\Lad,\mc{X}}(\mu) = \frac{1}{2} \norm{G^\dag q \Lad_1^\dag}^2_{L^2_{\ms{G}}(\mc{X})} + \frac{1}{2} \norm{G^\dag R \Lad_2^{-1}}^2_{L^2_{\ms{G}}(\mc{X})} < +\infty \,.
\end{align}

We now give useful a priori bounds for measures $\ms{q}$ and $\ms{R}$.

\begin{lemma} \label{lem:priorimu}
    For $\mu = (\ms{G},\ms{q},\ms{R}) \in \M(\mc{X},\xx)$ with $\jj_{\Lad,\mc{X}}(\mu) < +\infty$, it holds that for $E \in \mathscr{B}(\mc{X})$, 
    \begin{align} \label{priest_qR}
    |\ms{q}|(E) \le \sqrt{\tr \ms{G} (E)}\, \norm{\Lad_1}_\ff \norm{G^\dag q \Lad_1^\dag}_{L^2_{\ms{G}}(E)}\,,\q   |\ms{R}|(E) \le \sqrt{\tr \ms{G} (E)}\, \norm{\Lad_2}_\ff \norm{G^\dag R \Lad_2^{-1}}_{L^2_{\ms{G}}(E)}\,.
    \end{align}
\end{lemma}
\begin{proof}  
Recall that there exist bounded measurable functions  
$\sigma_{q}$ and $\sigma_{R}$
with $ \norm{\sigma_{q}}_\ff = \norm{\sigma_{R}}_\ff = 1$ such that $\rd \ms{q} = \sigma_{q}\, \rd |\ms{q}|$ and $\rd \ms{R} = \sigma_{R} \,\rd |\ms{R}|$.  Taking $\ms{R} = 0$ and $(u,W) = (\chi_E \sigma_q, 0)$ in \eqref{eq: ana_Cauchy_2} for $E \in \mathscr{B}(\mc{X})$, we obtain
\begin{align*}
    |\ms{q}|(E) =  \int_E u \dd \rd \ms{q} 
         \le 2 \sqrt{\jj_{\Lad,E}(\ms{G}, \ms{q}, 0)\jj_{\Lad,E}^{*}(\ms{G},u, 0)} \le  \sqrt{\tr \ms{G} (E) \norm{\Lad_1}_\ff^2} \norm{G^\dag q \Lad_1^\dag}_{L^2_{\ms{G}}(E)} \,,
\end{align*}
by \eqref{rep:enerycost} and the following estimate  derived from \eqref{def:jalpha_conj} and \eqref{eq:useineq},
    \begin{align*}
        \jj_{\Lad,E}^*(\ms{G},u, W) \le \frac{1}{2} \tr \ms{G}(E) \norm{\Lad_1}_\ff^2\,.
    \end{align*}
    Similarly, by taking $\ms{q} = 0$ and  $(u,W) = (0, \chi_E \sigma_R)$ in \eqref{eq: ana_Cauchy_2}, we obtain the estimate for $\ms{R}$ in \eqref{priest_qR}.
\end{proof}

With the help of the above lemma, the following proposition holds.

\begin{proposition}   \label{prop:disinte_path}  
Let $\mu = \ms{(G,q,R)}\in \ce_\infty ([0,1];\ms{G}_0,\ms{G}_1)$ with $\ms{G}_0,\ms{G}_1 \in \M(\Omega, \S^n_+)$. Then, 
\begin{enumerate}[(i)]
    \item  $\ms{G} \in \M(Q,\S^n_+)$ and $\pi_\#^t |\ms{G}| \ll \rd t$. Moreover, $\mu$ can be  disintegrated as:
\begin{equation} \label{eq:disintegra}
    \mu = \int_0^1 \d_t \otimes (\ms{G}_t, \ms{q}_t, \ms{R}_t)\, \rd t\,,
\end{equation} 
where $(\ms{G}_t, \ms{q}_t, \ms{R}_t) \in \mc{M}(\Omega,\xx)$ for $\rd t$-a.e.\,$t \in [0,1]$.
    \item There exists a weak* continuous curve $\big\{\w{\ms{G}}\big\}_{t \in [0,1]}$ in $\M(\Omega, \S^n_+)$ such that $\ms{G}_t = \w{\ms{G}}_t$ for a.e.\,$t \in [0,1]$ and, for any interval $[t_0,t_1] \subset [0,1]$,  it holds that 
    \begin{equation}\label{eq:time_a_b}
        \int_{Q_{t_0}^{t_1}} \p_t\Phi \dd \rd \ms{G} + \ms{D}^* \Phi \dd  \rd \ms{q} +  \Phi  \dd  \rd \ms{R} = \int_{\Omega}  \Phi_{t_1} \dd \rd \w{\ms{G}}_{t_1} -  \int_{\Omega}  \Phi_{t_0} \dd \rd \w{\ms{G}}_{t_0}\,, \q \forall \Phi \in C^1(Q_{t_0}^{t_1},\S^n)\,.
    \end{equation}
    Moreover, there holds, for some $C > 0$,
    \begin{align} \label{priest_G}
        \tr \w{\ms{G}}_t(\Omega) \le C \left(\tr \ms{G}_0 (\Omega) + \norm{G^\dag R \Lad_2^{-1}}^2_{L_{\ms{G}}^2(Q)}\norm{\Lad_2}_\ff^2\right),\q \forall t \in [0,1]\,.
    \end{align}
\end{enumerate}
\end{proposition}

\begin{remark} \label{rem1}
By above proposition, we can identify a measure $\mu = (\ms{G}, \ms{q},\ms{R})\in  \ce_\infty([0,1];\ms{G}_0,\ms{G}_1)$ with a  family of measures $\{\mu_t = (\ms{G}_t, \ms{q}_t,\ms{R}_t)\}_{t \in [0,1]}$ in $\M(\Omega, \xx)$ via the disintegration \eqref{eq:disintegra}, where $\ms{G}_t$ is weak* continuous. As an easy consequence,
 the initial and final distributions $\ms{G}_0$ and $\ms{G}_1$ in Definition \ref{def:abs_conti_eq} are redundant, as 
 they can be recovered by taking a limit of $\{\ms{G}_t\}_{t \in (0,1)}$. 
\end{remark}

\begin{proof}
(i) First, note from \cite[Theorem 5.3.1]{ambrosio2008gradient} that $\mu$ can be disintegrated with respect to $\nu = \pi^t_\# |\mu|$ as $\mu = \int_0^1 \d_t \otimes \mu_t\, \rd \nu$, where $\mu_t \in \M(\Omega, \xx)$ for $\nu$-a.e.\,$t \in [0,1]$. Then,  by Lemmas \ref{lem:fcabs} and \ref{lem:abs_time}, we have $\ms{G} \in \M(Q,\S^n_+)$ and $\nu \ll \pi^t_\# | \ms{G}| \ll \rd t$ on $[0,1]$,  which allows us to define $\w{\mu}_t : = \mu_t \frac{\rd \nu}{\rd t}$ and disintegrate $\mu$ as $\mu = \int_0^1 \d_t \otimes \w{\mu}_t\, \rd t$. 

\medskip
\noindent
(ii) Consider test functions $\Phi = a(t)\Psi(x)$ in \eqref{eq:weak_ctneq} with $a(t) \in C_c^1((0,1),\R)$ and $\Psi(x) \in C^1(\Omega,\S^n)$. Then, by \eqref{eq:disintegra}, $\int_{\Omega} \Psi \dd \rd \ms{G}_t$ is absolutely continuous in $t$ with the weak derivative:
 \begin{equation}\label{weak_g_inpro}
       \p_t \l \ms{G}_t, \Psi \r_{\Omega} = \l \ms{q}_t , \ms{D}^* \Psi \r_{\Omega} + \l \ms{R}_t, \Psi \r_{\Omega} \,.
    \end{equation}
Letting $\Psi = I$ in \eqref{weak_g_inpro}, we obtain $\p_t \tr \ms{G}_t(\Omega) = \tr \ms{R}^{\sym}_t(\Omega)$ a.e.\,by $\ms{D}^*(I) = 0$,  which implies that there exists a nonnegative function $m(t) \in C([0,1],\R)$ such that $\tr \ms{G}_t (\Omega) = m(t)$ a.e.\,on $[0,1]$ and
    \begin{equation} \label{auxest_forG0}
        m(t) - m(s) = \int_s^t \tr \ms{R}^{\sym}_{\tau}(\Omega)\, \rd \tau\,, \q \forall 0\le s \le t \le 1\,.
    \end{equation}
    By Lemma \ref{lem:priorimu}, it follows from \eqref{auxest_forG0} that, from some $C > 0$, 
    \begin{equation} \label{auxest_forG}
        |m(t) - m(s)| \le     
        C |\ms{R}|(Q) \le C \sqrt{\tr \ms{G} (Q)} \norm{\Lad_2}_\ff \norm{G^\dag R \Lad_2^{-1}}_{L^2_{\ms{G}}(Q)}\,.
    \end{equation}
    We choose $t_0$ such that $m(t_0) = \max_{t \in [0,1]} m(t)$. Then \eqref{auxest_forG} implies  
    \begin{equation*}
        m(t_0) \le m(0) + C \sqrt{m(t_0)} \norm{\Lad_2}_\ff \norm{G^\dag R \Lad_2^{-1}}_{L^2_{\ms{G}}(Q)}\,,
    \end{equation*}
    which further gives, by an elementary calculation,  
    \begin{align} \label{eq:auxpri_bound}
       \Big(m(t_0)^{1/2} - \frac{C}{2}\norm{G^\dag R \Lad_2^{-1}}_{L_{\ms{G}}^2(Q)} \norm{\Lad_2}_\ff\Big)^2 \le m(0) + \frac{C^2}{4} \norm{G^\dag R \Lad_2^{-1}}^2_{L_{\ms{G}}^2(Q)} \norm{\Lad_2}_\ff^2\,.
    \end{align}
    Then we have 
    \begin{equation}\label{eq:auxpri_bound_2}
        m(t) \le C \big(m(0) + \norm{G^\dag R \Lad_2^{-1}}^2_{L_{\ms{G}}^2(Q)}\norm{\Lad_2}_\ff^2\big)\,.      
    \end{equation}

With the above estimates, the existence of a weak* continuous representative of $\ms{G}_t$ and the formula \eqref{eq:time_a_b} can be proved similarly to \cite[Lemma 8.1.2]{ambrosio2008gradient}. 
We sketch the argument for completeness. 
By \eqref{priest_qR} and \eqref{eq:auxpri_bound_2}, as well as \eqref{weak_g_inpro}, there exists a subset $E \in [0,1]$ of Lebesgue measure zero such that $\tr \ms{G}_t (\Omega) = m(t)$ on $[0,1]\backslash E$, and there holds, for any $t,s \in [0,1]\backslash E$ with  $s < t$ and $\Psi \in C^1(\Omega,\S^n)$,
    \begin{align}  \label{eq:acgqr_2} 
        | \l \ms{G}_t, \Psi \r_{\Omega}   - \l \ms{G}_s, \Psi \r_{\Omega}| & \le C \norm{\Psi}_{1,\infty} \big(|\ms{q}|(Q_s^t) + |\ms{R}|(Q_s^t)\big)
        \notag \\ & \le C |t - s|^{1/2} \big(m(0) + \norm{G^\dag q \Lad_1^\dag}^2_{L_{\ms{G}}^2(Q)} \norm{\Lad_1}_\ff^2 +  \norm{G^\dag R \Lad_2^{-1}}^2_{L_{\ms{G}}^2(Q)}\norm{\Lad_2}_\ff^2\big) \norm{\Psi}_{1,\infty}\,.
    \end{align}
    The estimate \eqref{eq:acgqr_2} allows us to uniquely 
    extend $\{\ms{G}_t\}_{t \in [0,1]\backslash E}$ to a weak*  continuous curve $\{\w{\ms{G}}_t\}_{t \in [0,1]}$ in $C^1(\Omega,\S^n)^*$. Then, by the density of $C^1(\Omega,\S^n)$ in $C(\Omega,\S^n)$ and the boundedness \eqref{eq:auxpri_bound_2} of $\{\tr \w{\ms{G}}_t (\Omega)\}_{t \in [0,1]}$, the curve $\{\w{\ms{G}}_t\}_{t \in [0,1]}$ is also weak* continuous in $\M(\Omega, \S^n)$. The formula
    \eqref{eq:time_a_b} follows from taking test functions $\Phi_\ep(x,t) = \eta_\ep(t)\Phi(t,x)$ in \eqref{eq:weak_ctneq},  where $\Phi \in C^1(Q,\S^n)$ and $\eta_\ep \in C_c^\infty((t_0,t_1),\R)$ with $0 \le \eta_\ep \le 1$, $\lim_{\ep \to 0}\eta_\ep(t) = \chi_{(t_0,t_1)}(t)$ pointwisely, and $\lim_{\ep \to 0}\eta'_\ep = \d_{t_0} - \d_{t_1}$ in the distributional sense. Recalling $\tr \ms{G}_t (\Omega) = m(t)$ a.e., by the weak*  continuity of $\w{\ms{G}}_t$, we have $\tr \w{\ms{G}}_t = m(t)$. Then, the estimate \eqref{priest_G} follows from \eqref{eq:auxpri_bound_2}. 
\end{proof}

\medskip
\noindent \textbf{Time and space scaling.}
By writing $\jj_{\Lad,Q}(\mu) = \int_0^1\jj_{\Lad,\Omega}(\mu_t)\, \rd t$ for $\mu \in \ce_\infty([0,1];\ms{G}_0,\ms{G}_1)$, the following Lemma is a simple consequence of the change of variable.

\begin{lemma}\label{lemma:timescaling} 
Let $\mu \in \ce_\infty([0,1];\ms{G}_0,\ms{G}_1)$.  It holds that 
\begin{enumerate}
    \item Let $\ms{s}(t):[0,1] \to [a,b]$ be a strictly increasing absolutely continuous map with an absolutely continuous inverse: $\ms{t} = \ms{s}^{-1}$. Then
     $\w{\mu} := \int_a^b \d_s \otimes (\ms{G}_{\ms{t}(s)}, \ms{t}'(s) \ms{q}_{\ms{t}(s)},  \ms{t}'(s) \ms{R}_{\ms{t}(s)})\, \rd s \in \ce([a,b]; \ms{G}_0,\ms{G}_1)$. Moreover, we have 
    \begin{align}  \label{eq:timescinlemma}
        \int_0^1 \ms{t}'(\ms{s}(t)) \jj_{\Lad,\Omega}(\mu_t)\, \rd t = \int_a^b \jj_{\Lad,\Omega}(\w{\mu}_s) \,\rd s\,.
    \end{align} 
    \item  Let $T$ be a 
diffeomorphism on $\R^d$ mapping from $\Omega$ to $T(\Omega)$ and suppose that there exists $\TT_{\ms{D^*}}(x): \Omega \to \L(\R^{n \t k})$ such that for $\Phi \in C_c^\infty(\R^d, \S^n)$,
\begin{align} \label{def:tdlinear}
    \TT_{\ms{D^*}}[(\ms{D^*} \Phi)\circ T] := \ms{D^*} (\Phi\circ T)\,.
\end{align}
Then $\w{\mu} := \int_0^1 \d_t \otimes T_{\#} (\ms{G}_{t}, \TT_{\ms{D}} \ms{q}_{t},  \ms{R}_{t})\, \rd t \in \ce([0,1]; T_{\#}\ms{G}_0, T_{\#}\ms{G}_1)$ on $T(\Omega)$, where $T_{\#}(\dd)$ denotes the pushforward measure by $T$,  and $\TT_{\ms{D}}$ is the transpose of $\TT_{\ms{D}^*}$ defined via  $ (\TT_{\ms{D}} q) \dd p = q \dd (\TT_{\ms{D}^*} p)\,, \ \forall p,q \in \R^{n \t k}$. 
\end{enumerate}
\end{lemma}
\begin{remark} \label{rem:assumtd}
    The condition \eqref{def:tdlinear} is nontrivial and necessary for the second statement. Indeed, there holds
    \begin{align*}
        \ms{D^*} (\Phi\circ T) = \int_{\R^d} \widehat{\ms{D}^*}(\xi \dd \na T(x))\big[\h{\Phi}(\xi)\big]e^{i \xi \dd T(x)}\, \rd \xi \,,
    \end{align*}
    by Fourier transform, where $(\xi \dd \na T(x))_j = \xi \dd \p_j T(x)$. It follows that \eqref{def:tdlinear} is equivalent to a separation of variables: $\widehat{\ms{D}^*}(\xi \dd \na T(x)) = \TT_{\ms{D^*}}(x) \circ \widehat{\ms{D}^*}(\xi)$. A sufficient condition for \eqref{def:tdlinear} is that $\widehat{\ms{D}^*}$ is homogeneous of degree $0$, or homogeneous of degree $1$ with $T(x) = a x + b$ for $a \neq 0 \in \R$ and $b \in \R^d$, which is enough for our purposes. 
\end{remark}

\begin{remark} \label{rem;space}
We connect the weighted matrix $\Lad_1$ and the space scaling. 
Let us consider $\mu \in \ce_\infty([0,1]; \ms{G}_0, \ms{G}_1)$ and $\ms{D}^*$ be homogeneous of degree one for simplicity. Define $T(x) = a x: \Omega \to a \Omega$ and $\TT_{\ms{D}} = a I$. 
By Lemma \ref{lemma:timescaling}, we have $\w{\mu} := \int_0^1 \d_t \otimes T_{\#} (\ms{G}_{t}, a \ms{q}_{t},  \ms{R}_{t}) \, \rd t \in \ce_\infty([0,1]; T_{\#}\ms{G}_0, T_{\#}\ms{G}_1)$. Then, a direct computation gives
\begin{align*}
\jj_{\Lad, [0,1] \t a \Omega}(\w{\mu}) = \int_0^1 \jj_{(a^{-1} \Lad_1,\Lad_2), a \Omega}(T_{\#}(\ms{G}_t,\ms{q}_t,\ms{R}_t)) \,\rd t =  \int_0^1 \jj_{(a^{-1} \Lad_1,\Lad_2), \Omega}(\mu_t) \,\rd t = \jj_{(a^{-1} \Lad_1,\Lad_2), Q}(\mu)\,.
\end{align*}
\end{remark}

Using Lemma \ref{lemma:timescaling} with $\ms{s}(t) = (b - a) t + a : [0,1] \to [a,b]$, $b > a > 0$, we see that for $\mu \in \ce_\infty([0,1];\ms{G}_0,\ms{G}_1)$, there exists $\w{\mu} \in \ce_\infty([a,b];\ms{G}_0,\ms{G}_1)$ such that 
\begin{align*}
\int_0^1 \jj_{\Lad,\Omega}(\mu_t)\, \rd t = (b - a) \int_a^b \jj_{\Lad,\Omega}(\w{\mu}_t)\, \rd t\,,
\end{align*}
and vice versa, which gives the equivalent characterization of ${\rm WB}_{\Lad}$: 
\begin{align} \tag{$\mathcal{P}'$}\label{eq:change_time}
{\rm WB}^2_\Lad(\ms{G}_0,\ms{G}_1) = \inf_{\ce_\infty([a,b];\ms{G}_0,\ms{G}_1)} (b - a) \int_a^b  \jj_{\Lad,\Omega}(\mu_t) \,\rd t\,, \q \ms{G}_0,\ms{G}_1 \in \M(\Omega,\S_+^n)\,.
\end{align}

\medskip
\noindent \textbf{Compactness.}
We end the discussion of basic properties of $\ce_\infty([0,1]; \ms{G}_0, \ms{G}_1)$ with a compactness result.

\begin{proposition}\label{prop:compact}
Let $\mu^n = (\ms{G}^n,\ms{q}^n,\ms{R}^n) \in \ce_\infty([0,1]; \ms{G}^n_0, \ms{G}^n_1)$, $n \ge 1$, be a sequence of measures satisfying 
    \begin{equation} \label{assp_for_prop:1}
       m:= \sup_{n \in \NN} \tr (\ms{G}_0^n) < +\infty\,, \q  M: = \sup_{n \in \NN} \jj_{\Lad,Q}(\mu^n) < +\infty \,.
    \end{equation}
    Then there exists a subsequence, still denoted by $\mu^n$, and a measure $\mu = (\ms{G},\ms{q},\ms{R}) \in \ce_\infty([0,1]; \ms{G}_0, \ms{G}_1)$ such that for every $t\in [0,1]$, $\ms{G}^n_t$ weak* converges to $\ms{G}_t$ in $\M(\Omega,\S^{n})$, and $\ms{(q^n,R^n)}$ weak* converges to $\ms{(q,R)}$ in $\M(Q,\R^{n \t k} \t \MM^n)$. Moreover, it holds that, for $0\le a < b \le 1$, 
\begin{align}  \label{eq:lsc_jalpha}
    \jj_{\Lad,Q_a^b}(\mu) \le \liminf_{n \to \infty} \jj_{\Lad,Q_a^b}(\mu^n)\,.
    \end{align}
\end{proposition}
\begin{proof}
    By \eqref{assp_for_prop:1}, up to a subsequence, we can let $\ms{G}^n_0$ weak* converge to some $\ms{G}_0 \in \M(\Omega,\S^n_+)$. It is also clear from a priori estimates \eqref{priest_qR} and \eqref{priest_G}, as well as the assumption \eqref{assp_for_prop:1}, that 
 $\{\mu^n\}_{n \in \NN}$ is bounded in $\M(Q,\xx)$. Hence, 
    there exists a subsequence of $\{\mu^n\}_{n \in \NN}$, still indexed by $n$, weak* converging to some $\mu \in \M(Q,\xx)$. We next prove that the restriction of $\mu^n$ on $Q_a^b$, i.e., $\mu^n|_{Q_a^b}$, weak* converges to $\mu|_{Q_a^b}$ in $\M(Q_a^b,\xx)$ for any $ 0 \le a \le b \le 1$. For this, again by \eqref{priest_qR} and \eqref{priest_G}, we have, for some $C > 0$,
    \begin{equation} \label{eq:uni_inte_mu}
        |\mu^n|([t_0,t_1] \t \Omega) \le C|t_1 - t_0|^{1/2}\,, \q \forall 0 \le t_0 \le t_1 \le 1\,,
    \end{equation}
    which also holds for $\mu$. Let $\eta(t)$ be a smooth function, compactly supported in $[a,b]$, with $|\eta(t)| \le 1$ and $\eta = 1$ on $[a+\ep, b - \ep]$ for some small $\ep$. Then, for any $\Xi \in C(Q_a^b,\xx)$, we define $\w{\Xi}(t,x) = \eta (t) \Xi(t,x) \in C(Q,\xx)$. The following estimate readily follows from the properties of $\eta$ and the estimate \eqref{eq:uni_inte_mu}: 
    \begin{equation*}
        \big|\l \mu^n, \Xi\r_{Q_a^b} - \l \mu, \Xi\r_{Q_a^b}\big| \le  \big| \big\l  \mu^n, \w{\Xi}\big\r_{Q} - \big\l  \mu, \w{\Xi}\big\r_{Q} \big| + C \ep^{1/2}\,.
    \end{equation*}
   Since $\mu^n$ weak* converges to $\mu$ in $\M(Q,\xx)$ and $\ep$ is arbitrary, we have $\big|\l \mu^n, \Xi \r_{Q_a^b} - \l  \mu, \Xi \r_{Q_a^b}\big| \to 0$ as $n \to \infty$ for $\Xi \in C(Q_a^b,\xx)$. 
 Then, \eqref{eq:lsc_jalpha} follows from the lower semicontinuity of $\jj_{\Lad,Q_a^b}(\mu)$. We now show the weak* convergence of $\ms{G}^n_t$ for every $t\in [0,1]$. We note,  by taking $\Phi(s,x) = \chi_{[0,t]}(s) \Psi(x)$ in \eqref{eq:time_a_b} with $\Psi(x)\in C^1(\Omega, \S^n)$,
    \begin{equation*}
        \int_0^t \Big( \int_\Omega \ms{D}^* \Psi \dd \rd \ms{q}_s^n   +  \int_\Omega \Psi  \dd  \rd \ms{R}_s^n \Big) \rd s = \int_{\Omega}  \Psi  \dd \rd \ms{G}^n_{t} -  \int_{\Omega}  \Psi \dd \rd \ms{G}^n_{0}\,, \q \forall \Psi \in C^1(\Omega,\S^n)\,.
    \end{equation*} 
    Then, using the weak* convergences of $\ms{G}^n_0$ in $\M(\Omega,\S^n)$ and   
    $(\ms{q}^n,\ms{R}^n)|_{Q_0^t}$ in $\M(Q_0^t,\R^{n \t k}\t \MM^n)$, we get the convergence of $ \l \ms{G}^n_{t}, \Psi \r_{\Omega}$ as $n \to \infty$. The proof is completed  by the density of $C^1(\Omega,\S^n)$ in $C(\Omega,\S^n)$ and the uniform boundedness of $\tr \ms{G}^n_t (\Omega)$ with respect to $n$ from \eqref{priest_G}.
\end{proof}

\section{Properties of weighted Wasserstein-Bures metrics}   \label{sec:existenandopt} 

This section is devoted to the investigation of the convex optimization problem \eqref{eq:distance}. We shall first  show the existence of the minimizer and derive the corresponding optimality condition. We then explore its primal-dual formulations
in more detail, which will lead to a Riemannian interpretation 
of ${\rm WB}_{\Lad}$ in Section \ref{sec:Riema_iter}. Finally, we consider the dependence of ${\rm WB}_{\Lad}$ on the weighted matrix $\Lad$. 

\medskip 

\noindent \textbf{Existence of minimizer and optimality condition.}
For our purpose, let us first define the  Lagrangian of \eqref{eq:distance} with the multiplier $\Phi \in C^1(Q, \S^n)$:
\begin{align*}
    \L(\mu, \Phi) := \jj_{\Lad,Q}(\mu) - \l \mu, (\p_t \Phi, \ms{D}^*  \Phi, \Phi) \r_Q + \l \ms{G}_1, \Phi_1 \r_\Omega - \l \ms{G}_0, \Phi_0 \r_\Omega\,,
\end{align*}
which allows us to write 
\begin{equation*} 
    {\rm WB}^2_\Lad(\ms{G}_0,\ms{G}_1) =  \inf_{\mu \in \M(Q,\xx)} \sup_{\Phi \in C^1(Q,\S^n)} \L(\mu, \Phi)\,.
\end{equation*} 
By changing the order of $\sup$ and $\inf$,  a formal calculation via integration by parts gives the dual problem:
\begin{align} \label{def:jadp}
{\rm WB}_{\Lad}^2(\ms{G}_0,\ms{G}_1) & \ge   
    \sup_\Phi \inf_\mu \L(\mu, \Phi) \notag \\
& = \sup_{\Phi}  \Big\{ \l \ms{G}_1, \Phi_1 \r_\Omega - \l \ms{G}_0, \Phi_0 \r_\Omega  \,; \ \p_t \Phi + \frac{1}{2} (\ms{D}^* \Phi) \Lad_1^2 (\ms{D}^* \Phi)^{\rm T} + \frac{1}{2} \Phi \Lad_2^2 \Phi   \preceq 0 \Big\}\,. 
\end{align}
We next use the Fenchel-Rockafellar theorem (Lemma \ref{prop:fench_rock}) to show that the 
duality gap is zero, which will also give the existence of the minimizer to \eqref{eq:distance} and the optimality conditions. For this, we define 
\begin{equation} \label{def:colad}
    C(Q,\mathcal{O}_\Lad) := \{\vp \in C(Q,\xx)\,;\ \vp(x) \in \mathcal{O}_\Lad\,, \ \forall  x \in Q\}\,,
\end{equation}
with $\mc{O}_\Lad$ given in \eqref{eq:closed_convex_set},
which is a closed convex subset  of $C(Q,\xx)$. We then define lower semicontinuous convex 
functions: 
$f(\Phi) = \l \ms{G}_1, \Phi_1 \r_\Omega - \l \ms{G}_0, \Phi_0 \r_\Omega$ for $\Phi \in  C^1(Q,\S^n)$ and $g(\Xi) = \iota_{C(Q,\mathcal{O}_\Lad)}(\Xi)$ for $\Xi \in C(Q,\xx)
$. We also introduce the bounded linear operator: 
$L: \Phi \in C^1(Q, \S^n) \to (\p_t \Phi, \ms{D}^* \Phi, \Phi) \in C(Q,\xx)$ with the dual operator $L^*$. These notions help us to write \eqref{def:jadp} as 
$
    \sup\{f(\Phi) - g(L \Phi)\,;\ \Phi \in C^1(Q,\S^n)\}\,.
$

We now verify the condition in Lemma \ref{prop:fench_rock}. 
We consider $\Phi = -\ep t I + \frac{\ep}{2} I \in C^1(Q,\S^n)$. It is clear that $f(\Phi)$ is finite and 
$L \Phi = (- \ep I, 0,  -\ep t I + \frac{\ep}{2} I)$ by $\ms{D}^*(I) = 0$. By a simple calculation, we have 
\begin{align*}
    \p_t \Phi + \frac{1}{2} (\ms{D}^*\Phi) \Lad_1^2 (\ms{D}^* \Phi)^{{\rm T}} + \frac{1}{2} \Phi \Lad_2^2 \Phi 
     &= - \ep I+ \frac{1}{2}\ep^2 \big(- t + \frac{1}{2}\big)^2 \Lad_2^2 \preceq - \ep I+ \frac{1}{8} \ep^2 \Lad_2^2\,,
\end{align*}
which implies that for small enough $\ep$ and any $(t,x)\in Q$, $(L \Phi)(t,x)$ is in the interior of $\mathcal{O}_\Lad$ and hence $g$ is continuous at $L \Phi$. Then 
Lemma \ref{prop:fench_rock} readily gives
\begin{equation} \label{eq:auxeq_o}
     \min_{\mu \in \M(Q,\xx)}f^*(L^*\mu) + g^*(\mu) = \sup_{\Phi \in C^1(Q,\S^n)} f(\Phi) - g(L \Phi)\,, 
\end{equation}
where $f^*(L^*\mu) = \sup\{\l \mu, L\Phi \r_Q - f(\Phi)\,;\ \Phi \in C^1(Q,\S^n)\}$ can be easily computed as $\iota_{\ce([0,1];\ms{G}_0,\ms{G}_1)}$ by linearity of $f$, while 
$g^*(\mu)$ is nothing else than $\jj_{\Lad,Q}(\mu)$ by the following lemma, which is a 
direct application of general results \cite{bouchitte1988integral,rockafellar1971integrals}. 
We sketch the proof in Appendix \ref{app_B} for completeness.

\begin{lemma} \label{lem: conj_cvx}
Let $\mathcal{X}$ be a compact separable metric space 
and $C(\mathcal{X},\mathcal{O}_\Lad)$ be defined in \eqref{def:colad}.
Then, we have  
\begin{equation} \label{eq:conj_gmu}
    \iota^*_{C(\mathcal{X},\mathcal{O}_\Lad)} = \sup_{\Xi \in L_{|\mu|}^\infty(\mathcal{X},\mathcal{O}_\Lad)}\l \mu, \Xi \r_{\mathcal{X}} = \jj_{\Lad,\mathcal{X}} (\mu)\,,\q \text{for}\ \mu \in \mc{M}(\mc{X},\xx)\,,
\end{equation}
which is proper convex and lower semicontinuous with respect to the weak* topology of $\M(\mathcal{X},\xx)$. Moreover, the subgradient $\p \jj_{\Lad,\mathcal{X}}(\mu)$ in  $C(\mathcal{X},\xx)$ is given as follows:
\begin{equation} \label{eq:subg_conj_gmu}
    \p \jj_{\Lad,\mathcal{X}}(\mu)|_{C(\mathcal{X},\xx)} = \left\{\Xi \in C(\mathcal{X}, \mathcal{O}_\Lad)\,; \ \Xi(x) \in \p J_\Lad(\mu_\lad)(x)\,, \ \lad\text{--a.e.}\right\}\,,
\end{equation}
which is independent of the choice of the reference measure $\lad$ such that $|\mu| \ll \lad$.
\end{lemma}

By the above arguments, we have shown the following result. 

\begin{theorem} \label{thm:strong_dual}
    The optimization problem \eqref{eq:distance} always admits a minimizer $\mu \in \ce([0,1];\ms{G}_0,\ms{G}_1)$ and a dual formulation with zero duality gap:
    \begin{equation} \label{eq:dualforkb}
       {\rm WB}^2_{\Lad}(\ms{G}_0,\ms{G}_1) = \sup_{\Phi \in C^1(Q,\S^n)} \left\{ \l \ms{G}_1,  \Phi_1 \r_{\Omega} - \l \ms{G}_0,  \Phi_0 \r_\Omega - \iota_{C(Q,\mathcal{O}_\Lad)}(\p_t \Phi, \ms{D}^* \Phi, \Phi)   \right\}\,,
    \end{equation}
    where the $\sup$ is attained at $\Phi \in C^1(Q,\S^n)$ if and only if     
    there exists $\mu = \ms{(G,q,R)} \in \ce([0,1];\ms{G}_0,\ms{G}_1)$ such that 
    \begin{align} \label{opt_a}
        q_\lad =  G_\lad (\ms{D}^* \Phi) \Lad_1^2 \,, \q R_\lad = G_\lad \Phi \Lad_2^2\,,
    \end{align}
    and 
    \begin{equation} \label{opt_b}
      G_\lad \dd \Big(\p_t \Phi + \frac{1}{2} (\ms{D}^* \Phi) \Lad_1^2 (\ms{D}^* \Phi)^{{\rm T}} + \frac{1}{2} \Phi \Lad_2^2 \Phi \Big) = 0 \,,
    \end{equation}
    for $\lad$-a.e.\,$(t,x) \in Q$. In this case, $\mu$ is also the minimizer to the problem \eqref{eq:distance}. 
\end{theorem}

As a consequence of Lemma \ref{lem: conj_cvx} and the dual formulation \eqref{eq:dualforkb}, we have the sublinearity and the weak* lower semicontinuity of ${\rm WB}^2_{\Lad}(\dd,\dd)$. 
    
\begin{corollary}\label{lem:lsckbalpa}
 ${\rm WB}^2_{\Lad}(\dd,\dd)$ is sublinear: for $\alpha > 0$, $\ms{G}_0,\ms{G}_1,\w{\ms{G}}_0,\w{\ms{G}}_1 \in \M(\Omega,\S^n_+)$, there holds  
\begin{align} \label{sublienar}
    {\rm WB}^2_{\Lad}\big(\alpha \ms{G}_0, \alpha\ms{G}_1\big) = \alpha {\rm WB}^2_{\Lad}\big(\ms{G}_0, \ms{G}_1\big)\,,\q  {\rm WB}^2_{\Lad}\big(\ms{G}_0 +\w{\ms{G}}_0, \ms{G}_1 + \w{\ms{G}}_1\big) \le {\rm WB}^2_{\Lad}\big(\ms{G}_0, \ms{G}_1\big) +  {\rm WB}^2_{\Lad}\big(\w{\ms{G}}_0, \w{\ms{G}}_1\big)\,.
\end{align}
Moreover, ${\rm WB}_{\Lad}$ is lower semicontinuous with respect to the weak* topology, that is, for any sequences $\{\ms{G}^n_0\}_{n \in \NN}$ and  $\{\ms{G}^n_1\}_{n \in \NN}$ in $\M(\Omega,\S_+^n)$ that weak* converge to measures $\ms{G}_0, \ms{G}_1 \in \M(\Omega,\S_+^n)$, respectively, there holds 
    \begin{align}  \label{eq:kblsc}
        {\rm WB}_{\Lad} (\ms{G}_0,\ms{G}_1) \le \liminf_{n \to 0} {\rm WB}_{\Lad} (\ms{G}^n_0,\ms{G}^n_1)\,.
    \end{align}
\end{corollary}
\begin{proof}
Noting that $\jj_{\Lad,Q}(\mu)$ is positively homogeneous and convex, and hence sublinear, the sublinearity of ${\rm WB}^2_{\Lad}(\dd,\dd)$ follows from definition \eqref{eq:distance} and the linearity of the continuity equation.
For the weak* lower
semicontinuity, by \eqref{eq:dualforkb}, for any $\Phi \in C^1(Q,\S^n)$ with $\iota_{C(Q,\mathcal{O}_\Lad)}(\p_t \Phi, \ms{D}^* \Phi, \Phi) = 0$, there holds
    \begin{align} \label{kblac_aux_1}
      \liminf_{n \to \infty}  {\rm WB}^2_{\Lad}(\ms{G}^n_0,\ms{G}^n_1) \ge  \liminf_{n \to \infty} \l \ms{G}^n_1,  \Phi_1 \r_{\Omega} - \l \ms{G}^n_0,  \Phi_0 \r_\Omega = \l \ms{G}_1,  \Phi_1 \r_{\Omega} - \l \ms{G}_0,  \Phi_0 \r_\Omega\,,
    \end{align}
    by the weak* convergence of $\ms{G}_0^n$ and $\ms{G}_1^n$. Then \eqref{eq:kblsc} follows by taking the $\sup$ of \eqref{kblac_aux_1} over admissible $\Phi$. 
\end{proof}

In addition, we have the following explicit characterization of the 
minimizer (i.e., geodesic; see Corollary \ref{cor:geo_me}) to \eqref{eq:distance} for inflating measures from optimality conditions \eqref{opt_a} and \eqref{opt_b}, which extends \cite[Theorem 5]{brenier2020optimal} with a much simpler argument.
For $\ms{G} \in \M(\Omega, \S^n_+)$ and $A \in \S_+^n$, we denote by $\ms{G}^A$ the inflating measure $A \ms{G} A \in \M(\Omega, \S^n_+)$. 

\begin{proposition} \label{prop:inflatingmeaure}
For $\ms{G} \in \M(\Omega, \S^n_+)$ and matrices $A_0, A_1 \in \S_+^n$, we have
\begin{equation} \label{eq:expligeo}
    {\rm WB}_\Lad^2 \big(\ms{G}^{A_0},\ms{G}^{A_1}\big) = 2 \tr \big(\Lad_2^{-1}(A_1 - A_0)\ms{G}(\Omega)(A_1 - A_0)\Lad_2^{-1} \big)\,,
\end{equation}
with the minimizer $(\ms{G}_*,\ms{q}_*,\ms{R}_*) := 
(\ms{G}^{A_t}, 0, 2 A_t \ms{G} (A_1 - A_0)) \in \M(Q, \xx)$, where $A_t := tA_1 + (1- t)A_0$ for $t \in [0,1]$.
\end{proposition}

\begin{proof}
Let us first assume that $A_0$ and $A_1$ are invertible. By a direct calculation, we have  
\begin{equation*}
   \p_t \ms{G}^{\ms{A}_t} = (A_1 - A_0) \ms{G} A_t + A_t \ms{G} (A_1 - A_0) \,.
\end{equation*}
We define $\Phi = 2 A_t^{-1} (A_1- A_0)\Lad_2^{-2}$ and find 
$\ms{R}_* = \ms{G}^{A_t} \Phi \Lad_2^2$. It is also easy to see that 
$(\ms{G}_*,\ms{q}_*,\ms{R}_*)$ defined above is in the set 
 $\ce\big([0,1]; \ms{G}^{A_0},\ms{G}^{A_1}\big)$.  Moreover, recalling $ ((A + \ep H)^{-1} - A^{-1})/\ep \to - A^{-1}H A^{-1}$ as $\ep \to 0$ for invertible $A$ and $H \in \MM^n$ \cite{bhatia2013matrix}, we have 
\begin{equation*}
    \p_t \Phi =  - 2  A_t^{-1} (A_1 - A_0)A_t^{-1}(A_1- A_0) \Lad_2^{-2} = - \Phi \Lad_2^2 \Phi/2\,.
\end{equation*}
By the above computations, we have verified the optimality conditions \eqref{opt_a} and \eqref{opt_b}, which means that the measure $(\ms{G}_*,\ms{q}_*,\ms{R}_*)$ is the desired minimizer. Then, we can further compute
\begin{equation*}
    {\rm WB}_\Lad^2\big(\ms{G}^{A_0}, \ms{G}^{A_1}\big) = \frac{1}{2} \int_0^1 \int_\Omega (\Phi \Lad_2) \dd \rd \ms{G}^{A_t} (\Phi \Lad_2)\, \rd t = 2 ((A_1 - A_0)\Lad_2^{-1})\dd \ms{G}(\Omega)(A_1 - A_0)\Lad_2^{-1}\,.
\end{equation*}
For general $A_0,A_1 \in \S_+^n$, we first see that $\mu_* := 
(\ms{G}^{A_t}, 0, 2 A_t \ms{G} (A_1 - A_0))$ as above still satisfies the continuity equation and its associated action functional $\jj_{\Lad,Q}(\mu_*)$ gives the right-hand side of \eqref{eq:expligeo} by $\ran(A_1 - A_0) \subset \ran(A_t)$, which also means $ {\rm WB}_\Lad^2(\ms{G}^{A_0}, \ms{G}^{A_1}) \le \jj_{\Lad,Q}(\mu_*)$. 
To finish the proof, it suffices to show that the equality holds. For this, we consider $A_i^\ep = A_i +  \ep I \in \S_{++}^n$ for $i = 0,1$. 
Then, by triangle inequality of ${\rm WB}_\Lad$ (see Proposition \ref{prop:metric_prop} below) and Lemma \ref{lem:com_kb_and_b}, we have 
${\rm WB}_\Lad(\ms{G}^{A^\ep_0}, \ms{G}^{A^\ep_1}) \to {\rm WB}_\Lad(\ms{G}^{A_0}, \ms{G}^{A_1})$ as $\ep \to 0$. The proof is completed by 
\begin{align*}
    {\rm WB}^2_\Lad\big(\ms{G}^{A^\ep_0}, \ms{G}^{A^\ep_1}\big) = & 2 \tr \big(\Lad_2^{-1}(A^\ep_1 - A^\ep_0)\ms{G}(\Omega)(A^\ep_1 - A^\ep_0)\Lad_2^{-1} \big)\\
& \to 2 \tr \big(\Lad_2^{-1}(A_1 - A_0)\ms{G}(\Omega)(A_1 - A_0)\Lad_2^{-1} \big) = \jj_{\Lad,Q}(\mu_*) \,, \q \ep \to 0\,. \qedhere
\end{align*}
\end{proof}

\medskip
\noindent \textbf{Primal-dual formulations.}
We proceed to study in more depth the optimality conditions by viewing $\ms{G}$ as the main variable and $\ms{(q,R)}$ as the control variable, which will be useful in Section \ref{sec:Riema_iter}. 
We first observe 
\begin{align} \label{eq:dis_dual_1}
   {\rm WB}^2_{\Lad}(\ms{G}_0,\ms{G}_1) & = \inf_{\ms{G}} \inf_{\ms{q},\ms{R}} \left\{ 
   \jj_{\Lad,Q}(\mu)      
        \,;\ \mu = (\ms{G}, \ms{q}, \ms{R}) \in \ce_\infty([0,1];\ms{G}_0,\ms{G}_1)
         \right\}\,,
\end{align}
by taking the inf in \eqref{eq:distance} over $\ms{G}$ and $(\ms{q},\ms{R})$ separately. Recall 
the formulation \eqref{rep:enerycost} of $\jj_{\Lad,Q}(\mu)$, which motivates us to 
introduce a weighted semi-inner product: 
\begin{align} \label{eq:innerweight}
    \big\l (u,W), (u',W') \big\r_{L^2_{\ms{G},\Lad}(Q)} := \big\l u \Lad_1^{\dag}, u' \Lad_1^{\dag} \big\r_{L^2_{\ms{G}}(Q)} +  \big\l W \Lad_2^{-1}, W' \Lad_2^{-1} \big\r_{L^2_{\ms{G}}(Q)} \,,
\end{align}
and the associated seminorm $\norm{\dd}_{L^2_{\ms{G},\Lad}(Q)}$ on the space of measurable functions valued in $\R^{n \t k} \t \MM^n$. The corresponding Hilbert space, denoted by $L^2_{\ms{G},\Lad}(Q,\R^{n \t k} \t \MM^n)$, is defined as the quotient space by the subspace 
$\ker\big (\norm{\dd}_{L^2_{\ms{G},\Lad}(Q)}\big)$. Hence, we can rewrite \eqref{rep:enerycost} as $\jj_{\Lad,Q}(\mu) = \norm{(G^\dag q,G^\dag R)}^2_{L^2_{\ms{G},\Lad}(Q)}/2$. 
Moreover, we define the set 
\begin{equation} \label{def:acce}
    \mathcal{AC}([0,1];\ms{G}_0,\ms{G}_1): = \{\ms{G} \in \M(Q,\S^n)\,; \ \exists \ms{(q,R)} \in \M(Q,\R^{n \t k} \t \MM^n) \ \text{s.t.}\ \ms{(G,q,R)} \in \ce_\infty([0,1];\ms{G}_0,\ms{G}_1)\}\,,
\end{equation}
and the associated energy functional: for $\ms{G} \in  \mathcal{AC}([0,1];\ms{G}_0,\ms{G}_1)$, 
\begin{align} \label{eq:sub_optprob}
    \jj^{\Lad}_{\ms{G}_0,\ms{G}_1}(\ms{G}) := \inf_{\ms{(q,R)}}  \Big\{\frac{1}{2} \norm{(G^\dag q, G^\dag R )}^2_{L^2_{\ms{G},\Lad}(Q)}\,;\   \ms{(G,q,R)} \in \ce_\infty([0,1];\ms{G}_0,\ms{G}_1)
    \Big\}\,.
\end{align}
We will see in Remark \ref{rem_ac} that $\mathcal{AC}([0,1];\ms{G}_0,\ms{G}_1)$ is closely related to the set of absolutely continuous curves in the metric space $(\M(\Omega,\S^n_+), {\rm WB}_{\Lad})$. 
With the help of these notions, \eqref{eq:dis_dual_1} can be reformulated in a compact form: 
\begin{equation} \label{def:kbalphag}
   {\rm WB}^2_{\Lad}(\ms{G}_0,\ms{G}_1) = \inf_{\ms{G}\in 
   \mathcal{AC}([0,1];\ms{G}_0,\ms{G}_1)} \jj^{\Lad}_{\ms{G}_0,\ms{G}_1}(\ms{G})\,.
\end{equation} 

Similarly to \eqref{rep:enerycost}, by Lemma \ref{lem:fcabs}, we also note that for $\ms{(G,q,R)} \in \ce_\infty([0,1];\ms{G}_0,\ms{G}_1)$, the weak formulation \eqref{eq:weak_ctneq} can be written as
\begin{equation}  \label{eq:continui_equation}
  \big\l \big(\ms{D}^* \Phi \Lad^2_1, \Phi \Lad_2^2 \big), \big(G^\dag q, G^\dag R\big)
 \big\r_{L^2_{\ms{G}, \Lad}(Q)}  = l_{\ms{G}}(\Phi)\,, \q \forall \Phi \in C^1(Q,\S^n)\,,
\end{equation}
where $l_{\ms{G}}(\dd)$ for $\ms{G} \in \mathcal{AC}([0,1];\ms{G}_0,\ms{G}_1)$ is a linear functional on $C^1(Q,\S^n)$ defined by 
\begin{equation} \label{def:lgphi}
    l_{\ms{G}}(\Phi) =  \l \ms{G}_1, \Phi_1 \r_{\Omega} -  \l \ms{G}_0, \Phi_0 \r_{\Omega} -    \l \ms{G}, \p_t\Phi\r_Q \,.
\end{equation}
Define a bijective map $\Pi: \Phi \to (\ms{D}^* \Phi \Lad_1^2, \Phi \Lad_2^2)$ for $\Phi \in C^1(Q,\S^n)$ and denote $\w{l}_{\ms{G}}: =  l_{\ms{G}} \circ \Pi^{-1}$. In view of \eqref{eq:continui_equation},
the functional $\w{l}_{\ms{G}}$ can be uniquely extended to the space 
\begin{equation} \label{def:hladd}
    H_{\ms{G},\Lad}(\ms{D}^*): = \overline{\left\{ \Pi(\Phi) \,;\ \Phi \in C^1(Q,\S^n)\right\}}^{\norm{\dd}_{L^2_{\ms{G},\Lad}(Q)}}\,,  
\end{equation}
with the norm estimate 
\begin{equation} \label{est:normlg}
    \norm{\w{l}_{\ms{G}}}_{H^*_{\ms{G},\Lad}(\ms{D}^*)} \le \norm{(G^\dag q, G^\dag R)}_{L^2_{\ms{G},\Lad}(Q)}\,.    
\end{equation}
It is worth emphasizing that such kind of extension is independent of the choice of $\ms{(q,R)}$. 

Next, we show that \eqref{eq:sub_optprob} admits a unique minimizer $\ms{(q,R)}$ that also satisfies the equality in 
\eqref{est:normlg}. Note that $(u,W)$ and $(u \P_{\Lad_1}, W \P_{\Lad_2})$ are equivalent in $L^2_{\ms{G},\Lad}(Q,\R^{n \t k} \t \MM^n)$, where $\P_{\Lad_i}$ is the orthogonal projection to $\ran(\Lad_i)$. Hence, for any  $(u,W) \in L^2_{\ms{G},\Lad}(Q,\R^{n \t k} \t \MM^n)$, we can assume
$\ran(u^{\rm T}) \subset \ran (\Lad_1)$ and $\ran(W^{\rm T}) \subset \ran (\Lad_2)$.  Then, it holds that
any $L^2_{\ms{G},\Lad}$-field $(u,W)$ satisfying
 $ \l (\ms{D}^* \Phi \Lad_1^2, \Phi \Lad_2^2), (u,W)\r_{L^2_{\ms{G},\Lad}(Q)}  = l_{\ms{G}}(\Phi)$, $\forall \Phi \in C^1(Q,\S^n)$, induces a measure  $\ms{(q,R)} := (\ms{G} u,\ms{G}W)$ such that $\ms{(G,q,R)} \in \ce_\infty([0,1];\ms{G}_0,\ms{G}_1)$. 
This observation implies that $\jj^{\Lad}_{\ms{G}_0,\ms{G}_1}(\ms{G})$ is actually a uniquely solvable minimum norm problem with an affine constraint:
\begin{align} \label{eq:sub_opt_2}
    \jj^{\Lad}_{\ms{G}_0,\ms{G}_1}(\ms{G}) = \inf  \Big\{\frac{1}{2} \norm{(u,W)}^2_{L^2_{\ms{G},\Lad}(Q)}\,;\  (u,W) \in L^2_{\ms{G},\Lad}(Q,\R^{n \t k} \t \MM^n)\ \text{such that} & \notag  \\
    \big\l (\ms{D}^* \Phi \Lad_1^2, \Phi \Lad_2^2), (u,W) \big\r_{L^2_{\ms{G},\Lad}(Q)} 
    = l_{\ms{G}}(\Phi)\,, \ \forall \Phi \in C^1(Q,\S^n)
   & \Big\}\,.
\end{align}
The unique minimizer $(u_*,W_*)$ to \eqref{eq:sub_opt_2}
 is given by the orthogonal projection of $0$ on the constraint set, equivalently, the Riesz representation of the functional $\w{l}_{\ms{G}}$ on the space $H_{\ms{G},\Lad}(\ms{D}^*)$. It then follows that $(\ms{q}_*,\ms{R}_*) := (\ms{G} u_*,\ms{G}W_*)$ is the desired minimizer to \eqref{eq:sub_optprob} and there holds 
\begin{align} \label{eq:normlg}
    \norm{\w{l}_{\ms{G}}}_{H^*_{\ms{G},\Lad}(\ms{D}^*)} 
    & =   \norm{(u_*,W_*)}_{L^2_{\ms{G},\Lad}(Q)} = \norm{(G^\dag q_*, G^\dag R_*)}_{L^2_{\ms{G},\Lad}(Q)}\,.
\end{align}

We summarize the above facts in the following useful result. 

\begin{theorem} \label{thm:dual_minmax}
${\rm WB}^2_{\Lad}(\ms{G}_0,\ms{G}_1)$ has the following  representation: 
\begin{equation*}
    {\rm WB}^2_{\Lad}(\ms{G}_0,\ms{G}_1)  =  \inf_{\ms{G}\in \mathcal{AC}([0,1];\ms{G}_0,\ms{G}_1)} \jj^{\Lad}_{\ms{G}_0,\ms{G}_1}(\ms{G})\q \text{with}\q \jj^{\Lad}_{\ms{G}_0,\ms{G}_1}(\ms{G}) =  \frac{1}{2}\norm{(u_*, W_*)}^2_{L^2_{\ms{G},\Lad}(Q)}\,,
\end{equation*}
where 
$(u_*, W_*)$ is the Riesz representation of $\w{l}_{\ms{G}}$ in $H_{\ms{G},\Lad}(\ms{D}^*)$ that uniquely solves the minimum norm problem \eqref{eq:sub_opt_2}. 
Moreover, $\jj^{\Lad}_{\ms{G}_0,\ms{G}_1}(\ms{G})$ admits the following dual formulation: 
\begin{equation} \label{opt_3}
    \jj^{\Lad}_{\ms{G}_0,\ms{G}_1}(\ms{G}) = \sup \Big  \{l_{\ms{G}}(\Phi) - \frac{1}{2}\norm{(\ms{D}^* \Phi \Lad_1^2,\Phi \Lad_2^2)}^2_{L^2_{\ms{G},\Lad}(Q)}\,; \ \Phi \in C^1(Q,\S^n) \Big\}\,.
\end{equation} 
\end{theorem}
    
\begin{proof} 
    It suffices to derive the dual formulation \eqref{opt_3} of
$\jj^{\Lad}_{\ms{G}_0,\ms{G}_1}$.  For this, we first note  
\begin{equation*}
     \frac{1}{2}\norm{(u,W)}_{L^2_{\ms{G},\Lad}(Q)}^2 = \sup_{(u',W') \in L^2_{\ms{G},\Lad}(Q,\R^{n \t k} \t \MM^n)} \l (u,W), (u',W')\r_{L^2_{\ms{G},\Lad}(Q)} -  \frac{1}{2}\norm{(u',W')}^2_{L^2_{\ms{G},\Lad}(Q)} \,, 
\end{equation*}
which further implies, by  $(u_*, W_*) \in  H_{\ms{G},\Lad}(\ms{D}^*) \subset L^2_{\ms{G},\Lad}(Q,\R^{n \t k} \t \MM^n)$, for any $\Phi \in C^1(Q,\S^n)$,
\begin{align} \label{eq_dualaux}
    \jj^{\Lad}_{\ms{G}_0,\ms{G}_1}(\ms{G}) = \frac{1}{2}\norm{(u_*, W_*)}_{L^2_{\ms{G},\Lad}(Q)}^2 &\ge \l (u_*, W_*), \big(\ms{D}^* \Phi \Lad^2_1, \Phi \Lad_2^2 \big)\r_{L^2_{\ms{G},\Lad}(Q)}
    - \frac{1}{2}\norm{( \ms{D}^* \Phi \Lad_1^2,\Phi \Lad_2^2)}^2_{L^2_{\ms{G},\Lad}(Q)} \\
    & = l_G(\Phi) - \frac{1}{2}\norm{( \ms{D}^* \Phi \Lad_1^2,\Phi \Lad_2^2)}^2_{L^2_{\ms{G},\Lad}(Q)}\,. \notag
\end{align}
Then, recalling \eqref{def:hladd} and choosing a sequence $\{(\ms{D}^* \Phi_n \Lad_1^2, \Phi_n \Lad_2^2)\}$ with $\Phi_n \in C^1(Q,\S^n)$ in \eqref{eq_dualaux} that approximates $(u_*, W_*)$ gives the desired
\eqref{opt_3}.
\end{proof}

 \medskip
\noindent \textbf{Varying weighted matrices.} 
We regard ${\rm WB}_{\Lad}$ as a family of distances indexed by $\Lad$, and investigate the behaviors of ${\rm WB}_{\Lad}$ and its minimizer when $\Lad$ varies, in particular, when $|\Lad_1|$ or $|\Lad_2|$ tends to zero or infinity. We give a partial answer to this question in the following proposition. 
For ease of exposition, we introduce  
\begin{align} \label{def:sepata_cost}
   \jj^q_{\Lad_1}(\mu) = \jj_{\Lad,Q}(\ms{(G,q,0)}) 
    \,,  \q \jj^R_{\Lad_2}(\mu) = \jj_{\Lad,Q}(\ms{(G,0,R)}) \q \text{for}\  \mu \in \M(Q,\xx)\,.
\end{align}

\begin{proposition} \label{prop:limitweight}
Let $\ms{G}_0,\ms{G}_1 \in \M(\Omega,\S^n_+)$ and  
$\mu_{*,\Lad}$ denote the minimizer to ${\rm WB}^2_\Lad(\ms{G}_0,\ms{G}_1)$ \eqref{eq:distance}.
It holds that ${\rm WB}^2_{(\Lad_1,\Lad_2)}(\ms{G}_0,\ms{G}_1) \to {\rm WB}^2_{(0,\Lad_2)}(\ms{G}_0,\ms{G}_1)$ as $\norm{\Lad_1}_\ff \to 0$, and for any sequence $\{ \Lad_{1,j}\}_{j \in \NN} \subset \S^k_+$ with $\norm{\Lad_{1,j}}_\ff \to 0$, the associated minimizer $\mu_{*,(\Lad_{1,j},\Lad_2)}$, up to a subsequence, weak* converges to a minimizer $\mu_*$ to ${\rm WB}^2_{(0,\Lad_2)}(\ms{G}_0,\ms{G}_1)$. 
\end{proposition}

\begin{proof} 
We first claim that $\norm{\Lad_1}_\ff^2\jj^q_{\Lad_1}(\mu_{*,\Lad})$ and $\jj^R_{\Lad_2}(\mu_{*,\Lad})$ are bounded when $\norm{\Lad_1}_\ff \to 0$, which, by estimates \eqref{priest_qR} and \eqref{priest_G}, implies that $\mu_{*,\Lad}$ is bounded in $\M(Q,\xx)$. For this, we consider the set 
\begin{align} \label{def:ceq}
\ce_{\Lad_1, q} : =  \arg\min\{\jj^q_{\Lad_1}(\mu)\,; \ \mu \in \ce([0,1];\ms{G}_0,\ms{G}_1)\}\,.
\end{align}
Similarly to the proof of Lemma \ref{lem:com_kb_and_b},  we have that $\ce_{\Lad_1, q}$ is nonempty and contains at least one element with $\ms{q} = 0$, and $\min\{\jj^q_{\Lad_1}(\mu)\,; \ \mu \in \ce([0,1];\ms{G}_0,\ms{G}_1)\} = 0$. Since $\mu_{*,\Lad}$ minimizes $\jj_{\Lad,Q}(\dd)$, it follows that
\begin{equation} \label{est_aux_1}
    \jj_{\Lad,Q}(\mu_{*,\Lad}) = \jj^q_{\Lad_1}(\mu_{*,\Lad}) +  \jj^R_{\Lad_2}(\mu_{*,\Lad})\le \jj_{\Lad,Q}(\mu) =  \jj^R_{\Lad_2}(\mu)\,, \q \forall \mu = \ms{(G,0,R)}\in \ce_{\Lad_1, q}\,.
\end{equation}
Noting $\{\ms{(G,0,R)}\in \ce_{\Lad_1, q}\} = \{\ms{(G,0,R)} \in \ce([0,1];\ms{G}_0,\ms{G}_1)\}$, \eqref{est_aux_1} yields that $\jj^R_{\Lad_2}(\mu_{*,\Lad})$ is bounded by a constant independent of $\Lad_1$. Moreover, multiplying $\norm{\Lad_1}_\ff^2$ on both sides of \eqref{est_aux_1} and then letting $\norm{\Lad_1}_\ff \to 0$ , we obtain
\begin{equation}  \label{est_aux_2}
    \lim_{\norm{\Lad_1}_\ff \to 0} \norm{\Lad_1}_\ff^2\,\jj^q_{\Lad_1}(\mu_{*,\Lad})  = 0\,.
\end{equation}
Then the boundedness of $\norm{\Lad_1}_\ff^2\jj^q_{\Lad_1}(\mu_{*,\Lad})$ for small enough $\norm{\Lad_1}_\ff$ follows. We complete the proof of the claim. 

By the boundedness of $\norm{\mu_{*,\Lad}}_{\rm TV}$ as $\norm{\Lad_1}_\ff \to 0$, 
we are allowed to take a subsequence $\{\Lad_{1,j}\}_{j \in \NN} $ in $\S_+^n$ such that  the minimizer $\mu_{*,\w{\Lad}_j}$ with $\w{\Lad}_j = (\Lad_{1,j},\Lad_2)$ weak* converges to a measure $\mu_* \in \M(Q,\xx)$ when $n \to \infty$, which clearly satisfies $\mu_* \in \ce([0,1]; \ms{G}_0,\ms{G}_1)$. Then, by the weak* lower semicontinuity of $\jj^R_{\Lad_2}$ and \eqref{est_aux_1}, we have   
\begin{equation}  \label{est_aux_3}
  \jj^R_{\Lad_2}(\mu_*) \le \liminf_{j\to \infty} \jj^R_{\Lad_2}(\mu_{*,\w{\Lad}_j}) \le \limsup_{j\to \infty} {\rm WB}^2_{\w{\Lad}_j}(\ms{G}_0,\ms{G}_1) \le \inf \{\jj^R_{\Lad_2}(\mu)\,;\ \mu = \ms{(G,0,R)}\in \ce_{\Lad_1, q}\}\,.
\end{equation}
The right-hand side of \eqref{est_aux_3} is recognized as ${\rm WB}_{(0,\Lad_2)}(\ms{G}_0,\ms{G}_1)$ and the inf is attained; see Remark \ref{rem:bures_2} and Theorem \ref{thm:strong_dual}.
Also, by \eqref{priest_qR} and \eqref{est_aux_2}, it holds that the limit measure $\mu_* \in \ce([0,1];\ms{G}_0,\ms{G}_1)$ is of the form $\mu_* 
 = \ms{(G_*,0,R_*)}$. The proof is completed by \eqref{est_aux_3}.  
\end{proof}

Proposition \ref{prop:limitweight} above tells us that 
the measure $\ms{q}$ is forced to be nearly zero, if 
the transportation part is given too much weight (i.e., $\norm{\Lad_1}_\ff$ is small, cf.\,\eqref{rep:enerycost}), equivalently, if the problem is on a large scale (cf.\,Remark \ref{rem;space}). It is also possible and interesting to consider other limiting regimes, e.g., $\norm{\Lad_1}_\ff \to \infty$, $\norm{\Lad_2}_\ff \to 0$, or only let part of eigenvalues of $\Lad_i$ vanish, which, however, is beyond the scope of this work.

\section{Geometric properties and Riemannian interpretation} \label{sec:Riema_iter}

In this section, we shall study the space $\M(\Omega,\S^n_+)$ equipped with the distance ${\rm WB}_{\Lad}(\dd,\dd)$ from the metric point of view. In particular, we will prove that $(\M(\Omega,\S^n_+),{\rm WB}_{\Lad})$ is a complete geodesic space with a Riemannian interpretation. 
We first show that  
${\rm WB}_{\Lad}(\dd,\dd)$ is indeed a metric on $\M(\Omega,\S^n_+)$, which is a simple corollary of the following characterization of ${\rm WB}_{\Lad}(\dd,\dd)$ by standard reparameterization techniques 
(cf.\,\cite[Lemma 1.1.4]{ambrosio2008gradient} or \cite[Theorem 5.4]{dolbeault2009new}). We denote by $\w{\ce}([a,b];\ms{G}_0,\ms{G}_1)$ the set of measures $\mu \in \ce([a,b];\ms{G}_0,\ms{G}_1)$ that can be disintegrated as $\mu = \int_a^b \d_t \otimes \mu_t\, \rd t$. It is clear that $\ce_\infty \subset \w{\ce} \subset \ce$. 

\begin{lemma}  \label{thm:geodesic_space}
For $\ms{G}_0,\ms{G}_1 \in \M(\Omega,\S_+^n)$ and $b > a > 0$, there holds   
    \begin{align} \label{eq:length_energy}
    {\rm WB}_{\Lad}(\ms{G}_0,\ms{G}_1) = \inf_{\mu \in \w{\ce}([a,b];\ms{G}_0,\ms{G}_1)} \int_a^b  \jj_{\Lad,\Omega}(\mu_t)^{1/2} \,  \rd t \,.
    \end{align}
    Moreover, the minimizer to the problem \eqref{eq:change_time} gives  a constant-speed minimizer $\mu$ to \eqref{eq:length_energy}, which satisfies
    \begin{align} \label{eq:cons_energy}
        (b - a) J_{\Lad,\Omega}(\mu_t)^{1/2} ={\rm WB}_{\Lad}(\ms{G}_0,\ms{G}_1) \q \text{for}\ a.e.\, t \in [a,b]\,.
    \end{align}   
\end{lemma} 

The proof is provided in Appendix \ref{app_B} for completeness.
The above lemma is an analog of 
a well-known geometric fact that minimizing the energy of a parametric curve is the same as minimizing its length with constant-speed constraint \cite{flaherty2013riemannian}. The following result summarizes some fundamental properties of $(\M(\Omega,\S^n_+), {\rm WB}_{\Lad})$. 


\begin{proposition} \label{prop:metric_prop}
    $(\M(\Omega,\S^n_+), {\rm WB}_{\Lad})$ is a complete metric space. Moreover, the topology induced by the metric ${\rm WB}_{\Lad}$ is stronger than the weak* one, i.e., $\lim_{n \to \infty}{\rm WB}_{\Lad}(\ms{G}^n,\ms{G}) = 0$ implies the weak* convergence of $\ms{G}^n$ to $\ms{G}$.
\end{proposition}

The proof needs the following lemma from Lemma \ref{lem:com_kb_and_b} and a priori estimates \eqref{priest_qR} and \eqref{priest_G}. 

\begin{lemma} \label{lem:formetric}
A subset of $\M(\Omega,\S^n_+)$ is bounded with respect to 
the distance ${\rm WB}_\Lad$ if and only if it is bounded 
with respect to the total variation norm. Hence, a bounded set in $(\M(\Omega,\S^n_+), {\rm WB}_\Lad)$ is weak* relatively compact.   
\end{lemma}

\begin{proof}[Proof of Proposition \ref{prop:metric_prop}]
First, note that ${\rm WB}_{\Lad}$ is a function from $\M(\Omega,\S^n_+) \t \M(\Omega,\S^n_+)$ to $[0,+\infty)$. It is also easy to check the symmetry ${\rm WB}_{\Lad}(\ms{G}_0,\ms{G}_1) = {\rm WB}_{\Lad}(\ms{G}_1,\ms{G}_0)$ by Lemma \ref{lemma:timescaling} and the triangle inequality by \eqref{eq:length_energy}. To show that ${\rm WB}_{\Lad}$ is a metric, it suffices to prove that ${\rm WB}_{\Lad}(\ms{G}_0,\ms{G}_1) = 0$ implies $\ms{G}_0 = \ms{G}_1$. For this, suppose that $\mu = \ms{(G,q,R)}$ is a minimizer to \eqref{eq:distance} with $\jj_{\Lad,Q}(\mu) = 0$. Recalling the formula \eqref{rep:enerycost}, we have $\ms{(q,R)} = 0$. Then, taking test functions $\Phi(t,x) = \Psi(x)$ with $\Psi(x) \in C^1(\Omega,\S^n)$ in \eqref{eq:weak_ctneq}, we find $\l \ms{G}_1 - \ms{G}_0, \Psi\r_{\Omega} = 0$,  $\forall \Psi \in C^1(\Omega,\S^n)$, which implies $\ms{G}_0 = \ms{G}_1$.  
Next, we show that the metric space $(\M(\Omega,\S^n_+), {\rm WB}_{\Lad})$ is complete. Let $\{\ms{G}^n\}_{n \in \NN}$ be a Cauchy sequence in $(\M(\Omega,\S^n_+), {\rm WB}_{\Lad})$, and hence also bounded in ${\rm WB}_{\Lad}$. By Lemma \ref{lem:formetric}, we have that $\ms{G}^n$, up to a subsequence, weak* converges to a measure $\ms{G} \in \M(\Omega,\S_+^n)$. Then, by Corollary \ref{lem:lsckbalpa} and the fact that $\{\ms{G}^n\}$ is a Cauchy sequence, for small $\ep > 0$ and large enough $m$, there holds 
\begin{align*}
  \ep \ge  \liminf_{n \to 0}{\rm WB}_{\Lad}(\ms{G}^n,\ms{G}^m) \ge {\rm WB}_{\Lad}(\ms{G},\ms{G}^m)\,,
\end{align*}
which immediately gives ${\rm WB}_{\Lad}(\ms{G},\ms{G}^m) \to 0$ as $m  \to \infty$. To finish, we show that $\ms{G}^n$ weak* converges to $\ms{G}$ if $\ms{G}^n$ converges to $\ms{G}$ in $(\M(\Omega,\S^n_+), {\rm WB}_{\Lad})$. To do so, it suffices to note that by a similar argument as above,  every subsequence of $\ms{G}^n$ has a weak* convergent sub-subsequence to $\ms{G}$, which readily gives the weak* convergence of $\ms{G}^n$ to $\ms{G}$. 
\end{proof}

The main aim of this section is to show that  $(\M(\Omega,\S^n),{\rm WB}_{\Lad})$ is a geodesic space and then equip it with some differential structure that is consistent with the metric structure, in the spirit of \cite{dolbeault2009new,ambrosio2008gradient}.

For the reader's convenience, we recall some basic concepts for the analysis in metric spaces \cite{ambrosio2004topics}.  Let $(X,d)$ be a metric space 
and $\{\ww_t\}_{t \in [a, b]}$ be a curve in  $(X,d)$ (i.e., a continuous map from $[a,b]$ to $X$).  We say that it is absolutely continuous if there exists a $L^1$-function $g$ such that $d(\ww_t,\ww_s) \le \int_s^t g(r) \,\rd r$ for any  $a \le s \le t \le b$. 
 Moreover, the curve is said to have finite $p$-energy if $g \in L^p([a,b],\R)$.
The metric derivative $|\ww_t'|$ of $\{\ww_t\}_{t \in [a, b]}$ at the time point $t$ is defined by $|\ww'_t| : = \lim_{\d \to 0}|\d|^{-1} d(\ww_{t + \d},\ww_t)$, if the limit exists.
It can be shown \cite[Theorem 1.1.2]{ambrosio2008gradient} that for an absolutely continuous curve $\ww_t$, the metric derivative $|\ww'_t|$ is well-defined for a.e.\,$t \in [a,b]$ and satisfies $|\ww'_t| \le g(t)$.

The length ${\rm L}(\ww_t)$ of an absolutely continuous curve $\{\ww_t\}_{t \in [a,b]}$ is defined as ${\rm  L}(\ww_t) = \int_a^b |\ww'_t| \,\rd t$, which is invariant 
with respect to the reparameterization. Then, $(X,d)$ is a geodesic space if for any $x,y \in X$, there holds 
\begin{align} \label{eq3}
    d(x,y) = \min \{{\rm L}(\ww_t)\,; \  \{\ww_t\}_{t \in [0,1]}\ \text{is absolutely continuous with}\ \ww(0) = x\,, \ww(1) = y \}\,,
\end{align}
where the minimizer exists and is called the (minimizing) geodesic between $x$ and $y$. Recall
\cite[Lemma 1.1.4]{ambrosio2008gradient} that any absolutely continuous curve can be reparameterized as a Lipschitz one with constant metric derivative $|\ww'_t| = {\rm L}(\ww_t)$ a.e.. Hence, we can always assume that the geodesic is constant-speed (i.e., $|\ww_t'|$ is constant a.e.). Then, it is clear from definition \eqref{eq3} that a curve $\{\ww_t\}_{t \in [0,1]}$ is a constant-speed geodesic if and only if it satisfies $d(\ww_s,\ww_t) = |t -s |d (\ww_0,\ww_1)$ for any $0 < s < t < 1$. 

From the above concepts, we see that for our purpose, a key step is to characterize the absolutely continuous curves in the metric space $(\M(\Omega, \S^n_+),{\rm WB}_{\Lad})$, which is given by the following theorem extended from \cite[Theorem 5.17]{dolbeault2009new}.

\begin{theorem} \label{prop:for_geodesic}
A curve $\{\ms{G}_t\}_{t \in [a,b]}$, $b > a > 0$, is absolutely continuous with respect to 
the metric ${\rm WB}_{\Lad}$ if and only if  there exists $(\ms{q},\ms{R}) \in \M(Q,\R^{n \t k} \t \MM^n)$ such that 
 $\mu = (\ms{G},\ms{q},\ms{R}) \in 
 \w{\ce}([a,b];\ms{G}_0,\ms{G}_1)$  and 
\begin{equation} \label{eq:finite_length}
    \int_a^b \jj_{\Lad,\Omega}(\mu_t)^{1/2}\, \rd t < +  \infty \,.   
\end{equation}
In this case, we have the metric derivative $|\ms{G}_t'|$ satisfying 
\begin{equation} \label{eq:meder}
    |\ms{G}_t'| \le \jj_{\Lad,\Omega}(\mu_t)^{1/2}\q \text{for}\ a.e.\,t\in[a,b]\,,
\end{equation}
and there exists unique $(\ms{q_*}, \ms{R_*})$ such that the equality in \eqref{eq:meder}
holds a.e., where the uniqueness is in the sense of equivalence class: $\ms{(q,R)} \sim \ms{(q',R')}$ if and only if $\jj_{\Lad,Q_a^b}(\ms{(G,q-q',R-R')}) = 0 $. If $\ms{G}_t$ has finite $2$-energy, then $(\ms{q_*}, \ms{R_*}) = 
(\ms{G}u_*, \ms{G}W_*)$ with the $L^2_{\ms{G},\Lad}$-field $(u_*,W_*)$ given in 
Theorem \ref{thm:dual_minmax}.  
\end{theorem}

\begin{remark} \label{rem_ac}
As a corollary of Theorem \ref{prop:for_geodesic}, we have that $\mc{AC}([0,1];\ms{G}_0,\ms{G}_1)$ in \eqref{def:acce} is nothing else than the set of absolutely continuous curves with finite $2$-energy. 
\end{remark}

\begin{proof}
It suffices to consider the case $[a,b] = [0,1]$. 
We first consider the trivial \emph{if} part. For $\mu \in \w{\ce}([0,1];\ms{G}_0,\ms{G}_1)$ with the property \eqref{eq:finite_length}, it follows from \eqref{eq:length_energy} that 
\begin{align*} 
   {\rm WB}_{\Lad}(\ms{G}_s,\ms{G}_t) \le \int_s^t \jj_{\Lad,\Omega}(\mu_\tau)^{1/2}\, \rd \tau \q \forall 0 \le s\le t \le 1\,,
\end{align*}
which, by definition, readily implies that $\{\ms{G}_t\}_{t \in [0,1]}$ is absolutely continuous and 
\eqref{eq:meder} holds. 
We now consider the \emph{only if} part. Let $\{\ms{G}_t\}_{t \in [0,1]}$ be an absolutely continuous curve, which, by reparameterization, can be further assumed to be Lipschitz with the Lipschitz constant denoted by ${\rm Lip}(\ms{G}_t)$. We will approximate it by piecewise constant-speed curves. We fix an integer $N \in \NN$ with the step size $\tau = 2^{-N} $. Let $\{\mu_t^{k,N}\}_{t \in [(k-1)\tau,k \tau]}$ be a minimizer to \eqref{eq:change_time} with $[a,b] = [(k - 1)\tau, k \tau]$, which satisfies 
\begin{align} \label{auxeq_3201}
    \tau^{1/2} \jj_{\Lad,\Omega}(\mu^{k,N}_t)^{1/2} = \tau^{-1/2}{\rm WB}_{\Lad}(\ms{G}_{(k-1)\tau}, \ms{G}_{k \tau}) \le \Big(\int_{(k-1)\tau}^{k\tau} |\ms{G}_t'|^2 \,\rd t\Big)^{1/2}\,, \q a.e.\  t \in [(k-1)\tau,k\tau]\,,
\end{align}
by Lemma \ref{thm:geodesic_space} and the absolute continuity of $\ms{G}_t$.  
We glue the curves $\big\{\mu^{k,N}_t\big\}_{t \in [(k-1)\tau,k\tau]}$ with $k = 1,\ldots, 2^N$ and obtain a new one $\{\mu^N_t = (\ms{G}_t^N,\ms{q}_t^N,\ms{R}_t^N)\}_{t \in [0,1]} \in \ce_\infty([0,1];\ms{G}_0,\ms{G}_1)$.

Next, note that for any $(a,b) \subset [0,1]$, there exists $k_1^N, k_2^N \in \NN$ with $N$ large enough such that $[(k^N_1 + 1)\tau, (k_2^N - 1) \tau] \subset (a,b) \subset [k^N_1 \tau, k_2^N \tau]$. By squaring \eqref{auxeq_3201} and summing it from $k = k_1^N + 1$ to $ k = k_2^N$, there holds
\begin{align} \label{auxeq_3202}
\int_a^b \jj_{\Lad,\Omega}(\mu_t^N)\, \rd t \le \sum_{k = k^N_1 + 1}^{k^N_2} \int_{(k-1)\tau}^{k\tau} \jj_{\Lad,\Omega}(\mu^{k, N}_t)\, \rd t\le \int_{a}^{b} |\ms{G}_t'|^2\, \rd t + 2 \tau {\rm Lip}(\ms{G}_t)^2\,.
\end{align}
By taking $a = 0$, $b = 1$ in \eqref{auxeq_3202}, we observe that $\int_0^1 \jj_{\Lad,\Omega}(\mu_t^N) \,\rd t$ 
is uniformly bounded in $N$. By Proposition \ref{prop:compact}, up to a subsequence, 
$\{\mu^N_t\}_{t \in [0,1]}$ weak* converges to a measure $\w{\mu} = (\w{\ms{G}},\w{\ms{q}},\w{\ms{R}}) \in \ce_\infty([0,1],\ms{G}_0,\ms{G}_1)$. Moreover, 
it follows from \eqref{eq:lsc_jalpha} and \eqref{auxeq_3202} that, for $[a,b] \subset [0,1]$, 
\begin{align} \label{auxeq_3203}
    \int_a^b \jj_{\Lad,\Omega}(\w{\mu}_t)\, \rd t \le \liminf_{N \to +\infty} \int_a^b \jj_{\Lad,\Omega}(\mu_t^N)\, \rd t \le \int_{a}^{b} |\ms{G}_t'|^2\, \rd t \,.
    \end{align}

We now show $\w{\ms{G}}_t = \ms{G}_t$ for $0 \le t \le 1$.  
Note that for any $t \in [0,1]$, there exists a sequence of integers $k_N$ such that $s_N = k_N 2^{-N} \to t$ as $N \to \infty$, which implies that $\ms{G}^N_{s_N} = \ms{G}_{s_N}$ weak* converges to $\w{\ms{G}}_t$ by Proposition \ref{prop:compact}.  Meanwhile, $\ms{G}_{s_N}$ weak* converges to $\ms{G}_t$ by the continuity of 
 $\ms{G}_t$. We hence have 
   $\w{\ms{G}}_t = \ms{G}_t$. Then, it follows from \eqref{auxeq_3203} that
\begin{equation*}
    \jj_{\Lad,\Omega}(\w{\mu}_t) = \jj_{\Lad,\Omega}(\ms{G}_t,\w{\ms{q}}_t,\w{\ms{R}}_t) \le |\ms{G}_t'|^2\,,
\end{equation*} 
by Lebesgue differentiation theorem. The proof of the \emph{only if} direction is completed by noting that \eqref{eq:finite_length} and \eqref{eq:meder} are invariant with respect to the parameterization. The uniqueness of $(\ms{q}_*, \ms{R}_*)$ follows from the linearity of the continuity equation in the variable $(\ms{q},\ms{R})$ and the strict convexity of the $L^2_{\ms{G}}$-norm. 

We finally show that when $\ms{G}_t$ is absolutely continuous with finite $2$-energy, 
$\mu := (\ms{G}, \ms{G}u_*, \ms{G}W_*) \in \ce_\infty([0,1];\ms{G}_0,\ms{G}_1)$ satisfies $\jj_{\Lad,\Omega}(\mu_t)^{1/2} \le |\ms{G}_t'|$ for a.e.\,$t \in [0,1]$, where
$(u_*,W_*)$ is given in Theorem \ref{thm:dual_minmax} (i.e., the Riesz representation of $\w{l}_{\ms{G}}$ in $H_{\ms{G},\Lad}(\ms{D}^*)$).
Let $(a,b) \subset [0,1]$, and $\eta \in C_c^\infty((a,b))$ with $0 \le \eta \le 1$, and $\{(\ms{D}^* \Phi_n \Lad_1^2, \Phi_n \Lad_2^2)\}$ with $\Phi_n \in C^1(Q,\S^n)$ be a sequence approximating
$(u_*,  W_*)$. Then, by using \eqref{eq:continui_equation} and noting $\ms{D}^* (\eta^2 \Phi) = \eta^2 \ms{D}^* (\Phi)$, we have  
\begin{align} \label{auxeq_3204}
    &\normm{(\eta u_*, \eta W_*)}^2_{L^2_{\ms{G},\Lad}(Q)} = \lim_{n \to + \infty} \big\l (\eta^2 u_*, \eta^2 W_*), (\ms{D}^*\Phi_n \Lad_1^2, \Phi_n \Lad_2^2) \big\r_{L^2_{\ms{G},\Lad}(Q)} = \lim_{n \to + \infty} l_{\ms{G}}(\eta^2 \Phi_n)\,.
\end{align}
By \emph{only if} part proved above, there exists some $\ms{(q,R)}$ such that
\begin{align} \label{auxeq_3205}
    \left|l_{\ms{G}}(\eta^2 \Phi_n)\right| &\le \normm{(G^\dag q,  G^\dag R)}_{L^2_{\ms{G},\Lad}(Q_a^b)} \normm{(\ms{D}^* \eta^2 \Phi_n, \eta^2 \Phi_n)}_{L^2_{\ms{G},\Lad}(Q_a^b)} \notag \\
    & \le \Big(\int_a^b |\ms{G}_t'|^2 \, \rd t \Big)^{1/2}  \normm{( \ms{D}^* \Phi_n, \Phi_n)}_{L^2_{\ms{G},\Lad}(Q_a^b)}\,.
\end{align}
Combining \eqref{auxeq_3204} with \eqref{auxeq_3205} and letting $\eta$ approximate $\chi_{[a,b]}$, we obtain 
\begin{equation} \label{eq5}
    \normm{( u_*, W_*)}_{L^2_{\ms{G},\Lad}(Q_a^b)} \le \Big(\int_a^b |\ms{G}_t'|^2 \,\rd t \Big)^{1/2}\,.
\end{equation}
Then, by Lebesgue differentiation theorem again, the inequality \eqref{eq5} gives the desired $\jj_{\Lad,\Omega}(\mu_t)^{1/2} \le |\ms{G}_t'|$ for the measure $\mu = (\ms{G}, \ms{G}u_*, \ms{G}W_*)$. The proof is complete.
\end{proof}

From Lemma \ref{thm:geodesic_space} and Theorem \ref{prop:for_geodesic}, we have 
\begin{align} \label{eq:lengthspace}
    {\rm WB}_\Lad(\ms{G}_0,\ms{G}_1) & = \inf_{\ms{G}} \inf_{\ms{(q,R)}} \Big\{\int_0^1 \jj_{\Lad,\Omega}(\mu_t)^{1/2} \,\rd t\,;\ \mu = \ms{(G,q,R)} \in \w{\ce}([0,1];\ms{G}_0,\ms{G}_1) \Big\}\notag \\
    & = \inf_{\ms{G}} \Big\{ \int_0^1 |\ms{G}_t'| \, \rd t\,; \ \{\ms{G}\}_{t \in [0,1]} \ \text{is absolutely continuous with}\ \ms{G}_t|_{t = 0} = \ms{G}_0\,, \ms{G}_t|_{t = 1} = \ms{G}_1  \Big\}\,.
\end{align}
Note that if $\{\mu_t\}_{t \in [0,1]} \in \ce_\infty([0,1];\ms{G}_0,\ms{G}_1)$ minimizes
\eqref{eq:distance}, then for any $0 \le a <  b \le 1$, $\{\mu_t\}_{t \in [a,b]}$ is a minimizer to \eqref{eq:change_time} with $\ms{G}_0 = \ms{G}_t|_{t = a}$ and $\ms{G}_1 = \ms{G}_t|_{t = b}$. Recalling the constant-speed property \eqref{eq:cons_energy} of the minimizer $\mu = \ms{(G,q,R)}$, we readily see that the associated $\{\ms{G}_t\}_{t \in [0,1]}$ is the desired constant-speed geodesic: 
\begin{align} \label{eq:cons_spped}
    {\rm WB}_{\Lad}(\ms{G}_s,\ms{G}_t) =  |t - s| {\rm WB}_{\Lad}(\ms{G}_0,\ms{G}_1)\,, \q \forall 0 \le s \le t \le 1\,.
\end{align} 
It allows us to conclude that the $\inf$ in \eqref{eq:lengthspace} is attained, and the  main result follows. 

\begin{corollary} \label{cor:geo_me}
    $(\M(\Omega,\S^n_+), {\rm WB}_{\Lad})$ is a geodesic space. The constant-speed geodesic connecting $\ms{G}_0, \ms{G}_1 \in \M(\Omega,\S^n_+)$ is given by the minimizer to \eqref{eq:distance}. 
\end{corollary}

Another important application of Theorem \ref{prop:for_geodesic} is that we can view the set of  $\S^n_+$-valued measures as a pseudo-Riemannian manifold, following \cite[Proposition 8.4.5]{ambrosio2008gradient}. We define the tangent space at each $\ms{G} \in \mc{M}(\Omega,\S_+^n)$ by 
\begin{align} \label{def:tang}
    \Tan(\ms{G}): = \big\{&\ms{(q,R)} \in \mc{M}(\Omega,\R^{n \t k} \t \MM_n)\,; \ \jj_{\Lad,\Omega}(\ms{\mu}) < \infty \ \, \text{with}\ \, \mu = \ms{(G,q,R)} \in \mc{M}(\Omega, \xx); \notag \\ 
    & \jj_{\Lad,\Omega}(\mu) \le  \jj_{\Lad,\Omega}(\ms{(G,q+ \h{q}, R + \h{R})})\,,\ \forall \ms{(\h{q}, \h{R})} \ \text{satisfying} \ \ms{D} \h{\ms{q}} =\h{ \ms{R}}^{\rm sym}
   \big\}\,.
\end{align}
From Theorem \ref{prop:for_geodesic}, we have that among all the measures $\ms{(q,R)}$ 
generating $\{\ms{G}_t\}_{t \in [0,1]}$ by the continuity equation, there is a unique one $\ms{(q_*,R_*)}$ with minimal $\jj_{\Lad,\Omega}(\mu_t)$ given by $|\ms{G}_t'|$ for a.e.\,$t \in [0,1]$, that is,  $(\ms{q}_{*,t},\ms{R}_{*,t}) \in \Tan(\ms{G}_t)$ a.e.\,by \eqref{def:tang}. We also introduce the space $\Tan_{field}(\ms{G})$ similar to  $H_{\ms{G},\Lad}(\ms{D}^*)$ \eqref{def:hladd}:
$$
 \Tan_{field}(\ms{G}) = \overline{\left\{(\ms{D}^* \Phi \Lad_1^2, \Phi \Lad_2^2)\,;\ \Phi \in C^1(\Omega,\S^n)\right\}}^{\norm{\dd}_{L^2_{\ms{G},\Lad}(\Omega)}}\,.
$$
Then, similarly to the argument for Theorem \ref{thm:dual_minmax}, the tangent space $\Tan(\ms{G})$ can be characterized as follows:
\begin{equation} \label{eq6}
    \ms{(q,R)} \in \Tan(\ms{G})\q \text{if and only if}\q \ms{(q,R)}= \ms{G}(u,W)\  \text{with}\  (u,W) \in \Tan_{field}(\ms{G})\,.
\end{equation}
We summarize the above discussions in the following corollary, which provides a Riemannian interpretation of the transport distance ${\rm WB}_\Lad(\dd,\dd)$. 

\begin{corollary} \label{coro:riemann}
Let $\{\ms{G}_t\}_{t \in [0,1]}$ be an absolutely continuous curve in $(\M(\Omega,\S_+^n),{\rm WB}_\Lad)$ and $\{(\ms{q}_t,\ms{R}_t)\}_{t \in [0,1]}$ be the family of measures in 
$\M(\Omega,\R^{n \t k} \t \MM^n)$
such that $\mu = (\ms{G},\ms{q},\ms{R}) \in \ce ([0,1];\ms{G}_0,\ms{G}_1)$ and $ \jj_{\Lad,\Omega}(\mu_t)$ is finite a.e..
Then $|\ms{G}_t'| = \jj_{\Lad,\Omega}(\mu_t)$
holds for a.e.\,$t \in [0,1]$ if and only if $(\ms{q}_t, \ms{R}_t) \in \Tan(\ms{G}_t)$ a.e.,
where $\Tan(\ms{G})$ is defined in \eqref{def:tang}
and characterized by \eqref{eq6}.
Moreover, for absolutely continuous $\ms{G}_t$ with finite 2-energy (i.e., $\ms{G} \in \mc{AC}([0,1];\ms{G}_0,\ms{G}_1)$), let $(u_*,W_*)$ be the unique minimizer to \eqref{eq:sub_opt_2}. Then, there holds $(u_{*,t},W_{*,t}) \in \Tan_{field}(\ms{G}_t)$ a.e..
\end{corollary}

\section{Cone space and spherical distance} \label{seccone}

In this section, we discuss the conic structure of our weighted transport distance ${\rm WB}_\Lad$, which extends the results in \cite[Section 4]{brenier2020optimal} and \cite[Section 5]{monsaingeon2020schr}. The starting point is a spherical  distance associated with ${\rm WB}_\Lad$: 
 \begin{align} \label{def:sphdist}
        {\rm SWB}_{\Lad}^2(\ms{G}_0,\ms{G}_1) = \inf\big\{\jj_{\Lad,Q}(\mu)\,;\ \mu \in \w{\ce}([0,1];\ms{G}_0,\ms{G}_1)\,, \tr_\Lad\ms{G}_t(\Omega) = 1\big\}\,,\q \text{for}\ \ms{G}_0,\ms{G}_1 \in \mc{M}_1\,,
\end{align}
where $\tr_\Lad(X) = \tr\big(\w{\Lad}_2^{-1}X\w{\Lad}_2^{-1}\big)$ with $\w{\Lad}_2 = n \Lad_2/\tr(\Lad_2)$ is the scaled trace and 
\begin{equation} \label{def:normeasure}
    \mc{M}_1: = \{\ms{G} \in \mc{M}(\Omega,\S_+^n)\,;\ \tr_\Lad \ms{G}(\Omega) = 1\}\,.
\end{equation}
We will prove that $(\M_1, {\rm SWB}_{\Lad})$ is a complete geodesic space and $(\M(\Omega,\S_+^n),{\rm WB}_\Lad)$ can be viewed as its metric cone. 
Let us first recall some basic concepts \cite{burago2022course,laschos2019geometric}. We consider a metric space $(X,d_X)$ with  diameter ${\rm diam}(X) = \sup_{x,y\in X}d_X(x,y) \le \pi$. The associated cone is defined by $\mathfrak{C}(X) := X \t [0,\infty) \backslash X \times \{0\}$ with the metric 
\begin{equation} \label{eq:conedistance}
d^2_{\mathfrak{C}(X)}([x_0,r_0],[x_1,r_1]) := r_0^2 + r_1^2 - 2 r_0 r_1 \cos(d_X(x_0,x_1))\,,
\end{equation}
where a point in $\mathfrak{C}(X)$ is of the form $[x,r]$ with $x \in X$ and $r \ge 0$ and satisfies  the equivalence relation $[x_0,0] \sim [x_1,0]$. It can be proved that for $x_0, x_1 \in X$ with $0 < d_X(x_0, x_1) < \pi$ and $r_0,r_1 > 0$, there is one-to-one correspondence between the geodesics for $d_{\cc(X)}([x_0,r_0],[x_1,r_1])$ and for $d_{X}(x_0,x_1)$; see \cite[Theorem 2.6]{laschos2019geometric}. In particular, we have the following useful lemmas from \cite[Lemma 4.4]{brenier2020optimal}
and \cite[Theorem 2.2]{laschos2019geometric}, respectively. 
\begin{lemma}\label{lemcone1}  
If $X$ is a length space, then the distance $d_X(x_0,x_1)$ can be characterized by
\begin{align*}
    d_X(x_0,x_1) = \inf\Big\{\int_0^1 \big|[x_t,1]'\big|_{\cc(X)}\,\rd t\,;\ [x_t,1]\ \text{is absolutely continuous and connects}\ [x_0,1]\ \text{and}\ [x_1,1]\Big\}\,,
\end{align*}
where $|[x_t,1]'|_{\cc(X)}$ is the metric derivative in the space $(\cc(X),d_{\cc(X)})$. 
\end{lemma}


\begin{lemma} \label{lemcone3}
Let $\cc(X)$ be the cone as above and $(\cc(X), d)$ be a metric space for some metric $d$. If there holds 
\begin{equation} \label{eq:scaling}
    d^2([x_0,r_0], [x_1, r_1]) = r_0 r_1 d^2([x_0,1],[x_1,1]) + (r_0 - r_1)^2\,,
\end{equation}
and $0 < d^2([x_0,1],[x_1,1]) \le 4$ for $x_0 \neq x_1$, then $d_X(x_0,x_1): = \arccos(1 - d^2([x_0,1], [x_1, 1])/2)$ is a metric on $X$ such that \eqref{eq:conedistance} holds, equivalently, $(\cc(X), d)$ is a metric cone over $(X,d_X)$. 
\end{lemma}

We are now ready to consider the conic properties of $(\M(\Omega,\S^n_+), {\rm WB}_{\Lad})$. For this, we set $r := \sqrt{\tr_\Lad (\ms{G}(\Omega))} \ge 0$ for a measure $\ms{G} \in \mc{M}(\Omega,\S_+^n)$ and identify $\ms{G}$ with $[\ms{G}/r^2,r] \in \cc(\mc{M}_1)$. 

\begin{proposition} \label{propconee}
Suppose there holds $\ms{D}^*(\Lad_2^{-2}) = 0$ and let $c := \sqrt{2}n/\tr(\Lad_2)$. We have that
$(\M(\Omega,\S_+^n),{\rm WB}_\Lad/c)$ is a metric cone over $(\mc{M}_1, d)$ for some metric $d$.    
\end{proposition}
\begin{proof}
We note from \eqref{auxeqscaling} in the proof of 
Lemma \ref{lem:com_kb_and_b} that
\begin{equation*}
        {\rm WB}^2_{\Lad}(\ms{G}_0,\ms{G}_1) \le 
2 \int_{\Omega} \Big\lVert \Big(\sqrt {G_1} - \sqrt{G_0}\Big) \Lad_2^{-1} \Big\lVert_\ff^2\ \rd \lad \le 4 \big(n/\tr(\Lad_2)\big)^2 \big(\tr_{\Lad} \ms{G}_0(\Omega) + \tr_{\Lad}\ms{G}_1(\Omega)\big)\,,
\end{equation*}
which yields ${\rm WB}_{\Lad}^2(\ms{G}_0,\ms{G}_1) \le 4 c^2$ for $\ms{G}_0,\ms{G}_1 \in \mc{M}_1$. By Lemma \ref{lemcone3}, it suffices to check the scaling property \eqref{eq:scaling}:
\begin{equation}\label{eqsscalingeb}
     {\rm WB}_{\Lad}^2(r_0^2 \ms{G}_0, r_1^2 \ms{G}_1)/c^2 = r_0 r_1 {\rm WB}_{\Lad}^2(\ms{G}_0, \ms{G}_1)/c^2 + (r_0 - r_1)^2\,,
\end{equation}
for $\ms{G}_0, \ms{G}_1 \in \mc{M}_1$ and $r_0,r_1 \ge 0$ to show that $(\M(\Omega,\S_+^n),{\rm WB}_\Lad/c)$ is a 
metric cone. Note that \eqref{eqsscalingeb} for the case of $r_0 = 0$ or $r_1 = 0$ follows from Proposition \ref{prop:inflatingmeaure}. Thus, 
we can assume $r_0, r_1 > 0$. Let $\{\mu_t = (\ms{G}_t,\ms{q}_t,\ms{R}_t)\}_{t \in [0,1]} \in \w{\ce}([0,1];\ms{G}_0,\ms{G}_1)$ be an admissible 
curve. We define scalar functions $b(t) = r_0 + (r_1 - r_0)t$ and $a(t) := t r_1 /b(t)$. It is clear that $a(t)$ is strictly increasing with  inverse denoted by $t(a)$. We then define 
$\w{\ms{G}}_t = b(t)^2\ms{G}_{a((t)}$ with 
\begin{equation*}
    \w{\ms{q}}_t = a'(t)b(t)^2 \ms{q}_{a(t)}\,,\q \w{\ms{R}}_t = a'(t)b(t)^2 \ms{R}_{a(t)} + 2b(t)(r_1 - r_0)\ms{G}_{a(t)}\,,
\end{equation*}
which satisfies the continuity equation with endpoints $r_0^2 \ms{G}_0$ and $r_1^2 \ms{G}_1$. We now compute 
\begin{align}  \label{auxeqcone1} \jj_{\Lad,Q}\big(\w{\ms{G}},\w{\ms{q}},\w{\ms{R}}\big) = & \int_0^1 a'(t(a))\,  b(t(a))^2  \jj_{\Lad,\Omega}(\ms{G}_a, \ms{q}_a, \ms{R}_a) \, \rd a 
 + c^2 (r_1 - r_0)^2 \int_0^1 \tr_\Lad \ms{G}_{a(t)}(\Omega)\, \rd t \notag \\
 & + c^2 \int_0^1  b(t(a))\, (r_1 - r_0) \tr_\Lad \ms{R}_{a}(\Omega)\, \rd a\,.
\end{align}
 The last two terms in \eqref{auxeqcone1} can be  simplified by \eqref{eq:weak_ctneq} on $[0,1]$ with test function $\Phi_s = b(t(s)) \,\Lad_2^{-2}$: 
 \begin{equation*}
     \int_{0}^1 t'(a) (r_1 - r_0) \tr_\Lad \ms{G}_a(\Omega) +   b(t(a))   \tr_\Lad \ms{R}_a(\Omega) \,\rd a = r_1 \tr_\Lad \ms{G}_1(\Omega) - r_0 \tr_\Lad \ms{G}_0(\Omega)\,,
 \end{equation*}
 which implies, thanks to $\tr_\Lad \ms{G}_0(\Omega) = \tr_\Lad \ms{G}_1(\Omega) = 1$,
  \begin{equation} \label{auxeqcone2} 
     \int_{0}^1  (r_1 - r_0)^2 \tr_\Lad \ms{G}_{a(t)}(\Omega)\, \rd t  +  \int_0^1  b(t(a)) (r_1 - r_0)  \tr_\Lad \ms{R}_a(\Omega) \,\rd a = (r_1 - r_0)^2\,.
 \end{equation}
Therefore, by noting $a'(t) b(t)^2 = r_0 r_1$ and using \eqref{auxeqcone2}, it follows that 
\begin{align*}
\jj_{\Lad,Q}\big(\w{\ms{G}},\w{\ms{q}},\w{\ms{R}}\big) =  r_0 r_1 \int_0^1  \jj_{\Lad,\Omega}(\ms{G}_a, \ms{q}_a, \ms{R}_a) \, \rd a 
 + c^2 (r_1 - r_0)^2\,,
\end{align*}
which readily gives ${\rm WB}_{\Lad}^2(r_0^2 \ms{G}_0, r_1^2 \ms{G}_1)/c^2 \le r_0 r_1 {\rm WB}_{\Lad}^2(\ms{G}_0, \ms{G}_1)/c^2 + (r_0 - r_1)^2$. The other direction can be proved similarly and the proof is complete. 
\end{proof}

\begin{theorem} \label{thmcone}
Suppose there holds $\ms{D}^*(\Lad_2^{-2}) = 0$ and let $c := \sqrt{2}n/\tr(\Lad_2)$. Then, $(\M(\Omega,\S_+^n),{\rm WB}_\Lad/c)$ is a metric cone over $(\M_1, {\rm SWB}_\Lad/c)$, namely, for $\ms{G}_0, \ms{G}_1 \in \mc{M}_1$ and $r_0,r_1 \ge 0$, 
\begin{equation} \label{auxeqcone3}
     {\rm WB}_{\Lad}^2(r_0^2 \ms{G}_0, r_1^2 \ms{G}_1)/c^2 = 
     r_0^2 + r_1^2 - 2 r_0 r_1 \cos({\rm SWB}_\Lad(\ms{G}_0, \ms{G}_1)/c)\,,
\end{equation}
and $(\M_1, {\rm SWB}_\Lad/c)$ is a complete geodesic space with ${\rm diam}(\M_1) \le \pi$. 
\end{theorem}

\begin{proof}
We first show that the metric $d$ on $\mc{M}_1$ in Proposition \ref{propconee} is given by ${\rm SWB}_\Lad/c$. By Corollary \ref{cor:geo_me} 
and \cite[Corollary 5.11]{bridson2013metric}, we have that $(\mc{M}_1,d)$ is a geodesic space, which, by Lemma \ref{lemcone1}, gives, for $\ms{G}_0,\ms{G}_1 \in \mc{M}_1$, 
\begin{equation*}
    d(\ms{G}_0,\ms{G}_1) = \inf\Big\{\int_0^1 |\ms{G}_t'|\, \rd t\,;\ \ms{G}_t\ \text{is absolutely continuous in} \ (\M(\Omega,\S_+^n),{\rm WB}_\Lad/c)\ \text{with}\ \ms{G}_t \in \mc{M}_1  \Big\}\,.
\end{equation*}
It then follows from Theorem \ref{prop:for_geodesic} and definition \eqref{def:sphdist} that $d(\ms{G}_0,\ms{G}_1) = {\rm SWB}_{\Lad}(\ms{G}_0,\ms{G}_1)/c$ and hence \eqref{auxeqcone3} holds. 
Recalling ${\rm WB}_{\Lad}^2(\ms{G}_0,\ms{G}_1)/c^2 \le 4$ for $\ms{G}_0,\ms{G}_1 \in \mc{M}_1$, \eqref{auxeqcone3} gives $0 \le {\rm SWB}_\Lad(\ms{G}_0, \ms{G}_1)/c \le \pi$. 
Finally, for the completeness of $(\M_1, {\rm SWB}_\Lad/c)$, it suffices to note that ${\rm SWB}_\Lad$ and ${\rm WB}_\Lad$ are topologically equivalent on $\M_1$, again by \eqref{auxeqcone3}, and $\M_1$ is a closed set in $(\M(\Omega,\S_+^n),{\rm WB}_\Lad)$ by Proposition \ref{prop:metric_prop}.
\end{proof}

\section{Example and discussion} \label{sec:example_model}

In this section, we detail the connections between our model \eqref{eq:distance} and the existing ones.

\medskip

\noindent \textbf{Example} (Kantorovich-Bures metric \cite{brenier2020optimal}).  
We set the dimension parameters $n = m =d$ and $k = 1$ and the weighted matrices $\Lad_i = I$ for $i = 1, 2$ in \eqref{eq:closed_convex_set} and consider the differential operator $\ms{D} = \na_s$ for the continuity equation \eqref{eq:weak_ctneq}, where $ \na_s$ is the symmetric gradient defined by $\na_s(q) = \frac{1}{2}(\na q + (\na q)^{{\rm T}})$ for a smooth vector field $q \in C_c^\infty(\R^d,\R^d)$. Then, \eqref{eq:distance} gives the convex formulation of the Kantorovich-Bures metric $d_{KB}$ on $\M(\Omega, \S_+^d)$ \cite[Definition 2.1]{brenier2020optimal}: 
\begin{multline} 
{\rm WB}^2_{(I,I)}(\ms{G}_0,\ms{G}_1) = \frac{1}{2}d^2_{KB}(\ms{G}_0,\ms{G}_1) = \inf\big\{\jj_{\Lad,Q}(\mu)  \,;\ \mu = \ms{(G,q,R)} \in \M(Q,\xx)\ \text{satisfies} \notag\\ 
 \p_t \ms{G} = \{ - \na \ms{q}_t + \ms{R}_t\}^{\sym} \ \text{with}\  \ms{G}_t|_{t = 0} = \ms{G}_0\,,\ \ms{G}_t|_{t = 1} = \ms{G}_1
\big\}\,,\tag{$\mathcal{P}_{\rm WB}$}\label{def:kan_bure}
\end{multline} 
for $\ms{G}_0, \ms{G}_1 \in \M(\Omega,\S^d_+)$, where $\jj_{\Lad, Q}(\mu)$ with $\Lad = (I,I)$ is given by \eqref{rep:enerycost}:
\begin{equation*}
\jj_{\Lad,Q}(\mu) = \frac{1}{2} \norm{G^\dag q }^2_{L^2_{\ms{G}}(Q)} + \frac{1}{2} \norm{G^\dag R}^2_{L^2_{\ms{G}}(Q)}\,. 
\end{equation*}

\noindent \textbf{Example} (Wasserstein-Fisher-Rao metric \cite{chizat2018interpolating,kondratyev2016new,liero2016optimal}). If we set $n = m = 1$, $k = d$, and $\Lad_1 = \sqrt{\alpha} I$, $\Lad_2 = \sqrt{\beta}I$ with $\alpha, \beta > 0$, and consider the differential operator $\ms{D} = \ddiv$, then \eqref{eq:distance} gives the Wasserstein-Fisher-Rao metric \cite[(3.1)]{liero2016optimal}: for given distributions $\rho_0, \rho_1 \in \M(\Omega,\R_+)$, 
\begin{multline} 
{\rm WFR}^2(\rho_0, \rho_1)  = \inf\Big\{\int_0^1 \int_\Omega \rho^{\dag}\Big(\frac{1}{2\alpha}|q|^2 + \frac{1}{2\beta} r^2\Big)\,\rd x\,\rd t \,;\ 
 \p_t \rho + \ddiv \,q = r \ \text{with}\  \rho_t|_{t = 0} = \rho_0\,,\ \rho_t|_{t = 1} = \rho_1 \Big\}\,.\tag{$\mathcal{P}_{\rm WFR}$} \label{def:wfr_metric}
\end{multline}

\noindent \textbf{Example} (Matricial interpolation distance \cite{chen2019interpolation}).
Let $N$ be a positive integer and $(\MM^n)^N$ denote the space of block-row vectors $(A_1,\ldots,A_N)$ with $A_i \in \MM^n$. The spaces $(\S^n)^N$ and $(\AA^n)^N$ are defined similarly. For $M \in (\MM^n)^N$, we define its component transpose by $M^t := (M_1^{\rm T},\ldots, M_N^{\rm T})$. We fix a sequence of symmetric matrices $\{L_k\}_{k=1}^N \subset \S^n$ and define the linear operator $\na_L : \S^n \to (\AA^n)^N$ by $(\na_L X)_k =  L_k X - XL_k$. We denote by $\na_L^*$ its dual operator with respect to the Frobenius inner product. 
We now let $k = n (d + N)$ and write $\ms{q} \in  \M(Q, \R^{n \t k})$ for $[\ms{q}_0, \ms{q}_1]$ with $\ms{q}_0 \in \M(Q, (\MM^n)^d)$ and $\ms{q}_1 \in \M(Q,(\MM^n)^N)$. With the above notions, we define  
$$
\ms{D}\,\ms{q} := \frac{1}{2}\ddiv (\ms{q}_0 + \ms{q}_0^t)  - \frac{1}{2}\na_L^* (\ms{q}_1 - \ms{q}_1^t)\,. 
$$  
Then, it is clear that \eqref{eq:distance} with weighted matrices $\Lad_i = I$ for $i = 1,2$ gives 
the model in \cite[(5.7a)--(5.7c)]{chen2019interpolation}:
\begin{multline} \label{def:chen19int}
       {\rm W}_{2,{\rm FR}}(\ms{G}_0, \ms{G}_1)^2 = \frac{1}{2} \inf \big\{ \norm{G^\dag q_0}^2_{L^2_{\ms{G}(Q)}} + \norm{G^\dag q_1}^2_{L^2_{\ms{G}(Q)}} + \norm{G^\dag R}^2_{L^2_{\ms{G}(Q)}}\,; \\
\p_t \ms{G} = - \frac{1}{2}\ddiv (\ms{q}_0 + \ms{q}_0^t) + \frac{1}{2}\na_L^* (\ms{q}_1 - \ms{q}_1^t) + \ms{R}^{{\rm sym}}\ \text{with}\ \ms{G}_t|_{t = 0} = \ms{G}_0\,,\  \ms{G}_t|_{t = 1} = \ms{G}_1 \big\}\,.
       \tag{$\mathcal{P}_{2,{\rm FR}}$} 
\end{multline}

\smallskip 

We next relate our model \eqref{eq:distance} to the matrix-valued optimal ballistic transport problems in \cite{brenier2018initial,vorotnikov2022partial}. As reviewed in the introduction, 
Brenier \cite{brenier2018initial} recently attempted to find the weak solution of the incompressible Euler equation on the domain $[0,T] \times \Omega \subset \R^{1 + d}$ (we omit the initial and boundary conditions for simplicity): 
\begin{equation} \label{euler0}
 \p_t v + \ddiv\,(v \otimes v)  + \na  p = 0\,,\q \ddiv\, v = 0\,,
\end{equation}
by  minimizing the kinetic energy $\int_0^T\int_\Omega |v(t,x)|^2 \, \rd x\, \rd t$, where $v$ is a $\R^n$-valued vector field and $p$ is a scalar function. It turns out that
this problem admits a concave maximization dual problem, to which the relaxed solution always exists under very light assumptions. Such an approach was extended by Vorotnikov \cite{vorotnikov2022partial} in an abstract functional analytic framework that includes a broad class of PDEs with quadratic nonlinearity as examples, such as the Hamilton-Jacobi equation, the template matching equation, and the multidimensional Camassa-Holm equation. More precisely, \cite{vorotnikov2022partial} considered the following abstract Euler equation on $[0,T] \times \Omega$\,:
\begin{equation} \label{eq:abseular}
    \p_t v = \ms{P} \circ \ms{L}\, (v \otimes v)\,,\q v(0,\dd) = v_0 \in \ms{P}(L^2(\Omega, \R^n))\,, 
\end{equation}
where $\ms{P}$ is an orthogonal projection and $\ms{L}: L^2(\Omega,\S^n) \to L^2(\Omega,\R^n)$ is a (closed densely defined) linear operator. One can see that
for $\ms{L} = - \ddiv$ and $\ms{P}$ being the Leray projection, the problem \eqref{eq:abseular} reduces to \eqref{euler0}.
The dual problem associated with the weak solution of \eqref{eq:abseular} with minimal kinetic energy reads as follows:
 \begin{align} \label{dual1}
 \sup \Big\{\int_0^T\int_\Omega v_0 \dd q  - \frac{1}{2} q \dd G^\dag q \ \rd x\, \rd t\,;\ \p_t G + 2 (\ms{L}^* \circ \ms{P})\,q = 0\ \text{with}\  G(T) = I \Big\}\,,
\end{align}
where $G$ and $q$ are $\S^n_{+}$-valued and $\R^n$-valued vector fields, respectively. Note that the Hamilton-Jacobi equation $\p_t \psi + \frac{1}{2} |\na \psi|^2 = 0$ can be reformulated as $\p_t v + \frac{1}{2} \na \tr(v \otimes v) = 0$ by letting $v = \na \psi$, which is a special case of \eqref{eq:abseular} with $\ms{P} = I$ and $\ms{L} = - \frac{1}{2} \na \tr$. The corresponding dual maximization problem is given by
 \begin{align} \label{dual2}
 \sup \Big\{-\int_\Omega \psi_0 \rho_0 \, \rd x - \frac{1}{2} \int_0^T\int_\Omega  \rho^\dag |q|^2 \, \rd x\, \rd t\,;\ \p_t \rho + \ddiv \, q = 0\ \text{with}\  \rho(T) = 1 \Big\}\,,
\end{align}
which closely relates to the ballistic transport problem \cite{barton2019dynamic}. In view of \eqref{dual1} and \eqref{dual2}, one may regard 
\begin{equation} \label{matrixctn}
\p_t G + 2 (\ms{L}^* \circ \ms{P})\,q = 0    
\end{equation}
as a matricial continuity equation, and our model \eqref{eq:abs_ctn_eq} can be hence viewed as an unbalanced variant of \eqref{matrixctn}. Then, the conservativity condition $\ms{D}^*(I) = 0$ for \eqref{matrixctn} is simply $\ms{P} \circ \ms{L}(I)  = 
0$, which has been used to guarantee the existence of a measure-valued solution to \eqref{dual1}; see \cite[Theorem 4.6]{vorotnikov2022partial}. Thanks to the above observations, one may expect that each meaningful choice of $\ms{L}$ and $\ms{P}$ in \cite[Section 6]{vorotnikov2022partial} can 
generate a reasonable distance \eqref{eq:distance} with $\ms{D} = 2 (\ms{L}^* \circ \ms{P})$. For instance, setting $n = d$, $\ms{P} = I$, and $\ms{L} = - \ddiv - \frac{1}{2} \na \tr$ in \eqref{eq:abseular} gives the template matching equation $\p_t v + \ddiv\, (v \otimes v) + \frac{1}{2} \na |v|^2 = 0$ and  
a distance \eqref{eq:distance} with $\ms{D} = 2 (\ms{L}^* \circ \ms{P})$:
\begin{equation}\label{eq:newmodel}
    \inf\big\{\jj_{\Lad,Q}\ms{(G,q,R)} \,;\ 
 \p_t \ms{G} + 2 \na_s \ms{q} + \ddiv \ms{q} I =  \ms{R}_t^{\sym} \ \text{with}\  \ms{G}_t|_{t = 0} = \ms{G}_0\,,\, \ms{G}_t|_{t = 1} = \ms{G}_1
\big\}\,.
\end{equation}

\begin{remark}
An important question is how to compare these matrix-valued optimal transport models \eqref{def:kan_bure}, \eqref{def:chen19int}, and \eqref{eq:newmodel} (as well as others in the literature), which requires a deeper theoretical analysis and is completely open, to the best of our knowledge. 
\end{remark}



\section{Concluding remarks}
We have proposed a general class of unbalanced matrix-valued optimal transport distances ${\rm WB}_{\Lad}(\dd,\dd)$ over the space $\M(\Omega,\S^n_+)$, called the weighted Wasserstein-Bures metric. The definition relies on a dynamic formulation and convex analysis.
We have shown that $\M(\Omega,\S^n_+)$ equipped with the metric ${\rm WB}_{\Lad}(\dd,\dd)$ is a complete geodesic space, and it can be viewed as a metric cone. In the follow-up work \cite{li2020general2}, 
we have considered the convergence of 
the discrete approximation of the transport model \eqref{eq:distance}. Our results provide a unified framework for unbalanced
transport distances on matrix-valued measures and directly apply to various existing models such as 
the Kantorovich-Bures distance \eqref{def:kan_bure}, the matricial interpolation distance \eqref{def:chen19int}, and the WFR one \eqref{def:wfr_metric}. Meanwhile, it paves the way for practical applications, in particular, diffusion tensor imaging as in \cite{chen2018efficient,ryu2018vector,peyre2019quantum}.

\titleformat{\section}{\bfseries}{\appendixname~\thesection .}{0.5em}{}
\titleformat{\subsection}{\normalfont\itshape}{\thesubsection.}{0.5em}{}

\appendices

\section{Auxiliary proofs} \label{app_B}



\begin{proof} [Proof of Lemma \ref{lem: conj_cvx}]
    For $\mu \in \M(\mathcal{X},\xx)$, by definition,      
    we have 
   $
        \iota_{C(\mathcal{X},\mathcal{O}_\Lad)}^*(\mu) = \sup\{\l \mu, \Xi\r_{\mathcal{X}}\,; \Xi \in C(\mathcal{X},\mathcal{O}_\Lad) \}\,.
  $
To show that the admissible set $C(\mathcal{X},\mathcal{O}_\Lad)$ can be relaxed to 
$L^\infty_{|\mu|} (\mathcal{X},\mathcal{O}_\Lad)$, it suffices to prove 
\begin{equation} \label{inproof:claim}
    \sup_{ \Xi \in L_{|\mu|}^\infty(\mathcal{X},\mathcal{O}_\Lad)}\l \mu, \Xi \r_{\mathcal{X}} \le  \sup_{ \Xi \in C(\mathcal{X},\mathcal{O}_\Lad)}\l \mu, \Xi\r_{\mathcal{X}}\,.
\end{equation}
For this, we consider an essentially bounded measurable field $\Xi \in L_{|\mu|}^\infty(\mathcal{X},\mathcal{O}_\Lad)$. Without loss of generality,
we assume that it is bounded by $\norm{\Xi}_\infty$ everywhere.  By Lusin's theorem, for any $\ep > 0$, there exists a continuous field with compact support $\w{\Xi}$ such that 
    \begin{equation} \label{eq:lusinset_est}
        |\mu|(\{x \in \mathcal{X}\,; \ \Xi(x) \neq \w{\Xi}(x) \}) \le \ep\,.
    \end{equation}
Define $\P_{\mathcal{O}_\Lad}$ as the $L^2$-projection from $\xx$ to the closed convex set $\mathcal{O}_\Lad$. By abuse of notation, we still denote by $\w{\Xi}$ the composite function $\P_{\mathcal{O}_\Lad} \circ \w{\Xi} \in C(\mathcal{X},\mathcal{O}_\Lad)$. It is clear that  $\norm{\w{\Xi}}_\infty \le  \norm{\Xi}_\infty$, and \eqref{eq:lusinset_est} still holds.
Then it follows that 
   $
        | \l \mu, \Xi\r_{\mathcal{X}} - \l \mu, \w{\Xi} \r_{\mathcal{X}}| \le  2 \ep \norm{\Xi}_\infty\,,
   $
    which further implies 
    \begin{equation*}
        \l \mu, \Xi\r_{\mathcal{X}} \le \l \mu, \w{\Xi}\r_{\mathcal{X}} + 2 \ep \norm{\Xi}_\infty \le  \sup_{ \Xi \in C(\mathcal{X},\mathcal{O}_\Lad)}\l \mu, \Xi\r_{\mathcal{X}} + 2 \ep \norm{\Xi}_\infty \,.
    \end{equation*}
    Since $\ep$ is arbitrary, we have proved the claim \eqref{inproof:claim}. Thus, we can take the pointwise $\sup$ in \eqref{eq:conj_gmu} and obtain the desired $\iota_{C(\mathcal{X},\mathcal{O}_\Lad)}^*(\mu) = \jj_{\Lad,\mathcal{X}}(\mu)$ by Proposition \ref{prop:subgrad_J}. Next, we characterize the subgradient $\p \jj_{\Lad,\mathcal{X}}(\mu)$. By Lemma \ref{lem:subgrad_relation}, we have 
      $\Xi \in \p \jj_{\Lad,\mathcal{X}}(\mu) \bigcap C(\mathcal{X},\xx)$ if and only if 
 $
        \l \mu, \Xi\r_{\mathcal{X}} =   \iota_{C(\mathcal{X},\mathcal{O}_\Lad)}(\Xi) + \jj_{\Lad,\mathcal{X}}(\mu) \,, 
$
    which yields $\Xi \in C(\mathcal{X},\mathcal{O}_\Lad)$ and 
    \begin{align} \label{auxeq:app_1}
        \int_{\mathcal{X}}  \mu_\lad \dd \Xi - J_\Lad(\mu_\lad)\, \rd \lad = 0  \,,
    \end{align}
     where $\lad$ is a reference measure such that $|\mu| \ll \lad$ and $\mu_\lad$ is the density of $\mu$. We  note from $J_\Lad = \iota^*_{\mathcal{O}_\Lad}$ and $\Xi(x) \in \mathcal{O}_\Lad$ that  $ \mu_\lad \dd \Xi - J_\Lad(\mu_\lad) \le 0$, $\lad$-a.e., where by \eqref{auxeq:app_1}, the equality actually holds $\lad$-a.e.. Then \eqref{eq:subg_conj_gmu} follows.
\end{proof}

\begin{proof} [Proof of Lemma \ref{thm:geodesic_space}]
It suffices to consider $[a,b] = [0,1]$. We denote by $\w{{\rm WB}}_\Lad$ the right-hand side of \eqref{eq:length_energy}. By H\"{o}lder's inequality and recalling \eqref{eq:distance} with the admissible set $\w{\ce}([0,1];\ms{G}_0,\ms{G}_1)$, we have $\w{{\rm WB}}_\Lad \le {\rm WB}_\Lad$. For the other direction, we consider $\{\mu_t\}_{t \in [0,1]} \in \w{\ce}([0,1];\ms{G}_0,\ms{G}_1)$ and reparameterize it by the $\ep$-arc length function $s = \ms{s}_\ep(t)$: 
\begin{equation*}
    s = \ms{s}_\ep(t) = \int_0^t \Big(\jj_{\Lad,\Omega}(\mu_\tau)^{1/2} + \ep \Big)\, \rd \tau: [0,1] \to [0, L(\mu_t) + \ep ]\,,
\end{equation*}
where $L(\mu_t): = \int_0^1 \jj_{\Lad,\Omega}(\mu_\tau)^{1/2}\,  \rd \tau $. It is clear that $\ms{s}_\ep(t)$ is strictly increasing and absolutely continuous and has an absolutely continuous inverse. Then,  by Lemma \ref{lemma:timescaling} and writing 
 $\w{\mu}^\ep_s = \mu_{\ms{s}_\ep^{-1}(s)}$ for short, we have 
     \begin{align} \label{auxeq:app_2}
        {\rm WB}^2_\Lad(\ms{G}_0,\ms{G}_1)  \le  (L(\mu_t) + \ep) \int^{L(\mu_t) + \ep}_0 \jj_{\Lad,\Omega}(\w{\mu}^\ep_s)\, \rd s = (L(\mu_t) + \ep) \int^1_0 \frac{\jj_{\Lad,\Omega}(\mu_t)}{\jj_{\Lad,\Omega}(\mu_t)^{1/2} + \ep}\, \rd t\,.
     \end{align}
     where the first inequality is by \eqref{eq:change_time} with $[a,b] = [0, L(\mu_t) + \ep]$. Letting $\ep \to 0$ in \eqref{auxeq:app_2}, we can find ${\rm WB}_{\Lad} \le \w{{\rm WB}}_\Lad$. 
     If we assume that $\mu$ minimizes \eqref{eq:distance}, we have 
    \begin{align*}
        {\rm WB}_\Lad(\ms{G}_0,\ms{G}_1) =  \Big(\int_0^1 \jj_{\Lad,\Omega}(\mu_t) \, \rd t \Big)^{1/2} \le \int_0^1 \jj_{\Lad,\Omega}(\mu_t)^{1/2}\, \rd t \,,
    \end{align*}  
     which implies that
     $\jj_{\Lad, \Omega}(\mu_t)$ is constant a.e..  
     Then \eqref{eq:cons_energy} immediately follows. 
\end{proof}


\begin{thebibliography}{10}

\bibitem{ambrosio2008gradient}
L.~Ambrosio, N.~Gigli, and G.~Savar{\'e}.
\newblock {\em Gradient flows: in metric spaces and in the space of probability
  measures}.
\newblock Springer Science \& Business Media, 2005.

\bibitem{ambrosio2004topics}
L.~Ambrosio and P.~Tilli.
\newblock {\em Topics on analysis in metric spaces}, volume~25.
\newblock Oxford University Press on Demand, 2004.

\bibitem{arjovsky2017wasserstein}
M.~Arjovsky, S.~Chintala, and L.~Bottou.
\newblock Wasserstein generative adversarial networks.
\newblock In {\em International conference on machine learning}, pages
  214--223. PMLR, 2017.

\bibitem{barbu2012convexity}
V.~Barbu and T.~Precupanu.
\newblock {\em Convexity and optimization in Banach spaces}.
\newblock Springer Science \& Business Media, 2012.

\bibitem{barton2019dynamic}
A.~Barton and N.~Ghoussoub.
\newblock Dynamic and stochastic propagation of the brenier optimal mass
  transport.
\newblock {\em European Journal of Applied Mathematics}, 30(6):1264--1299,
  2019.

\bibitem{bauschke2011convex}
H.~H. Bauschke, P.~L. Combettes, et~al.
\newblock {\em Convex analysis and monotone operator theory in Hilbert spaces},
  volume 408.
\newblock Springer, 2011.

\bibitem{benamou2003numerical}
J.-D. Benamou.
\newblock Numerical resolution of an “unbalanced” mass transport problem.
\newblock {\em ESAIM: Mathematical Modelling and Numerical
  Analysis-Mod{\'e}lisation Math{\'e}matique et Analyse Num{\'e}rique},
  37(5):851--868, 2003.

\bibitem{benamou2000computational}
J.-D. Benamou and Y.~Brenier.
\newblock A computational fluid mechanics solution to the monge-kantorovich
  mass transfer problem.
\newblock {\em Numerische Mathematik}, 84(3):375--393, 2000.

\bibitem{bhatia2013matrix}
R.~Bhatia.
\newblock {\em Matrix analysis}, volume 169.
\newblock Springer Science \& Business Media, 2013.

\bibitem{bhatia2019bures}
R.~Bhatia, T.~Jain, and Y.~Lim.
\newblock On the bures--wasserstein distance between positive definite
  matrices.
\newblock {\em Expositiones Mathematicae}, 37(2):165--191, 2019.

\bibitem{bouchitte2020convex}
G.~Bouchitt{\'e}.
\newblock Convex analysis and duality.
\newblock {\em arXiv preprint arXiv:2004.09330}, 2020.

\bibitem{bouchitte1988integral}
G.~Bouchitt{\'e} and M.~Valadier.
\newblock Integral representation of convex functionals on a space of measures.
\newblock {\em Journal of functional analysis}, 80(2):398--420, 1988.

\bibitem{brenier1991polar}
Y.~Brenier.
\newblock Polar factorization and monotone rearrangement of vector-valued
  functions.
\newblock {\em Communications on pure and applied mathematics}, 44(4):375--417,
  1991.

\bibitem{brenier2018initial}
Y.~Brenier.
\newblock The initial value problem for the euler equations of incompressible
  fluids viewed as a concave maximization problem.
\newblock {\em Communications in Mathematical Physics}, 364(2):579--605, 2018.

\bibitem{brenier2020optimal}
Y.~Brenier and D.~Vorotnikov.
\newblock On optimal transport of matrix-valued measures.
\newblock {\em SIAM Journal on Mathematical Analysis}, 52(3):2849--2873, 2020.

\bibitem{brezis2010functional}
H.~Brezis.
\newblock {\em Functional analysis, Sobolev spaces and partial differential
  equations}.
\newblock Springer Science \& Business Media, 2010.

\bibitem{bridson2013metric}
M.~R. Bridson and A.~Haefliger.
\newblock {\em Metric spaces of non-positive curvature}, volume 319.
\newblock Springer Science \& Business Media, 2013.

\bibitem{burago2022course}
D.~Burago, Y.~Burago, and S.~Ivanov.
\newblock {\em A course in metric geometry}, volume~33.
\newblock American Mathematical Society, 2022.

\bibitem{caffarelli2010free}
L.~A. Caffarelli and R.~J. McCann.
\newblock Free boundaries in optimal transport and monge-ampere obstacle
  problems.
\newblock {\em Annals of mathematics}, pages 673--730, 2010.

\bibitem{carlen2014analog}
E.~A. Carlen and J.~Maas.
\newblock An analog of the 2-wasserstein metric in non-commutative probability
  under which the fermionic fokker--planck equation is gradient flow for the
  entropy.
\newblock {\em Communications in mathematical physics}, 331(3):887--926, 2014.

\bibitem{carlen2017gradient}
E.~A. Carlen and J.~Maas.
\newblock Gradient flow and entropy inequalities for quantum markov semigroups
  with detailed balance.
\newblock {\em Journal of Functional Analysis}, 273(5):1810--1869, 2017.

\bibitem{chen2020matrix}
Y.~Chen, W.~Gangbo, T.~T. Georgiou, and A.~Tannenbaum.
\newblock On the matrix monge--kantorovich problem.
\newblock {\em European Journal of Applied Mathematics}, 31(4):574--600, 2020.

\bibitem{chen2017matrix}
Y.~Chen, T.~T. Georgiou, and A.~Tannenbaum.
\newblock Matrix optimal mass transport: a quantum mechanical approach.
\newblock {\em IEEE Transactions on Automatic Control}, 63(8):2612--2619, 2017.

\bibitem{chen2019interpolation}
Y.~Chen, T.~T. Georgiou, and A.~Tannenbaum.
\newblock Interpolation of matrices and matrix-valued densities: The unbalanced
  case.
\newblock {\em European Journal of Applied Mathematics}, 30(3):458--480, 2019.

\bibitem{chen2018efficient}
Y.~Chen, E.~Haber, K.~Yamamoto, T.~T. Georgiou, and A.~Tannenbaum.
\newblock An efficient algorithm for matrix-valued and vector-valued optimal
  mass transport.
\newblock {\em Journal of Scientific Computing}, 77:79--100, 2018.

\bibitem{chizat2018interpolating}
L.~Chizat, G.~Peyr{\'e}, B.~Schmitzer, and F.-X. Vialard.
\newblock An interpolating distance between optimal transport and fisher--rao
  metrics.
\newblock {\em Foundations of Computational Mathematics}, 18(1):1--44, 2018.

\bibitem{chizat2018scaling}
L.~Chizat, G.~Peyr{\'e}, B.~Schmitzer, and F.-X. Vialard.
\newblock Scaling algorithms for unbalanced optimal transport problems.
\newblock {\em Mathematics of Computation}, 87(314):2563--2609, 2018.

\bibitem{chizat2018unbalanced}
L.~Chizat, G.~Peyr{\'e}, B.~Schmitzer, and F.-X. Vialard.
\newblock Unbalanced optimal transport: Dynamic and kantorovich formulations.
\newblock {\em Journal of Functional Analysis}, 274(11):3090--3123, 2018.

\bibitem{cole2023quantum}
S.~Cole, M.~Eckstein, S.~Friedland, and K.~{\.Z}yczkowski.
\newblock On quantum optimal transport.
\newblock {\em Mathematical Physics, Analysis and Geometry}, 26(2):14, 2023.

\bibitem{datta2020relating}
N.~Datta and C.~Rouz{\'e}.
\newblock Relating relative entropy, optimal transport and fisher information:
  a quantum hwi inequality.
\newblock {\em Annales Henri Poincar{\'e}}, 21(7):2115--2150, 2020.

\bibitem{de2021quantum2}
G.~De~Palma, M.~Marvian, D.~Trevisan, and S.~Lloyd.
\newblock The quantum wasserstein distance of order 1.
\newblock {\em IEEE Transactions on Information Theory}, 67(10):6627--6643,
  2021.

\bibitem{de2021quantum}
G.~De~Palma and D.~Trevisan.
\newblock Quantum optimal transport with quantum channels.
\newblock In {\em Annales Henri Poincar{\'e}}, volume~22, pages 3199--3234.
  Springer, 2021.

\bibitem{dolbeault2009new}
J.~Dolbeault, B.~Nazaret, and G.~Savar{\'e}.
\newblock A new class of transport distances between measures.
\newblock {\em Calculus of Variations and Partial Differential Equations},
  34(2):193--231, 2009.

\bibitem{duran1997lpspace}
A.~J. Duran and P.~Lopez-Rodriguez.
\newblock The lpspace of a positive definite matrix of measures and density of
  matrix polynomials inl1.
\newblock {\em journal of approximation theory}, 90(2):299--318, 1997.

\bibitem{evans2015measure}
L.~C. Evans and R.~F. Gariepy.
\newblock {\em Measure theory and fine properties of functions}.
\newblock CRC press, 2015.

\bibitem{ferradans2014regularized}
S.~Ferradans, N.~Papadakis, G.~Peyr{\'e}, and J.-F. Aujol.
\newblock Regularized discrete optimal transport.
\newblock {\em SIAM Journal on Imaging Sciences}, 7(3):1853--1882, 2014.

\bibitem{figalli2010optimal}
A.~Figalli.
\newblock The optimal partial transport problem.
\newblock {\em Archive for rational mechanics and analysis}, 195(2):533--560,
  2010.

\bibitem{figalli2010new}
A.~Figalli and N.~Gigli.
\newblock A new transportation distance between non-negative measures, with
  applications to gradients flows with dirichlet boundary conditions.
\newblock {\em Journal de math{\'e}matiques pures et appliqu{\'e}es},
  94(2):107--130, 2010.

\bibitem{flaherty2013riemannian}
F.~Flaherty and M.~do~Carmo.
\newblock {\em Riemannian Geometry}.
\newblock Mathematics: Theory \& Applications. Birkh{\"a}user Boston, 2013.

\bibitem{fleissner2021minimizing}
F.~C. Fleissner.
\newblock A minimizing movement approach to a class of scalar
  reaction--diffusion equations.
\newblock {\em ESAIM: Control, Optimisation and Calculus of Variations}, 27:18,
  2021.

\bibitem{folland1999real}
G.~B. Folland.
\newblock {\em Real analysis: modern techniques and their applications},
  volume~40.
\newblock John Wiley \& Sons, 1999.

\bibitem{frogner2015learning}
C.~Frogner, C.~Zhang, H.~Mobahi, M.~Araya, and T.~A. Poggio.
\newblock Learning with a wasserstein loss.
\newblock {\em Advances in neural information processing systems}, 28, 2015.

\bibitem{gallouet2017jko}
T.~O. Gallou{\"e}t and L.~Monsaingeon.
\newblock A jko splitting scheme for kantorovich--fisher--rao gradient flows.
\newblock {\em SIAM Journal on Mathematical Analysis}, 49(2):1100--1130, 2017.

\bibitem{golse2016mean}
F.~Golse, C.~Mouhot, and T.~Paul.
\newblock On the mean field and classical limits of quantum mechanics.
\newblock {\em Communications in Mathematical Physics}, 343:165--205, 2016.

\bibitem{golse2017schrodinger}
F.~Golse and T.~Paul.
\newblock The schr{\"o}dinger equation in the mean-field and semiclassical
  regime.
\newblock {\em Archive for Rational Mechanics and Analysis}, 223:57--94, 2017.

\bibitem{golse2018wave}
F.~Golse and T.~Paul.
\newblock Wave packets and the quadratic monge--kantorovich distance in quantum
  mechanics.
\newblock {\em Comptes Rendus Mathematique}, 356(2):177--197, 2018.

\bibitem{gross1975hypercontractivity}
L.~Gross.
\newblock Hypercontractivity and logarithmic sobolev inequalities for the
  clifford-dirichlet form.
\newblock {\em Duke Mathematical Journal}, 42(3):383--396, 1975.

\bibitem{guittet2002extended}
K.~Guittet.
\newblock Extended kantorovich norms: a tool for optimization.
\newblock {\em Technical Report 4402}, 2002.

\bibitem{hanin1999extension}
L.~G. Hanin.
\newblock An extension of the kantorovich norm.
\newblock {\em Contemporary Mathematics}, 226:113--130, 1999.

\bibitem{jordan1998variational}
R.~Jordan, D.~Kinderlehrer, and F.~Otto.
\newblock The variational formulation of the fokker--planck equation.
\newblock {\em SIAM journal on mathematical analysis}, 29(1):1--17, 1998.

\bibitem{kantorovich1942translocation}
L.~V. Kantorovich.
\newblock On the translocation of masses.
\newblock In {\em Dokl. Akad. Nauk. USSR (NS)}, volume~37, pages 199--201,
  1942.

\bibitem{kantorovich1957functional}
L.~V. Kantorovich and G.~S. Rubinshtein.
\newblock On a functional space and certain extremum problems.
\newblock {\em Doklady Akademii Nauk}, 115(6):1058--1061, 1957.

\bibitem{kantorovich1958space}
L.~V. Kantorovich and S.~Rubinshtein.
\newblock On a space of totally additive functions.
\newblock {\em Vestnik of the St. Petersburg University: Mathematics},
  13(7):52--59, 1958.

\bibitem{kastoryano2013quantum}
M.~J. Kastoryano and K.~Temme.
\newblock Quantum logarithmic sobolev inequalities and rapid mixing.
\newblock {\em Journal of Mathematical Physics}, 54(5):052202, 2013.

\bibitem{kondratyev2016new}
S.~Kondratyev, L.~Monsaingeon, D.~Vorotnikov, et~al.
\newblock A new optimal transport distance on the space of finite radon
  measures.
\newblock {\em Advances in Differential Equations}, 21(11/12):1117--1164, 2016.

\bibitem{kondratyev2019spherical}
S.~Kondratyev and D.~Vorotnikov.
\newblock Spherical hellinger--kantorovich gradient flows.
\newblock {\em SIAM Journal on Mathematical Analysis}, 51(3):2053--2084, 2019.

\bibitem{kondratyev2020convex}
S.~Kondratyev and D.~Vorotnikov.
\newblock Convex sobolev inequalities related to unbalanced optimal transport.
\newblock {\em Journal of Differential Equations}, 268(7):3705--3724, 2020.

\bibitem{kondratyev2020nonlinear}
S.~Kondratyev and D.~Vorotnikov.
\newblock Nonlinear fokker-planck equations with reaction as gradient flows of
  the free energy.
\newblock {\em Journal of Functional Analysis}, 278(2):108310, 2020.

\bibitem{laschos2019geometric}
V.~Laschos and A.~Mielke.
\newblock Geometric properties of cones with applications on the
  hellinger--kantorovich space, and a new distance on the space of probability
  measures.
\newblock {\em Journal of Functional Analysis}, 276(11):3529--3576, 2019.

\bibitem{le2014diffusion}
D.~Le~Bihan.
\newblock Diffusion mri: what water tells us about the brain.
\newblock {\em EMBO molecular medicine}, 6(5):569--573, 2014.

\bibitem{li2022interpolation}
B.~Li and J.~Lu.
\newblock Interpolation between modified logarithmic sobolev and poincare
  inequalities for quantum markovian dynamics.
\newblock {\em Journal of Statistical Physics}, 2023.

\bibitem{li2020general2}
B.~Li and J.~Zou.
\newblock On the convergence of discrete dynamic unbalanced transport models.
\newblock {\em arXiv preprint arXiv:2310.09420}, 2023.

\bibitem{liero2016optimal}
M.~Liero, A.~Mielke, and G.~Savar{\'e}.
\newblock Optimal transport in competition with reaction: The
  hellinger--kantorovich distance and geodesic curves.
\newblock {\em SIAM Journal on Mathematical Analysis}, 48(4):2869--2911, 2016.

\bibitem{liero2018optimal}
M.~Liero, A.~Mielke, and G.~Savar{\'e}.
\newblock Optimal entropy-transport problems and a new hellinger--kantorovich
  distance between positive measures.
\newblock {\em Inventiones mathematicae}, 211(3):969--1117, 2018.

\bibitem{lombardi2015eulerian}
D.~Lombardi and E.~Maitre.
\newblock Eulerian models and algorithms for unbalanced optimal transport.
\newblock {\em ESAIM: Mathematical Modelling and Numerical
  Analysis-Mod{\'e}lisation Math{\'e}matique et Analyse Num{\'e}rique},
  49(6):1717--1744, 2015.

\bibitem{lombardini2022obstructions}
L.~Lombardini and F.~Rossi.
\newblock Obstructions to extension of wasserstein distances for variable
  masses.
\newblock {\em Proceedings of the American Mathematical Society},
  150(11):4879--4890, 2022.

\bibitem{lott2009ricci}
J.~Lott and C.~Villani.
\newblock Ricci curvature for metric-measure spaces via optimal transport.
\newblock {\em Annals of Mathematics}, pages 903--991, 2009.

\bibitem{maas2015generalized}
J.~Maas, M.~Rumpf, C.~Sch{\"o}nlieb, and S.~Simon.
\newblock A generalized model for optimal transport of images including
  dissipation and density modulation.
\newblock {\em ESAIM: Mathematical Modelling and Numerical Analysis},
  49(6):1745--1769, 2015.

\bibitem{mccann1997convexity}
R.~J. McCann.
\newblock A convexity principle for interacting gases.
\newblock {\em Advances in mathematics}, 128(1):153--179, 1997.

\bibitem{monge1781memoire}
G.~Monge.
\newblock M{\'e}moire sur la th{\'e}orie des d{\'e}blais et des remblais.
\newblock {\em Histoire de l'Acad{\'e}mie Royale des Sciences de Paris}, 1781.

\bibitem{monsaingeon2020schr}
L.~Monsaingeon and D.~Vorotnikov.
\newblock The schr{\"o}dinger problem on the non-commutative fisher-rao space.
\newblock {\em Calculus of Variations and Partial Differential Equations},
  60(1):14, 2021.

\bibitem{olkiewicz1999hypercontractivity}
R.~Olkiewicz and B.~Zegarlinski.
\newblock Hypercontractivity in noncommutative $l_p$ spaces.
\newblock {\em Journal of functional analysis}, 161(1):246--285, 1999.

\bibitem{otto2001geometry}
F.~Otto.
\newblock The geometry of dissipative evolution equations: the porous medium
  equation.
\newblock {\em Communications in Partial Differential Equations}, 26(1-2),
  2001.

\bibitem{otto2000generalization}
F.~Otto and C.~Villani.
\newblock Generalization of an inequality by talagrand and links with the
  logarithmic sobolev inequality.
\newblock {\em Journal of Functional Analysis}, 173(2):361--400, 2000.

\bibitem{peyre2019quantum}
G.~Peyr{\'e}, L.~Chizat, F.-X. Vialard, and J.~Solomon.
\newblock Quantum entropic regularization of matrix-valued optimal transport.
\newblock {\em European Journal of Applied Mathematics}, 30(6):1079--1102,
  2019.

\bibitem{piccoli2014generalized}
B.~Piccoli and F.~Rossi.
\newblock Generalized wasserstein distance and its application to transport
  equations with source.
\newblock {\em Archive for Rational Mechanics and Analysis}, 211(1):335--358,
  2014.

\bibitem{piccoli2016properties}
B.~Piccoli and F.~Rossi.
\newblock On properties of the generalized wasserstein distance.
\newblock {\em Archive for Rational Mechanics and Analysis}, 222(3):1339--1365,
  2016.

\bibitem{powers1970free}
R.~T. Powers and E.~St{\o}rmer.
\newblock Free states of the canonical anticommutation relations.
\newblock {\em Communications in Mathematical Physics}, 16(1):1--33, 1970.

\bibitem{reid1970some}
W.~T. Reid.
\newblock Some elementary properties of proper values and proper vectors of
  matrix functions.
\newblock {\em SIAM Journal on Applied Mathematics}, 18(2):259--266, 1970.

\bibitem{robertson1968decomposition}
J.~B. Robertson, M.~Rosenberg, et~al.
\newblock The decomposition of matrix-valued measures.
\newblock {\em The Michigan Mathematical Journal}, 15(3):353--368, 1968.

\bibitem{rockafellar1971integrals}
R.~Rockafellar.
\newblock Integrals which are convex functionals. ii.
\newblock {\em Pacific Journal of Mathematics}, 39(2):439--469, 1971.

\bibitem{rouze2019concentration}
C.~Rouz{\'e} and N.~Datta.
\newblock Concentration of quantum states from quantum functional and
  transportation cost inequalities.
\newblock {\em Journal of Mathematical Physics}, 60(1):012202, 2019.

\bibitem{rudin2006real}
W.~Rudin.
\newblock {\em Real and complex analysis}.
\newblock Tata McGraw-hill education, 2006.

\bibitem{ryu2018vector}
E.~K. Ryu, Y.~Chen, W.~Li, and S.~Osher.
\newblock Vector and matrix optimal mass transport: theory, algorithm, and
  applications.
\newblock {\em SIAM Journal on Scientific Computing}, 40(5):A3675--A3698, 2018.

\bibitem{santambrogio2015optimal}
F.~Santambrogio.
\newblock Optimal transport for applied mathematicians.
\newblock {\em Birk{\"a}user, NY}, 55(58-63):94, 2015.

\bibitem{sturm2006geometry}
K.-T. Sturm.
\newblock On the geometry of metric measure spaces {I} and {II}.
\newblock {\em Acta mathematica}, 196(1):65--177, 2006.

\bibitem{villani2003topics}
C.~Villani.
\newblock {\em Topics in optimal transportation}.
\newblock Number~58 in Graduate Studies in Mathematics. American Mathematical
  Soc., 2003.

\bibitem{villani2008optimal}
C.~Villani.
\newblock {\em Optimal transport: old and new}, volume 338.
\newblock Springer Science \& Business Media, 2008.

\bibitem{vorotnikov2022partial}
D.~Vorotnikov.
\newblock Partial differential equations with quadratic nonlinearities viewed
  as matrix-valued optimal ballistic transport problems.
\newblock {\em Archive for Rational Mechanics and Analysis}, 243(3):1653--1698,
  2022.

\bibitem{wandell2016clarifying}
B.~A. Wandell.
\newblock Clarifying human white matter.
\newblock {\em Annual review of neuroscience}, 39, 2016.

\bibitem{wirth2021curvature}
M.~Wirth and H.~Zhang.
\newblock Curvature-dimension conditions for symmetric quantum markov
  semigroups.
\newblock {\em Annales Henri Poincar{\'e}}, pages 1--34, 2022.

\end{thebibliography}
\end{document}